\documentclass[a4paper,11pt]{article}
\usepackage{amsmath}
\usepackage{amsfonts}
\usepackage{amssymb}
\usepackage{mathrsfs}
\usepackage{amscd}
\usepackage{pb-diagram}
\usepackage{amsthm}
\usepackage{color}
\usepackage[all]{xy}
\usepackage{graphicx}
\usepackage{url}
\usepackage{enumerate}
\usepackage{tikz-cd}
\numberwithin{equation}{section} 
\usepackage{caption}
\DeclareUnicodeCharacter{2212}{-}
\usepackage{tikz}

\usepackage{titlesec}
\titleformat{\subsection}[runin]{\normalsize\bfseries}{\thesubsection}{5pt}{}

\usepackage{tikz}          
\usetikzlibrary{matrix, arrows, decorations.pathmorphing}

\newcommand{\moment}{
\textup{\textbf{J}}
}

\author{Mathieu Molitor
\\
\small{\it{e-mail:}}\,\,\url{pergame.mathieu@gmail.com}}
\title{K\"{a}hler toric manifolds from dually flat spaces\\
}
\date{}

\begin{document}

\theoremstyle{definition}
\newtheorem{lemma}{Lemma}[section]
\newtheorem{definition}[lemma]{Definition}
\newtheorem{proposition}[lemma]{Proposition}
\newtheorem{corollary}[lemma]{Corollary}
\newtheorem{theorem}[lemma]{Theorem}
\newtheorem{remark}[lemma]{Remark}
\newtheorem{example}[lemma]{Example}
\bibliographystyle{alpha}

\maketitle 


\begin{abstract}
	We present a correspondence between real analytic K\"{a}hler toric manifolds and dually flat spaces, similar to Delzant correspondence in symplectic geometry. 
	This correspondence gives rise to a lifting procedure: if $f:M\to M'$ is an affine isometric map between dually flat spaces and if $N$ and $N'$ are K\"{a}hler toric manifolds 
	associated to $M$ and $M'$, respectively, then there is an equivariant K\"{a}hler immersion $N\to N'$. For example, we show that 
	the Veronese and Segre embeddings are lifts of inclusion maps between appropriate statistical manifolds. 
	We also discuss applications to Quantum Mechanics. 
\end{abstract}


\section{Introduction and summary}

\subsection{From Delzant to Dually flat spaces.}

	In 1988, Delzant showed the following result:

	\begin{theorem}[\cite{Delzant}]\label{nkfnkankdksn}
		Given $i=1,2$, let $(M_{i},\omega_{i})$ be a connected compact symplectic manifold of dimension $2n$ equipped with an effective Hamiltonian 
		torus action $\Phi_{i}:\mathbb{T}^{n}\times M_{i}\to M_{i}$, with momentum map $\moment_{i}:M_{i}\to \mathbb{R}^{n}$. If $\moment_{1}(M_{1})=
		\moment_{2}(M_{2})$, then there is an equivariant symplectomorphism $\varphi:M_{1}\to M_{2}$ satisfying $\moment_{1}=\moment_{2}\circ \varphi$. 
	\end{theorem}

	A quadruplet $(M,\omega,\Phi,\moment)$ as in Theorem \ref{nkfnkankdksn} is called a \textit{symplectic toric manifold}. Two 
	symplectic toric manifolds $(M_{i},\omega_{i},\Phi_{i},\moment_{i})$, $i=1,2,$ are said to be \textit{equivalent} if there is an equivariant symplectomorphism
	$\varphi:M_{1}\to M_{2}$ such that $\moment_{1}=\moment_{2}\circ \varphi$. Call any subset $\Delta$ of $\mathbb{R}^{n}$ a 
	\textit{Delzant polytope} if it is the image under the momentum map $\moment$ of a symplectic toric manifold. 
	Then Theorem \ref{nkfnkankdksn} yields a bijection between the set of Delzant polytopes $\Delta\subset \mathbb{R}^{n}$ and 
	the set of equivalence classes $[(M,\omega,\Phi,\moment)]$ of symplectic toric manifolds:
	\begin{eqnarray}\label{nknwkenfknkn}
	\begin{tabular}{ccc}
		$\left\{
		\begin{tabular}{ll}
		Equivalence classes of \\
		symplectic toric manifolds
		\end{tabular}
		\right\}$
			& $\overset{}{\rightarrow}$  & Delzant polytopes in $\mathbb{R}^{n}$\\[1em]
		$[(M,\omega,\Phi,\moment)]$           & $\mapsto$                                                       & $\moment(M)$ 
	\end{tabular}
	\end{eqnarray}
	This bijection is called \textit{Delzant correspondence}. Delzant actually characterized Delzant polytopes: they are exactly the convex polytopes in $\mathbb{R}^{n}$ 
	that are simple, rational and smooth (see, e.g., \cite{Ana}). 

	From a K\"{a}hler geometrical perspective, 
	the constructions employed by Delzant 
	show that the manifold $M_{\Delta}$ associated to a Delzant polytope $\Delta$ is canonically a K\"{a}hler manifold and that the torus acts by holomorphic and 
	isometric transformations. Moreover, on the set $M_{\Delta}^{\circ}$ of points $p\in M_{\Delta}$ where the torus action is free, the momentum map $\moment:M_{\Delta}^{\circ}\to \Delta^{\circ}$ 
	is a principal $\mathbb{T}^{n}${-}bundle ($\Delta^{\circ}=$ topological interior of $\Delta$) and hence 
	the K\"{a}hler metric on $M_{\Delta}^{\circ}$ descends to a Riemannian metric on $\Delta^{\circ}$. This metric is Hessian;
	a potential $\phi:\Delta^{\circ}\to \mathbb{R}$ was explicitly computed by Guillemin \cite{Victor} (see also \cite{Gauduchon}). 
	Abreu generalized the results of Guillemin by characterizing all potentials $\phi$ on $\Delta^{\circ}$ 
	that are induced from a compatible K\"{a}hler metric $g$ on $M_{\Delta}$ \cite{Abreu}. 

	Delzant correspondence and the results obtained by Guillemin and Abreu 
	are very rigid in the sense that they depend in a crucial way on the geometry of the momentum polytope $\Delta$ (vertices, faces, etc.). 
	The objective of this paper is to ``reshape" Delzant's correspondence in a more flexible way, in the K\"{a}hler setting, 
	by replacing Delzant polytopes by dually flat manifolds (see below). Our motivation comes from the geometrization program of Quantum Mechanics that we pursued in previous 
	works \cite{Molitor2012,Molitor-exponential,Molitor2014,Molitor-2015}\footnote{We learned after publication that many results in \cite{Molitor-2015} were 
	already obtained independently by John David Lafferty in \cite{Lafferty} and Max von Renesse in \cite{Renesse}.}. Let us describe the objects involved. 
	\begin{itemize}
	\item 
	We call \textit{K\"{a}hler toric manifold} a connected K\"{a}hler manifold $N$ 
	of complex dimension $n$ equipped with an effective isometric and holomorphic action $\Phi:\mathbb{T}^{n}\times N\to N$ of the $n${-}dimension real torus 
	$\mathbb{T}^{n}=\mathbb{R}^{n}/\mathbb{Z}^{n}$ that can be globally extended to a holomorphic action $\Phi^{\mathbb{C}}:(\mathbb{C}^{*})^{n}\times N\to N$ 
	of the algebraic torus $(\mathbb{C}^{*})^{n}$. We say that a K\"{a}hler manifold is \textit{regular} if it is complete, connected, simply connected 
	and if the K\"{a}hler metric is real analytic. 
	\item A \textit{dually flat manifold} is a triple $(M,h,\nabla)$, where $(M,h)$ is a Riemannian manifold 
	and $\nabla$ is a flat linear connection whose dual $\nabla^{*}$ with respect to $h$ is also flat. 
	(The dual connection $\nabla^{*}$ is the unique linear connection on $M$ satisfying $Xh(Y,Z)=h(\nabla_{X}Y,Z)+h(Y,\nabla^{*}_{X}Z)$ for all 
	vector fields $X,Y,Z$ on $M$). 
	\end{itemize}
	Two K\"{a}hler toric manifolds $N_{1}$ and $N_{2}$ with torus actions $\Phi_{1}$ and $\Phi_{2}$, respectively, are \textit{equivalent} if there is a K\"{a}hler isomorphism $G:N_{1}\to N_{2}$ 
	and a Lie group isomorphism $\rho:\mathbb{T}^{n}\to \mathbb{T}^{n}$ such that 
	\begin{eqnarray}\label{neknwknkdnwknk}
		G\circ (\Phi_{1})_{a}=(\Phi_{2})_{\rho(a)}\circ G 	
	\end{eqnarray}
	for all $a\in \mathbb{T}^{n}$, where $(\Phi_{i})_{x}:N_{i}\to N_{i},$ $y\mapsto \Phi_{i}(x,y)$. 
	Two dually flat manifolds $(M_{i},h_{i},\nabla_{i})$, $i=1,2$, are \textit{equivalent} if there exists an afine isometric diffeomorphism between them. 

	With this notation, we show that there is a one-to-one correspondence, which we will regard as a duality, of the form (see Theorem \ref{neknkneknkn}) :
	\begin{eqnarray}\label{nknwkenfknknqqdwdw}
	\begin{tabular}{ccc}
		\bigg\{
		\begin{tabular}{ll}
		Equivalence classes of \\
		regular K\"{a}hler toric manifolds\\
		\end{tabular}
		\bigg\}  & $ \overset{\textup{1-to-1}}{\longrightarrow}$  & $\left\{ 
		\begin{tabular}{ll}
		Equivalence classes of \\
		dually flat manifolds\\
			($+$ conditions)
		\end{tabular}
		\right\}
		$	
	\end{tabular}
	\end{eqnarray}

	\noindent The conditions that appear on the right hand side of the duality are discussed below, after we have defined the concept of torification.  

	There are naturally two equivalent ways to define the duality map \eqref{nknwkenfknknqqdwdw}, that we now describe. \\
	
	\noindent\textbf{The symplectic point of view.} Since a regular toric K\"{a}hler manifold $N$ is simply connected (by definition), its torus action is 
	Hamiltonian and hence there is a momentum map $\moment:N\to \mathbb{R}^{n}$. Let $N^{\circ}$ be the set of points $p\in N$ where the torus action is free.  
	Since $\moment:N^{\circ}\to \moment(N^{\circ})$ is a principal $\mathbb{T}^{n}${-}bundle, the K\"{a}hler metric on $N^{\circ}$ descends to a Riemannian metric, say 
	$k$, on $\moment(N^{\circ})$. Let $\nabla^{k}$ be the dual connection of the flat connection on $\moment(N^{\circ})\subset \mathbb{R}^{n}$. Then $[\moment(N^{\circ}),k,\nabla^{k}]$ 
	is the image under the duality map \eqref{nknwkenfknknqqdwdw} of $[N]$ (here $[\,.\,]$ denotes the equivalence class). 

	\noindent\textbf{The complex point of view.} Let $\Phi^{\mathbb{C}}:(\mathbb{C}^{*})^{n}\times N\to N$ be the holomorphic extension of the 
	torus action. Choose any point $p\in N^{\circ}$ and consider the map $\Phi_{p}^{\mathbb{C}}:(\mathbb{C}^{*})^{n}\to N^{\circ}$, $z\mapsto \Phi^{\mathbb{C}}(z,p)$. 
	Then $\Phi_{p}^{\mathbb{C}}$ is a holomorphic diffeomorphism. Let $\sigma:(\mathbb{C}^{*})^{n}\to \mathbb{R}^{n}$, $(z_{1},...,z_{n})\mapsto 
	(\tfrac{\ln(|z_{1}|)}{2\pi},...,\tfrac{\ln(|z_{n}|)}{2\pi})$. The map $\sigma \circ (\Phi_{p}^{\mathbb{C}})^{-1}:N^{\circ}\to \mathbb{R}^{n}$ is a principal $\mathbb{T}^{n}$-bundle.
	Therefore the K\"{a}hler metric on $N^{\circ}$ descends to a Riemannian metric $h$ on $\mathbb{R}^{n}$. Then $[\mathbb{R}^{n},h,\nabla^{\textup{flat}}]$ is the 
	image under the duality map \eqref{nknwkenfknknqqdwdw} of $[N]$, where $\nabla^{\textup{flat}}$ is the flat connection on $\mathbb{R}^{n}$. \\

	The equivalence between the two points of view was essentially proven by Guillemin \cite{Guillemin-Legendre}, who noted that
	the metrics $h$ on $\mathbb{R}^{n}$ and $k$ on $\moment(N^{\circ})$ are related by a Legendre transform, as follows (see Appendix \ref{nfeknknkefnknknk}).
	Both metrics are Hessian, that is, there are smooth functions 
	$\phi:\mathbb{R}^{n}\to \mathbb{R}$ and $\phi^{*}:\moment(N^{\circ})\to \mathbb{R}$ such that $h=\textup{Hess}(\phi)$ and 
	$k=\textup{Hess}(\phi^{*})$. With the right convention, $\phi^{*}$ is the Legendre transform of $\phi$.
	Moreover, (minus) the gradient map $-\nabla\phi:\mathbb{R}^{n}\to 
	\moment(N^{\circ})$ is an isometric diffeomorphism with inverse $-\nabla\phi^{*}$:
	\begin{eqnarray*} 
	\begin{tikzcd}
		&   N^{\circ}  \arrow{rd}{\displaystyle \moment}\arrow[swap]{ld}{\displaystyle \sigma\circ (\Phi_{p}^{\mathbb{C}})^{-1}}  & \\
		\mathbb{R}^{n}   \arrow[swap]{rr}{-\nabla\phi=(-\nabla\phi^{*})^{-1}}  & &   \moment(N^{\circ})
	\end{tikzcd}
	\end{eqnarray*}

	It is not difficult to see that the map $-\nabla\phi:\mathbb{R}^{n}\to \moment(N^{\circ})$ preserves the affine structures of $\mathbb{R}^{n}$ and $\moment(N^{\circ})$ described 
	above and hence the dually flat spaces $(\mathbb{R}^{n},h,\nabla^{\textup{flat}})$ and $(\moment(N^{\circ}),k,\nabla^{k})$ are equivalent (see Proposition \ref{nekfnkwdnkneknknk}). 
	This shows that both points of view can indeed be used to define the duality map \eqref{nknwkenfknknqqdwdw}, and illustrates the flexibility of using dually flat spaces. 

	The symplectic and complex points of view are well-known in the literature \cite{Abreu,Guillemin-Legendre} and they allow to define the duality map \eqref{nknwkenfknknqqdwdw} 
	easily. However, in this paper we focus on the inverse problem: given a dually flat manifold $(M,h,\nabla)$, how can we associate a K\"{a}hler toric manifold? Delzant 
	used symplectic reduction techniques to construct a symplectic toric manifold from a Delzant polytope. 
	Our approach is different (non constructive) and is based on what we call ``torification" of a dually flat space. 

\subsection{Torification.}\label{nekwdnkfenknk}

	Let $(M,h,\nabla)$ be a connected dually flat space. Suppose that there is a global frame $(X_{1},...,X_{n})$ of vector fields on $M$. 
	Then there is a natural action of $\mathbb{R}^{n}$ on the tangent bundle $TM$, defined by $x\cdot u:=u+x_{1}X_{1}+...+x_{n}X_{n}$, where $x\in \mathbb{R}^{n}$ and $u\in TM$. 
	It is a general fact that if a commutative
	group $G$ acts on a set $S$ and if $K$ is any subgroup of $G$, then $G/K$ acts on $S/K$. With $G=\mathbb{R}^{n}$ and $K=\mathbb{Z}^{n}$, we get 
	an effective action of the real torus $\mathbb{T}^{n}=\mathbb{R}^{n}/\mathbb{Z}^{n}$ on the quotient space $TM/\mathbb{Z}^{n}$:
	\begin{eqnarray}\label{nwdnknfeknkndkndk}
		\mathbb{T}^{n}\times TM/\mathbb{Z}^{n}\to TM/\mathbb{Z}^{n}. 
	\end{eqnarray}
	From an analytical point of view, it follows from Dombrowski's construction that the tangent bundle $TM$ is naturally a K\"{a}hler manifold 
	(see Proposition \ref{prop:3.8}). Moreover, if the vector fields $X_{j}$'s are parallel with respect to $\nabla$, then the $\mathbb{R}^{n}$-action on $TM$ is holomorphic and isometric
	(Lemma \ref{nkwnkfenkwn}). Thus the quotient $TM/\mathbb{Z}^{n}$ is naturally a K\"{a}hler manifold and the torus action \eqref{nwdnknfeknkndkndk} is isometric 
	and holomorphic (Lemma \ref{nndknfkdnfkdnk}).

	Let $N$ be a connected K\"{a}hler manifold equipped with an effective 
	holomorphic and isometric torus action $\mathbb{T}^{n}\times N\to N$. We say that 
	$N$ is a \textit{torification} of the dually flat space $(M,h,\nabla)$ if there is a global frame $(X_{1},...,X_{n})$ of $\nabla$-parallel vector fields on $M$ 
	and an equivariant K\"{a}hler isomorphism $F:TM/\mathbb{Z}^{n}\to N^{\circ}$ (see Definition \ref{nksdnkfnksdn}). 

	Comments are in order. To say that $N$ is a torification of $M$ simply means that $N^{\circ}$ can be parametrized by points in $TM/\mathbb{Z}^{n}$ in a 
	consistant way with the geometric structures involved, but the definition does not say anything about the set $N\backslash N^{\circ}$ of points where the torus action is not free. As 
	a result, if a torification exists, then it is not unique in general (if $N$ is a torification, then so does $N^{\circ}$). In that sense, torification is a weak concept. 
	However, the situation changes if the following conditions are imposed on the K\"{a}hler manifold $N$: completeness, simply connectedness and real analyticity of the K\"{a}hler 
	metric. When these conditions are satisfied, we say that $N$ is \textit{regular}. If $M$ has a regular torification, we say that it is \textit{toric}. 

	One of the main results of this paper is the following: all regular torifications of a given dually flat manifold are equivalent 
	(Theorem \ref{aaanfeknkenfknk}). Therefore one can unambiguously associate a K\"{a}hler manifold (equipped with a torus action) to 
	a toric dually flat manifold, up to an equivariant K\"{a}hler isomorphism. This yields a map:
	\begin{eqnarray}\label{nckenkefnknknkef}
	\begin{tabular}{ccc}
		\big\{
		\begin{tabular}{ll}
		Toric dually flat manifolds\\
		\end{tabular}
		\big\}
			& $\overset{\textup{torification}}{\longrightarrow}$  & \bigg\{
		\begin{tabular}{ll}
		Equivalence classes of \\
		regular torifications\\
		\end{tabular}
		\bigg\}  \\[1em]
		$(M,h,\nabla)$           & $\overset{\textup{}}{\longmapsto}$   & $\big[\,\mathbb{T}^{n}\times N\to N\,\big].$
	\end{tabular}
	\end{eqnarray}
	The torification map \eqref{nckenkefnknknkef} is analogous to the map 
	$\Delta \mapsto [M_{\Delta}]$ given by Delzant correspondence, but there are two notable differences: $N$ is not necessarily compact 
	(see examples below) and the equivalence relation between torifications 
	is weaker than the equivalence relation used in Delzant correspondence (the torus action can be reparametrized by a Lie group isomorphism 
	$\rho:\mathbb{T}^{n}\to \mathbb{T}^{n}$, see \eqref{neknwknkdnwknk}). 

	Regarding the duality \eqref{nknwkenfknknqqdwdw}, we show that the torification map \eqref{nckenkefnknknkef} depends only on the 
	equivalence class of $(M,h,\nabla)$ and hence it descends to a quotient map 
	between quotient spaces. When certain conditions are imposed on $(M,h,\nabla),$ this quotient map becomes the inverse of the duality map 
	\eqref{nknwkenfknknqqdwdw}. These conditions are (see Theorem \ref{neknkneknkn}): (1) $(M,h,\nabla)$ is toric, 
	(2) the torus action associated to the regular torification of $(M,h,\nabla)$ is holomorphically extendable, and (3) $(M,h,\nabla)$ 
	admits a global pair of dual coordinate systems (see Definition \ref{nfkdwnkenfkdk}). These three conditions are the conditions 
	that appear on the right hand side of the duality \eqref{nknwkenfknknqqdwdw}.

\subsection{Lifting property.} One of the main results of this paper is that the torification map \eqref{nckenkefnknknkef} gives rise to a lifting procedure. 
	Suppose $N$ and $N'$ are regular torifications of the dually flat manifolds $(M,h,\nabla)$ and $(M',h',\nabla')$, respectively. By definition of a torification,
	there are equivariant K\"{a}hler isomorphisms $N^{\circ}\cong TM/\mathbb{Z}^{n}$ and $(N')^{\circ}\cong TM'/\mathbb{Z}^{d}$. In this situation, we show that if 
	$f:M\to M'$ is an affine isometric map, then its derivative $f_{*}:TM\to TM'$ is automatically equivariant with respect to the actions 
	of $\mathbb{Z}^{n}$ and $\mathbb{Z}^{d}$ and hence it descends to a map $m:N^{\circ}\to (N')^{\circ}$ (Proposition \ref{nfekwnknfkenk}). 
	This map extends uniquely to an equivariant K\"{a}hler immersion $m:N\to N'$. We call it a \textit{lift} of $f$. 
	Lifts are not unique (they depend on the choice of the parametrizations $N^{\circ}\cong TM/\mathbb{Z}^{n}$), but they are conjugate (Proposition \ref{neknmfkrnknk}). 
	Under suitable parametrizations, compositions of lifts are lifts (Proposition \ref{cnekdwnkndknk}). 	



%
%
%
%
%
%

\subsection{Fundamental lattice.} If $N$ is a torification of $(M,h,\nabla)$, then there is a frame $X=(X_{1},...,X_{n})$ of $\nabla$-parallel vector fields $X_{j}$'s on $M$
	that induces a $\mathbb{Z}^{n}$-action on $TM$ and there is an equivariant K\"{a}hler isomorphism 
	$TM/\mathbb{Z}^{n}\to N^{\circ}$. The pair $(X,F)$ is not unique in general. 
	However, we show that if $N$ is regular, then the set 
	\begin{eqnarray*}
		\mathscr{L}=\{k_{1}X_{1}(p)+...+k_{n}X_{n}(p)\,\,\big\vert\,\,p\in M,\,\,k_{1},...,k_{n}\in \mathbb{Z}\}\subset TM
	\end{eqnarray*}
	is independent of the the frame $(X_{1},...,X_{n})$. We call $\mathscr{L}\subset TM$ the \textit{fundamental lattice} of $(M,h,\nabla)$. 

	The fundamental lattice is closely related to the space of K\"{a}hler functions on $TM$. Recall that a K\"{a}hler function on a K\"{a}hler manifold 
	is a smooth function whose Hamiltonian vector field is Killing. We show that when the space of K\"{a}hler functions separates the points of $N$, 
	then $\mathscr{L}$ coincides with the set of diffeomorphisms $\varphi:TM\to TM$ satisfying $f\circ \varphi=f$ for all K\"{a}hler functions 
	$f:TM\to \mathbb{R}$, where $\mathscr{L}$ is identified with a group of translations (see Proposition \ref{nfeknkwneknk}). Using this, we 
	deduce a criteria for a dually flat manifold $(M,h,\nabla)$ not to be toric (Corollary \ref{nekdnwkwndknk}).  

\subsection{Examples from Information Geometry.}\label{nfeeknkwefknkn}
	Information geometry is the branch of mathematics that studies statistical manifolds, that is, manifolds 
	whose points are probability density functions defined over a fixed measure space \cite{Jost2,Amari-Nagaoka,Murray}. For example, the set of all normal distributions 
	$(\sqrt{2\pi}\sigma)^{-1}\textup{exp}(-(x-\mu)^{2}/(2\sigma^{2}))$ is a statistical manifold 
	parametrized by the mean $\mu\in\mathbb{R}$ and $\sigma>0$. Another important example is the set of all 
	probabilities $p:\Omega\to \mathbb{R}$ $p>0$, $\sum_{k=1}^{n}p(x_{k})=1$, defined over a finite set $\Omega=\{x_{1},...,x_{n}\}$. These probabilities are called 
	\textit{categorical distributions}. Together they form an $(n-1)$-dimensional statistical manifold. 

	In general, a statistical manifold $M$ is equipped with a Riemannian metric $h_{F}$, 
	namely the \textit{Fisher metric}, and a linear connection $\nabla^{(e)}$, called \textit{exponential connection}. In many cases, $(M,h_{F},\nabla^{(e)})$ is a dually flat manifold. 
	For example, this is true for the most common families of statistical distributions: Binomial, Poisson, Normal, Categorical, just to name a few.
	Throughout the rest of the introduction, we will use these names 
	to denote the corresponding families of probability distributions. Some of these families are characterized by some parameters, like the number of trials 
	for Binomial. For simplicity, we will omit 
	these parameters from our notation. Thus, for instance, we will denote by $\textsf{Binomial}$ 
	the set of all Binomial distributions $p(k)=\binom{n}{k}q^{k}(1-q)^{n-k}$, $q\in (0,1)$, defined over the set $\{0,1,...,n\}$, instead of using an explicit notation like 
	$\textsf{Binomial}(n)$.

	Let $\mathbb{P}_{n}(c)$ denote the complex projective space of complex dimension $n$ and holomorphic sectional curvature $c>0$. Let $\mathbb{D}(c)$ denote the 
	unit disk in $\mathbb{C}$ endowed with the Hyperbolic metric of constant holomorphic sectional curvature $c<0$. Below we give a list of statistical 
	manifolds that are toric together with their regular torifications (see Proposition \ref{nfeknwknwknk}).
	\begin{center}
		\begin{tabular}{l|l}
			K\"{a}hler toric manifold                  &    Dually flat space \\
			\hline \\[-0.5em]
			$S^{2}=\mathbb{P}_{1}(\tfrac{1}{n})$       &    \textsf{Binomial} \\[0.3em]
			$\mathbb{C}$                               &    \textsf{Poisson}  \\ [0.3em]
			$\mathbb{P}_{n}(1)$                        &   \textsf{Categorical}\\[0.3em]
			$\mathbb{P}_{n}(\tfrac{1}{m})$             &   \textsf{Multinomial}\\[0.3em]
			$\mathbb{D}(-\tfrac{1}{r})$                &   \textsf{Negative Binomial} 
		\end{tabular}
	\end{center}
	The numbers $n,m,r$ are all positive integers. For the corresponding torus actions, see Proposition \ref{nfeknwknwknk}. A simple counterexample is the set of 
	normal distributions with known variance $\sigma=1$, which is not toric (see Example \ref{nceknwdkenknk}). 
	Note that $\mathbb{C}$ and $\mathbb{D}$ are non-compact. 

	Given two dually flat manifolds $M_{1}$ and $M_{2}$ with torifications $N_{1}$ and $N_{2}$, respectively, we show that $N_{1}\times N_{2}$ is a torification of 
	$M_{1}\times M_{2}$ (Proposition \ref{nfeknknefknfknk}). 
	Therefore one can construct new torifications from old ones. For example, $\mathbb{C}^{n}$ is the torification of a finite product of Poisson families. 

	An interesting connection with algebraic geometry and projective varieties arises when the measure space $\Omega$ of a dually flat statistical manifold 
	$M$ is finite. In this case, we prove that the natural inclusion $M\subset \textsf{Categorical}$ is always affine and isometric (Lemma \ref{nefknknknk}). 
	Thus if $M$ has a regular torification $N$, then there is a lift $f:N\to \mathbb{P}_{r}(1)$ which is an equivariant K\"{a}hler immersion. 
	For example (see Propositions \ref{ncednwkneksnk} and \ref{nekwnknkwdkejjkknk}), 
	\begin{itemize}
	\item the Veronese embedding $\mathbb{P}_{1}(\tfrac{1}{n})\to \mathbb{P}_{n}(1)$, $[z_{1},z_{2}]\mapsto\big[
			z_{1}^{n},...,\binom{n}{k}^{1/2}z_{1}^{n-k}z_{2}^{k},...,z_{2}^{n}\big]$ (homogeneous coordinates) and 
	\item the Segre embedding $\mathbb{P}_{n}(1)\times \mathbb{P}_{m}(1)\to \mathbb{P}_{(n+1)(m+1)-1}(1)$, $([z_{i}],[w_{j}])\mapsto [z_{i}w_{j}]$, 
		$i=1,...,n,j=1,...,m$ (lexicographic ordering),
	\end{itemize}
	are the lifts of the inclusion maps $\textsf{Binomial}\subset \textsf{Categorical}$ and $\textsf{Categorical}\times \textsf{Categorical}\subset \textsf{Categorical}$, 
	respectively. Both maps are well-known in algebraic geometry (see, e.g., \cite{Harris}). 

\subsection{Applications to Geometric Quantum Mechanics.}

	Geometric Quantum Mechanics (GQM) is a geometric reformulation of Quantum Mechanics (QM) based on the K\"{a}hler structure of the complex projective space $\mathbb{P}_{n}$
	\cite{Ashtekar}. In GQM, the Schr\"{o}dinger evolution is replaced by the Hamiltonian flows of K\"{a}hler functions 
	$f:\mathbb{P}_{n}\to \mathbb{R}$, and the probabilistic features (quantum uncertainties and state vector reduction in a measurement process) 
	are described by the Riemannian metric on $\mathbb{P}_{n}$, namely the Fubini-Study metric. This is an elegant formalism that 
	enlightens many aspects of QM, such as Berry’s phase, entanglement and the measurement problem \cite{Spera-geometric}. 


	GQM naturally leads to K\"{a}hler manifolds other than just the complex projective space $\mathbb{P}_{n}$. Let us give two examples.
	\begin{enumerate}[(1)]
		\item The \textit{Segre variety} characterizes the set of completely disentangled states (see \cite{Benvegnu}).
		\item The \textit{Veronese variety} characterizes the so-called \textit{spin coherent states} of 
			$\textup{SU}(n)$ (see \cite{Eva,Molitor-exponential}).
	\end{enumerate}
	We stress the fact that these examples, and their connections with QM, are not the result of a quantization scheme (like geometric quantization \cite{Blau,Nair,Woodhouse}); 
	they arise in a purely quantum context, where no classical counterpart is required\footnote{GQM is not concerned with quantization.}.

	It is natural to ask whether or not GQM could be extended to a ``natural" class of K\"{a}hler manifolds 
	that would include the examples above (projective space, Segre and Veronese varieties), and if a simple mechanism could explain their connections with QM. 
	It is our opinion that the torification approach provides an appropriate mathematical framework to tackle these questions. Let us elaborate this further 
	by examining the examples above in more details. 

	The projective space $\mathbb{P}_{n}$, the Veronese and Segre varieties are all torifications of statistical manifolds (see examples above). 
	Because of this, each one of them is canonically associated to a probability space, and thus a probabilistic interpretation is available. This is important 
	because QM is probabilistic in nature. More concretely, if $N$ is the torification of a dually flat manifold $M$, then there is a canonical projection 
	$\kappa:N^{\circ}\to M$ (it is unique up to an equivariant K\"{a}hler isomorphism of $N$, see Lemma \ref{efnkfenknfknk}). If $M$ is a statistical manifold, then $\kappa(p)$ 
	is a probability for every $p\in N^{\circ}$. For example, 
	let $M=\textsf{Categorical}$ be the set of positive probability density functions defined on a finite set 
	$\Omega=\{x_{0},...,x_{n}\}$. 
	Then $N=\mathbb{P}_{n}$ is the regular torification of $M$ and the corresponding projection $\kappa:\mathbb{P}_{n}^{\circ}\to M$ is defined by 
	$\kappa([z])(x_{k})=\tfrac{|z_{k}|^{2}}{|z_{0}|^{2}+...+|z_{n}|^{2}}$, where $[z]=[z_{0},...,z_{n}]\in \mathbb{P}_{n}^{\circ}$ (homogeneous coordinates) and 
	$x_{k}\in \Omega.$ (For more examples, see Proposition \ref{njewndknknknk}.) In \cite{Molitor-exponential}, we explain how to recover the probabilistic interpretation of GQM from 
	the projection $\kappa$ instead of the geodesic distance on $\mathbb{P}_{n}$. We then applied this approach to the torification of \textsf{Binomial}, which is the sphere $S^{2}$. 
	The result is essentially the mathematical description of the spin 
	of a non-relativistic particle. This example shows that GQM can be generalized by considering torifications of appropriate statistical manifolds. 


	Regarding the Veronese variety, there is a simple way to understand its connection to QM, that we now describe. By definition, the Veronese variety 
	is the image of the sphere $S^{2}$ under the Veronese embedding $S^{2}\to \mathbb{P}_{n}$ and hence it is 
	a complex submanifold of $\mathbb{P}_{n}$. Thus, one may consider the following problem: given a K\"{a}hler function $f$ on $S^{2}$, is there 
	a K\"{a}hler function $\hat{f}$ on $\mathbb{P}_{n}$ that extends $f$? In \cite{Molitor-exponential}, we proved that every 
	K\"{a}hler function $f$ on $S^{2}$ extends uniquely to a K\"{a}hler function $\hat{f}$ on $\mathbb{P}_{n}$ and that the map 
	$\mathscr{K}(S^{2})\to \mathscr{K}(\mathbb{P}_{n})$, $f\mapsto \hat{f}$ is 
	a Lie algebra homomorphism (here $\mathscr{K}(N)$ denotes the space of K\"{a}hler functions on $N$). Since $\mathscr{K}(S^{2})\cong \mathfrak{u}(2)$ and $\mathscr{K}(\mathbb{P}_{n})
	\cong \mathfrak{u}(n+1)$, one obtains a unitary representation of $\mathfrak{u}(2)$. From this one can extract all irreducible unitary representations of $\mathfrak{su}(2)$, 
	which is what physicists use to describe the spin of a particle. The Veronese embedding itselft is interesting from 
	the physical point of view, because it characterizes the so-called \textit{spin coherent states} \cite{Brody,Molitor-exponential}. 
	In summary, the spin representations and spin coherent states are manifestations of 
	the exterior geometry of $S^{2}$, that is, the way $S^{2}$ is embedded into $\mathbb{P}_{n}$. From a torification point of view, this 
	embedding is just the lift of the inclusion map $\textsf{Binomial}\subset \textsf{Categorical}$. 
	
	Now we turn our attention to the Segre embedding. In QM, the Hilbert space for a combined system of two particles is the tensor product $H_{1}\otimes H_{2}$, where 
	each Hilbert space $H_{i}$ is the state space of a single particle. If the state $\Psi\in H_{1}\otimes H_{2}$ of the particles can be written as a single 
	term $u\otimes v$ for some $u\in H_{1}$ and $v\in H_{2}$, then we say that the particles are \textit{disentangled}. For simplicity, suppose that 
	$H_{1}=H_{2}=\mathbb{C}^{2}$. In this case, it can be proven that the two particles (``qubits") are disentangled if and only if their state vector 
	$\Psi\in \mathbb{P}_{3}=\mathbb{P}(\mathbb{C}^{2}\otimes \mathbb{C}^{2})$ is a point in the Segre variety $X\subset \mathbb{P}_{3}$. More generally, 
	the generalized Segre Variety $X\subset \mathbb{P}_{2^{n}-1}$ coincides with the set of disentangled states of the $n$-qubit space \cite{Benvegnu}.
	Mathematically, the Segre variety $X\subset \mathbb{P}_{3}$ is the image of $\mathbb{P}_{1}\times \mathbb{P}_{1}$ under the Segre embedding $f([z_{1},z_{2}],[w_{1},w_{2}])=
	[z_{1}w_{1},z_{1}w_{2},z_{2}w_{1},z_{2}w_{2}]$. From a torification point of view, the Segre embedding is just the lift of the inclusion map
	$\textsf{Categorical}\times \textsf{Categorical}\subset \textsf{Categorical}$. 

	It is interesting to note that a multiparticle formalism is readily implemented in 
	the torification approach. Suppose $M_{1}$ and $M_{2}$ are dually flat statistical manifolds defined over the finite sets 
	$\Omega_{1}=\{x_{0},...,x_{r}\}$ and $\Omega_{2}=\{y_{0},...,y_{s}\}$, respectively. Then the product $M_{1}\times M_{2}$ is naturally a dually flat 
	statistical manifold defined over $\Omega_{1}\times \Omega_{2}$ (see Definition \ref{nfknknefknknekn}). 
	If $N_{1}$ and $N_{2}$ are regular torifications of $M_{1}$ and $M_{2}$, respectively, then the product $N_{1}\times N_{2}$ is the regular torification of 
	$M_{1}\times M_{2}$ (Proposition \ref{nfeknknefknfknk}). Moreover, since the natural inclusion $M_{1}\times M_{2}\subset \textsf{Categorical}$ is affine and isometric, there is a lift 
	$N_{1}\times N_{2}\to \mathbb{P}_{(r+1)(s+1)-1}=\mathbb{P}(\mathbb{C}^{r+1}\otimes \mathbb{C}^{s+1})$ which generalizes the Segre embedding. \\

	The discussion above highlights some of the advantages of the torification approach, like its simplicity, minimality and naturalness, and paves the way for 
	further developments of GQM. Perhaps more importantly, the torification approach places information geometry at the very heart of QM in finite dimension.
	This naturally leads to the idea that QM might be derived from 
	information-theoretical principles, like, for instance, the constancy of speed of light in special relativity.  
	Many authors already reached similar conclusions using axiomatic approaches 
	\cite{Brukner,Clifton,Grimbaum,Chiribella,Goyal-2008,Goyal,Goyal-2010,Masanes,Rovelli}. These approaches have their own merits and 
	respective successes, but to our knowledge, no consensus has emerged yet. The torification approach, in this regard, 
	might offer some guidance or insight.

\subsection{Organization of the paper.} 
	For the convenience of the reader, the paper starts with a discussion on the relation between K\"{a}hler geometry and statistics (see Section \ref{nfekkdwnknk}). 
	The material presented is mostly taken from \cite{Molitor2014}, except for Proposition \ref{nfeknwknekn} which seems to be new. 
	We shall present the subject in a uniform way by using the concept of dually flat structure.
	In Section \ref{nkdnksnkenkn}, we introduce the concept of \textit{parallel lattice} on a dually flat manifold and show how to construct a torus action from it. 
	The analytical and geometrical properties of this torus action are discussed in Sections \ref{nmknfkenkfnk} and \ref{ndnkneknkdnksk}. The concept of torification is defined in Section 
	\ref{nfeknkwdnkk}. In Section \ref{nekdwknknwkdnknkn}, we re-examine carefully some already known results in toric geometry, adopting 
	systematically the torification point of view. In particular we show that every K\"{a}hler toric manifold is a torification of an appropriate 
	dualistic structure on $\mathbb{R}^{n}$ (complex point of view, Theorem \ref{nckdnwknknednenkenk}) but also the torification 
	of its momentum polytope (symplectic point of view, Corollary \ref{nknkwkkdnkn}). Examples from information geometry are presented in Section \ref{nfknwkdenknwkn}. 
	The lifting property that comes with the concept of torification is discussed in Section \ref{nfwkwnkenkwndk}. The fundamental lattice of a toric dually flat manifold 
	is defined in Section \ref{nfknnkndknknknk}. In Section \ref{nwkwnkenfkwnknknk}, we describe a close connection between algebraic geometry and the torification of statistical 
	manifolds. Finally, in Section \ref{nnknkenkwnkdnk} we discuss the duality between dually flat spaces and K\"{a}hler toric manifolds mentioned at the beginning of this paper.

       \textbf{}\\
       \textbf{Notation.} Let $M$ be a manifold. The map $\pi:TM\to M$ denotes the canonical projection. The set $\textup{Diff}(M)$ is the group of 
       diffeomorphisms of $M$. The space of vector fields on $M$ is denoted by $\mathfrak{X}(M)$. 
       Let $\Phi:G\times M\to M$ be a Lie group action of a Lie group $G$ on $M$. Given $g\in G$ and $p\in M$, 
       we will denote by $\Phi_{p}:G\to M$ and $\Phi_{g}:M\to M$ the maps defined by $\Phi_{p}(g)=\Phi_{g}(p)=\Phi(g,p)$. If $N$ is a K\"{a}hler manifold, we usually denote 
       its metric by $g$, its complex structure by $J$ and the corresponding symplectic form by $\omega$. The set $\textup{GL}(n,\mathbb{K})$ is the group of 
       $n\times n$ invertible matrices with entries in $\mathbb{K}\in \{\mathbb{R},\mathbb{Z}\}$.

\section{Preliminaries: dually flat spaces and K\"ahler geometry}\label{nfekkdwnknk}

\subsection{Connections and connectors.}\label{pouette}

	This section follows closely \cite{Dombrowski}. Let $M$ be a manifold. 
\begin{definition}\label{def:2.1}
	A \textit{linear connection} $\nabla$ on $M$ is a map 
	$\mathfrak{X}(M)\times\mathfrak{X}(M)\to\mathfrak{X}(M)$, $(X,Y)\mapsto\nabla_{X}Y$, 
	satisfying the following properties:
	\begin{enumerate}[(1)]
		\item $\nabla_{fX+gY}Z=f\nabla_{X}Z+g\nabla_{Y}Z$,
		\item $\nabla_{X}(Y+Z)=\nabla_{X}Y+\nabla_{X}Z$,
		\item $\nabla_{X}(fY)=X(f)Y+f\nabla_{X}Y$,
	\end{enumerate}
	for all vector fields $X,Y,Z\in\mathfrak{X}(M)$ and for all functions 
	$f,g\in C^{\infty}(M)$. 
\end{definition}

	In local coordinates $(x_{1},...,x_{n})$ on $U\subseteq M$, if $X=\sum_{k=1}^{n}X^{i}\tfrac{\partial}{\partial x_{k}}$ 
	and $Y=\sum_{k=1}^{n}Y^{k}\tfrac{\partial}{\partial x_{k}}$, then, by standard computations, 
	\begin{eqnarray}\label{nekwmdkvnksnk}
		\nabla_{X}Y=\sum_{k=1}^{n}\bigg(X(Y^{k})+\sum_{i,j=1}^{n}X^{i}Y^{j}\Gamma_{ij}^{k}\bigg)
		\dfrac{\partial}{\partial x_{k}},
	\end{eqnarray}
	where $\Gamma_{ij}^{k}:U\to \mathbb{R}$ are the Christoffel symbols, defined by the formula 
	\begin{eqnarray*}
		\nabla_{\frac{\partial}{\partial x_{i} }}\frac{\partial}{\partial x_{j}}
		=\sum_{k=1}^n\Gamma_{ij}^{k}\frac{\partial}{\partial x_{k}}\,\,\,\,\,\,\qquad \textup{for}\,\,i,j=1,...,n.
	\end{eqnarray*}
		
	Let $\pi:TM\to M$ be the canonical projection and let 
	$(U,\varphi)$ be a chart for $M$ with local coordinates $(x_{1},...,x_{n})$. Define 
	$\widetilde{\varphi}:\pi^{-1}(U)\to \mathbb{R}^{2n}$ by
	\begin{eqnarray*}
		\widetilde{\varphi}\bigg(\sum_{i=1}^{n}u_{i}\dfrac{\partial}{\partial x_{i}}\bigg\vert_{p}\bigg)
		=(x_{1}(p),...,x_{n}(p),u_{1},...,u_{n}). 
	\end{eqnarray*}
	Then $(\pi^{-1}(U),\widetilde{\varphi})$ is a chart for $TM$. We will usually denote the corresponding local coordinates on 
	$\pi^{-1}(U)\subseteq TM$ by $(q,r)=(q_{1},...,q_{n},r_{1},...,r_{n})$ (in particular, $q_{i}=x_{i}\circ \pi$ for every $i=1,...,n)$. 
	Informally, we sometimes write $(x,\dot{x})=(x_{1},...,x_{n},\dot{x}_{1},...,\dot{x}_{n})$ instead of $(q,r)=(q_{1},...,q_{n},r_{1},...,r_{n})$. 


	Let $u=\sum_{k=1}^{n}u_{k}\tfrac{\partial}{\partial x_{k}}\big\vert_{p}\in \pi^{-1}(U)$ be arbitrary. 
	Define a linear map $K_{u}:T_{u}(TM)\to T_{p}M$ by 
	\begin{eqnarray*}
		&&K_{u}\bigg(\dfrac{\partial}{\partial q_{a}}\bigg\vert_{u}\bigg):=
			\sum_{k,j=1}^{n}\Gamma_{aj}^{k}(p)u_{j}\dfrac{\partial}{\partial x_{k}}\bigg\vert_{p}
			 \,\,\,\,\,\,\,\,\,\,\textup{for}\,\,a=1,...,n,\\[0.9em]
		&&K_{u}\bigg(\dfrac{\partial}{\partial r_{a}}\bigg\vert_{u}\bigg):=
			\dfrac{\partial}{\partial x_{a}}\bigg\vert_{p}\,\,\,\,\,\,\,\,\,\,\textup{for}\,\,a=1,...,n.
	\end{eqnarray*}
\begin{lemma}\label{nfkewefnknknk}
	Let $X$ and $Y$ be vector fields on $M$. Suppose $Y(p)=u$. Then $K_{u}(Y_{*_{p}}X_{p})=(\nabla_{X}Y)(p)$.
\end{lemma}
\begin{proof}
	By standard computations, 
	\begin{eqnarray*}
		Y_{*_{p}}X_{p}=\sum_{a=1}^{n}\bigg(X^{a}(p)\dfrac{\partial}{\partial q_{a}}\bigg\vert_{u}
		+X_{p}(Y^{a})\dfrac{\partial}{\partial r_{a}}\bigg\vert_{u}\bigg)
	\end{eqnarray*}
	so 
	\begin{eqnarray*}
		K_{u}(Y_{*_{p}}X_{p})=
		\sum_{a=1}^{n}\bigg(X^{a}(p)\sum_{k,j=1}^{n}\Gamma_{aj}^{k}(p)Y^{j}(p)\dfrac{\partial}{\partial x_{k}}\bigg\vert_{p}
			+X_{p}(Y^{a}) \dfrac{\partial}{\partial x_{a}}\bigg\vert_{p}\bigg),
	\end{eqnarray*}
	where we have used $u_{j}=Y^{j}(p)$. Comparing the above formula with the local 
	expression for $\nabla_{X}Y$ in coordinates (see \eqref{nekwmdkvnksnk}), one obtains the desired formula. 
\end{proof}

	Clearly, vectors of the form $Y_{*_{p}}X_{p}$, with $Y_{p}=u$, generate $T_{u}(TM)$, and so the above lemma implies 
	that the definition of $K_{u}$ is independant of the choice of the chart $(U,\varphi)$. 
	
	The map 
	\begin{eqnarray*}
		K:TTM\to TM,
	\end{eqnarray*}
	defined for $A\in T_{u}(TM)$ by $K(A):=K_{u}(A)$, is called \textit{connector}, or \textit{connection map}, 
	associated to $\nabla$. 

	The following result is an immediate consequence of the definition of $K$.
\begin{proposition}\label{prop:2.3}
	Let $K$ be the connector associated to a connection $\nabla$ on $M$. The following holds.
	\begin{enumerate}[(1)]
	\item For every pair $X,Y$ of vector fields on $M$, $\nabla_{X}Y=KY_{*}X$,
		where $Y_{*}X$ denotes the derivative of $Y$ in the direction of $X$.
	
	\item For every $u\in T_{p}M$, the restriction of $K$ to $T_{u}(TM)$ is a linear map $T_{u}(TM)\to T_{p}M$.
	\end{enumerate}
\end{proposition}

	If $A\in T_{u}(TM)$ is such that $\pi_{*_{u}}A=0$ and $K(A)=0$, then a simple calculation using local coordinates 
	shows that $A=0$. Therefore, 
\begin{proposition}\label{cor:2.7}
	Let $K$ be the connector associated to a connection $\nabla$ on $M$. 
	Given $u\in T_{p}M$, the map $T_{u}(TM)\to T_{p}M\oplus T_{p}M$,
	defined by
	\begin{eqnarray}
		A\mapsto(\pi_{*_{u}}A,KA),\label{eq:2.5}
	\end{eqnarray}
	is a linear bijection.
\end{proposition}
	Thus, given a linear connection $\nabla$, we can identify at any point $u\in T_{p}M$ the vector spaces  $T_{u}(TM)$
	and $T_{p}M\oplus T_{p}M$ via the map (\ref{eq:2.5}).

\subsection{Dually flat spaces.}

\begin{definition}
	Let $(M,h)$ be a Riemannian manifold endowed with a linear connection $\nabla$. The \textit{dual connection} of $\nabla$ with 
	respect to $h$ is the unique linear connection, denoted by $\nabla^{*}$, satisfying 
	\begin{eqnarray*}
		Xh(Y,Z)=h(\nabla_{X}Y,Z)+h(Y,\nabla_{X}^{*}Z)
	\end{eqnarray*}
	for all vector fields $X,Y,Z$ in $M$. The triple $(h,\nabla,\nabla^{*})$ is called a \textit{dualistic structure} on $M$. 
\end{definition}
\begin{example}
	If $\nabla$ is the Levi-Civita connection of $h$, then $\nabla^{*}=\nabla$. 
\end{example}

	As the literature is not uniform, let us agree that the torsion $T$ and the curvature tensor $R$ of a
	connection $\nabla$ are defined as
	\begin{alignat*}{1}
		 T(X,Y)&:=\nabla_{X}Y-\nabla_{Y}X-[X,Y],\nonumber\\
		 R(X,Y)Z &:=\nabla_{X}\nabla_{Y}Z-\nabla_{Y}\nabla_{X}Z-\nabla_{[X,Y]}Z,\nonumber
	\end{alignat*}
	where $X,Y,Z$ are vector fields on $M$. By definition, a linear connection is \textit{flat} 
	if the torsion and curvature tensor are identically zero on $M$. A manifold endowed with a 
	flat linear connection is called an \textit{affine manifold}. 

\begin{definition}\label{def:3.7} 
	A dualistic structure $(h,\nabla,\nabla^{*})$ is \textit{dually flat} if both $\nabla$ and $\nabla^{*}$ are flat. A manifold 
	endowed with a dually flat structure is called a \textit{dually flat manifold}, or \textit{dually flat space}. 
\end{definition}

	Since the dual connection $\nabla^{*}$ is completely determined by $h$ and $\nabla$, we will often regard 
	a dually flat manifold as a triple $(M,h,\nabla)$. 

\begin{proposition}\label{neewknkfenknkn}
	Let $(h,\nabla,\nabla^{*})$ be a dualistic structure on a manfold $M$. Let $R$ and $R^{*}$ denote the curvature tensores of $\nabla$ and $\nabla^{*}$, respectively.
	Then 
	\begin{eqnarray*}
		h(R(X,Y)Z,W)=-h(R^{*}(X,Y)W,Z)
	\end{eqnarray*}
	for all vector fields $X,Y,Z,W$ on $M$. In particular, $R\equiv 0$ if and only if $R^{*}\equiv 0$.
	\end{proposition}
\begin{proof}
	See \cite{Amari-Nagaoka}, Theorem 3.3.
\end{proof}

	Therefore an affine manifold $(M,\nabla)$ endowed with a Riemannian metric $h$
	is a dually flat manifold if and only if the torsion of the dual connection $\nabla^{*}$ is identically zero. 

	Let $\nabla^{\textup{flat}}$ denote the canonical 
	flat connection on $\mathbb{R}^{n}$ (or any open subset of it). Thus, if $x=(x_{1},...,x_{n})$ are standard coordinates on $\mathbb{R}^{n}$, then 
	\begin{eqnarray}\label{nkwndkddnkwnneknfnk}
		(\nabla^{\textup{flat}})_{X}Y=\sum_{i=1}^{n}X(Y^{i})\dfrac{\partial}{\partial x_{i}},
	\end{eqnarray}
	where $X,Y$ are vector fields on $\mathbb{R}^{n}$, $Y=\sum_{i=1}^{n}Y^{i}\tfrac{\partial}{\partial x_{i}}$.
	
	The next result describes a simple yet important class of dually flat manifolds.   
\begin{lemma}
	Let $\psi:U\to \mathbb{R}$ be a smooth function defined on an open set $U\subseteq \mathbb{R}^{n}$. Suppose that the Hessian matrix of $\psi$, 
	denoted by $\textup{Hess}(\psi)$, is positive definite at each point of $U$ (thus $\textup{Hess}(\psi)$ can be regarded as a Riemannian metric on $\mathbb{R}^{n}$). 
	Then $(U,\textup{Hess}(\psi),\nabla^{\textup{flat}})$ is a dually flat manifold. 
\end{lemma}
\begin{proof}
	Let $\nabla$ be the dual connection of $\nabla^{\textup{flat}}$ with respect to $h=\textup{Hess}(\psi)$. In view of Proposition \ref{neewknkfenknkn}, 
	it suffices to show that the torsion $T$ of $\nabla$ is identically zero. Let $(x_{1},...,x_{n})$ denote the usual coordinates on $\mathbb{R}^{n}$. Given $i,j,k=1,...,n$, 
	we compute:
	\begin{eqnarray*}
		\lefteqn{h\big(T(\tfrac{\partial}{\partial x_{i}},\tfrac{\partial}{\partial x_{j}}),\tfrac{\partial}{\partial x_{k}}\big)=
		h\Big(\nabla_{\tfrac{\partial}{\partial x_{i}}}\tfrac{\partial}{\partial x_{j}}-\nabla_{\tfrac{\partial}{\partial x_{j}}}\tfrac{\partial}{\partial x_{i}},
		\tfrac{\partial}{\partial x_{k}}\Big)}\\
		&=&\tfrac{\partial}{\partial x_{i}}h(\tfrac{\partial}{\partial x_{j}},\tfrac{\partial}{\partial x_{k}})
		-h\Big(\tfrac{\partial}{\partial x_{j}},\nabla^{\textup{flat}}_{\tfrac{\partial}{\partial x_{i}}}\tfrac{\partial}{\partial x_{k}}\Big)
		-\tfrac{\partial}{\partial x_{j}}h(\tfrac{\partial}{\partial x_{i}},\tfrac{\partial}{\partial x_{k}})
		+h\Big(\tfrac{\partial}{\partial x_{i}},\nabla^{\textup{flat}}_{\tfrac{\partial}{\partial x_{j}}}\tfrac{\partial}{\partial x_{k}}\Big)\\
		&=& \tfrac{\partial}{\partial x_{i}}\tfrac{\partial^{2}\psi}{\partial x_{j}\partial x_{k}} -\tfrac{\partial}{\partial x_{j}}\tfrac{\partial^{2}\psi}{\partial x_{i}\partial x_{k}}=0.
	\end{eqnarray*}
	The result follows.  
\end{proof}


\begin{definition}\label{nfekndwnknknk}
	Suppose $(M,\nabla)$ is an affine manifold. An \textit{affine coordinate system} is 
	a coordinate system $(x_{1},...,x_{n})$ defined on some open set $U\subseteq M$ such that 
	$\nabla_{\tfrac{\partial}{\partial x_{i} }}\tfrac{\partial}{\partial x_{j}}=0$
	for all $i,j=1,...,n$. 
\end{definition}

	It can be shown that for every point $p$ in an affine manifold $M$, there is an affine coordinate system 
	$(x_{1},...,x_{n})$ defined on some neighborhood $U\subseteq M$ of $p$ (see \cite{Shima}, Proposition 1.1). 

\begin{definition}\label{nfkdwnkenfkdk}
	Let $(M,h,\nabla)$ be a dually flat manifold, with dual connection $\nabla^{*}$, and 
	let $x=(x_{1},...,x_{n})$ and $y=(y_{1},...,y_{n})$ be coordinate systems on the same open set 
	$U\subseteq M$. The pair $(x,y)$ is said to be a 
	\textit{pair of dual coordinate systems} if the following conditions are satisfied.
	\begin{enumerate}[(1)]
		\item $x$ is affine with respect to $\nabla$.
		\item $y$ is affine with respect to $\nabla^{*}$.
		\item $h\bigg(\dfrac{\partial}{\partial x_{i}}, \dfrac{\partial}{\partial y_{j}}\bigg)=\delta_{ij}$ $(=\textup{Kronecker delta)}$ for all 
			$i,j=1,...,n$. 
	\end{enumerate}
\end{definition}
	When $U=M$ in the above definition, the pair $(x,y)$ is said to be a \textit{global pair of dual coordinate systems}.  
	It can be shown that there exists a pair of dual coordinate systems around each point $p\in M$ 
	(see \cite{Amari-Nagaoka}, Section 3.3). 

\begin{proposition}\label{neknkneknekenfkenk}
	Let $(M,h,\nabla)$ be a dually flat manifold and let $(x,y)=\big((x_{1},...,x_{n}),(y_{1},...,y_{n})\big)$ be a pair of dual 
	coordinate systems on $U\subset M.$ Given $i,j=1,...,n$, let $h_{ij}:U\to \mathbb{R}$ and $h^{ij}:U\to \mathbb{R}$ be defined by $h_{ij}(x)=
	h_{x}(\tfrac{\partial}{\partial x_{i}}, \tfrac{\partial}{\partial x_{j}})$ and $h^{ij}(x)=h_{x}(\tfrac{\partial}{\partial y_{i}}, \tfrac{\partial}{\partial y_{j}})$. 
	\begin{enumerate}[(1)]
		\item For every $x\in U$, the matrices $[h_{ij}(x)]$ and $[h^{ij}(x)]$ are inverses of each other. 
		\item $\dfrac{\partial x_{i}}{\partial y_{j}}=h^{ij}$ and $\dfrac{\partial y_{i}}{\partial x_{j}}=h_{ij}$. 
		\item $\dfrac{\partial}{\partial x_{i}}=\displaystyle\sum_{j=1}^{n}h_{ij}\dfrac{\partial}{\partial y_{j}}$ and 
			$\dfrac{\partial}{\partial y_{i}}=\displaystyle\sum_{j=1}^{n}h^{ij}\dfrac{\partial}{\partial x_{j}}$. 
		\item If $U$ is connected and simply connected, then there are smooth functions $\psi,\phi:U\to \mathbb{R}$ such that 
			$\dfrac{\partial \psi}{\partial x_{i}}=y_{i}$ and $\dfrac{\partial \phi}{\partial y_{i}}=x_{i}$
			for all $i=1,...,n$. 
		\item Suppose $U$ connected and simply connected. Let $\psi,\phi:U\to \mathbb{R}$ be as in (4). Then, 
			\begin{enumerate}[(a)]
				\item $\dfrac{\partial^{2} \psi}{\partial x_{i}\partial x_{j}}=h_{ij}$ and $\dfrac{\partial^{2} \phi}{\partial y_{i}\partial y_{j}}=h^{ij}$ 
					for all $i,j=1,...,n.$
				\item $\psi+\phi-\displaystyle\sum_{k=1}^{n}x_{k}y_{k}$ is constant on $U$. 
			\end{enumerate}
	\end{enumerate}
\end{proposition}
\begin{proof}
	See \cite{Amari-Nagaoka}, Section 3.3.
\end{proof}

\begin{lemma}\label{nefkwnknkenknk}
	Let $(M,h,\nabla)$ be a dually flat manifold and let $x=(x_{1},...,x_{n})$ be a $\nabla$-affine coordinate system on $U\subseteq M$ such that 
	$x(U)\subseteq \mathbb{R}^{n}$ is convex. Suppose there exists 
	a smooth function $\psi:U\to \mathbb{R}$ such that $\tfrac{\partial^{2}\psi}{\partial x_{i}\partial x_{j}}=h_{ij}=
	h(\tfrac{\partial}{\partial x_{i}},\tfrac{\partial}{\partial x_{j}})$ for all $i,j=1,...,n$ on $U$. Let $y=(y_{1},...,y_{n}):U\to \mathbb{R}^{n}$ be defined by 
	$y=\textup{grad}(\psi)=(\tfrac{\partial \psi}{\partial x_{1}}...,\tfrac{\partial \psi}{\partial x_{n}})$. Then $(x,y)$ is a pair of dual 
	coordinate systems on $U$. 
\end{lemma}
\begin{proof}
	In order to simplify notation, we identify $U$ and $x(U)$. Thus $\psi$ is a smooth function defined on the convex set $U\subseteq \mathbb{R}^{n}$, whose Hessian 
	is the matrix representation of $h$. In particular, $\psi$ is strictly convex. 
	Because $\tfrac{\partial y_{i}}{\partial x_{j}}=\tfrac{\partial^{2} \psi}{\partial x_{j}\partial x_{i}}=h_{ji}$, $y$ is a local diffeomorphism and $y(U)$ is 
	an open subset of $\mathbb{R}^{n}$. To see that $y$ is injective, let $x_{1},x_{2}$ be a pair of points in $U$ such that $x_{1}\neq x_{2}$. Since 
	$\psi$ is strictly convex, $\psi(x_{2})>\psi(x_{1})+\langle \textup{grad}(\psi)(x_{1}), x_{2}-x_{1}\rangle$, 
	where $\langle\,,\,\rangle$ is the Euclidean pairing on $\mathbb{R}^{n}$. Reversing the roles of $x_{1}$ and $x_{2}$, and adding the two inequalities we get 
	\begin{eqnarray*}
		\langle y(x_{1})-y(x_{2}), x_{1}-x_{2}    \rangle>0.
	\end{eqnarray*}
	This forces $y(x_{1})\neq y(x_{2})$ and shows that $y$ is injective. It follows that $y:U\to y(U)$ 
	is a diffeomorphism. In particular, $(y_{1},...,y_{n})$ is a coordinate system on $U$. Given $a=1,...,n$, we compute:
	\begin{eqnarray*}
		\dfrac{\partial}{\partial x_{a}}=\sum_{k=1}^{n} \dfrac{\partial y_{k}}{\partial x_{a}} \dfrac{\partial}{\partial y_{k}}=
		\sum_{k=1}^{n} \dfrac{\partial}{\partial x_{a}}\bigg(\dfrac{\partial \psi}{\partial x_{k}}\bigg)\dfrac{\partial}{\partial y_{k}}
		=\sum_{k=1}^{n}h_{ak}\dfrac{\partial}{\partial y_{k}}, 
	\end{eqnarray*}
	where we have used $y_{a}=(\textup{grad}(\psi))_{a}=\tfrac{\partial \psi}{\partial x_{a}}$ and the fact that $h$ is the Hessian of $\psi.$ Inverting these equations, 
	we get 
	\begin{eqnarray*}
		\dfrac{\partial}{\partial y_{a}}= \sum_{k=1}^{n}h^{ak}\dfrac{\partial}{\partial x_{k}}, 
	\end{eqnarray*}
	where $h^{ij}$ is the $(i,j)$-entry of the inverse of the matrix $[h_{ij}]$. It follows that 
	\begin{eqnarray*}
		h\bigg(\dfrac{\partial}{\partial x_{a}},\dfrac{\partial}{\partial y_{b}} \bigg)=\sum_{k=1}^{n}h^{bk} 
		h\bigg(\dfrac{\partial}{\partial x_{a}},\dfrac{\partial}{\partial x_{k}} \bigg)= \sum_{k=1}^{n}h^{bk}h_{ak}=\delta_{ab}.
	\end{eqnarray*}
	It remains to show that $y$ is affine with respect to the dual connection $\nabla^{*}$ of $\nabla$. Given $i,j,k=1,...,n$, we have
	\begin{eqnarray*}
		h\bigg(\nabla^{*}_{\tfrac{\partial}{\partial y_{i}}}\dfrac{\partial}{\partial y_{j}}, \dfrac{\partial}{\partial x_{k}}\bigg)=
		\dfrac{\partial}{\partial y_{i}}h\bigg(\dfrac{\partial}{\partial y_{j}},\dfrac{\partial}{\partial x_{k}}\bigg)-h\bigg(\dfrac{\partial}{\partial y_{j}}, 
		\nabla_{\tfrac{\partial}{\partial y_{i}}}\dfrac{\partial}{\partial x_{k}}\bigg)=0,
	\end{eqnarray*}
	where we have used the following facts: (1) $h(\tfrac{\partial}{\partial y_{j}},\tfrac{\partial}{\partial x_{k}})=\delta_{jk}$ is constant and (2) $\tfrac{\partial}{\partial x_{k}}$ is 
	parallel with respect to $\nabla$ (since $x=(x_{1},...,x_{n})$ is $\nabla$-affine), which implies 
	$\nabla_{\tfrac{\partial}{\partial y_{i}}}\tfrac{\partial}{\partial x_{k}}=0.$ 
	Therefore $\nabla^{*}_{\tfrac{\partial}{\partial y_{i}}}\tfrac{\partial}{\partial y_{j}}=0$ 
	for all $i,j$, which shows that $y$ is $\nabla^{*}$-affine. 
\end{proof}

	Let $\varphi:M\to M'$ be a diffeomorphism between manifolds and $\nabla$ a linear connection on $M'$. We will use the following notation:
	\begin{itemize}
		\item Given a vector field $X$ on $M$, $\varphi_{*}X$ is the vector field on $M'$ defined by $(\varphi_{*}X)(p)=\varphi_{*_{\varphi^{-1}(p)}}X(\varphi^{-1}(p))$. 
		\item $\varphi^{*}\nabla$ is the unique connection on $M$ satisfying $(\varphi^{*}\nabla)_{X}Y=(\varphi^{-1})_{*}\big(\nabla_{\varphi_{*}X}\varphi_{*}Y\big)$ 
			for all vector fields $X,Y$ on $M$. 
	\end{itemize}

\begin{definition}
	Let $(M,h,\nabla)$ and $(M',h',\nabla')$ be dually flat spaces. Let $\nabla^{*}$ and $(\nabla')^{*}$ denote the dual connections of $\nabla$ and $\nabla'$ with respect to $h$ and $h'$, 
	respectively. A diffeomorphism $\varphi:M\to M'$ is said to be an \textit{isomorphism of dually flat spaces} 
	if $\varphi^{*}h'=h$, $\varphi^{*}\nabla'=\nabla$ and $\varphi^{*}(\nabla')^{*}=\nabla^{*}$. 
\end{definition}

	In practice, it suffices to prove that $\varphi^{*}h'=h$ and $\varphi^{*}\nabla'=\nabla$, as the following lemma shows.  

\begin{lemma}\label{ncekndwkneknku}
	Let $(M,h)$ and $(M',h')$ be Riemannian manifolds and let $\nabla$ and $\nabla'$ be affine connections on $M$ and $M'$, respectively. 
	We denote the dual connections by $\nabla^{*}$ and $(\nabla')^{*}$, respectively. Let $\varphi:M\to M'$ be a diffeomorphism. If 
	$\varphi^{*}h'=h$ and $\varphi^{*}\nabla'=\nabla$, then $\varphi^{*}(\nabla')^{*}=\nabla^{*}$.
\end{lemma}
\begin{proof}
	Let $X,Y,Z$ be arbitrary vector fields on $M$. Given $x\in M$, we compute
	\begin{eqnarray*}
		\lefteqn{h_{x}\big((\varphi^{*}(\nabla')^{*})_{X}Y,Z\big)=(\varphi^{*}h')_{x}\big((\varphi^{*}(\nabla')^{*})_{X}Y,Z\big)}\\
		&=& h'_{\varphi(x)}\big(\varphi_{*_{x}}(\varphi^{*}(\nabla')^{*})_{X}Y,\varphi_{*_{x}}Z\big)=h'_{\varphi(x)}\big((\nabla')^{*}_{\varphi_{*}X}\varphi_{*}Y,\varphi_{*}Z\big)\\
		&=&  (\varphi_{*}X)\big( h'(\varphi_{*}Y,\varphi_{*}Z)\big)-h'_{\varphi(x)}\big(\varphi_{*}Y,\nabla'_{\varphi_{*}X}\varphi_{*}Z\big)\\
		&=& X\big(h'(\varphi_{*}Y,\varphi_{*}Z)\circ \varphi\big)-h'_{\varphi(x)}\big(\varphi_{*}Y,\varphi_{*_{x}}(\varphi^{-1})_{*_{\varphi(x)}}
			\nabla'_{\varphi_{*}X}\varphi_{*}Z\big)\\
		&=& X\big(h(Y,Z)\big) - h_{x}\big(Y,(\varphi^{*}\nabla')_{X}Z \big)=X\big(h(Y,Z)\big) - h_{x}\big(Y,\nabla_{X}Z \big)\\
		&=& h_{x}\big(\nabla_{X}Y,Z\big)
	\end{eqnarray*}
	and hence $h_{x}\big((\varphi^{*}(\nabla')^{*})_{X}Y,Z\big)=h_{x}\big(\nabla_{X}Y,Z\big)$. It follows that $\varphi^{*}(\nabla')^{*} =\nabla$. 
\end{proof}

\subsection{Dombrowski's construction.}\label{nknknfeknaekn}

	Let $M$ be a smooth manifold endowed with a connection $\nabla$. We will denote by 
	$\pi:TM\rightarrow M$ the canonical projection. 

	Given $p\in M$ and $u\in T_{p}M$, the connection $\nabla$ induces an identification of 
	vector spaces $T_{u}(TM)\cong T_{p}M\oplus T_{p}M$ given by $A\in T_{u}(TM)\mapsto (\pi_{*_{u}}A,KA)\in T_{p}M\oplus T_{p}M$, where 
	$\pi_{*_{u}}$ is the derivative of $\pi$ at $u$ and $K$ is the connector associated to $\nabla$ (see Section \ref{pouette}). 
	If there is no danger of confusion, we will therefore regard an element of $T_{u}(TM)$ as a pair $(v,w)$, where $v,w\in T_{p}M$.

	Let $h$ be a Riemannian metric on $M$. The pair $(h,\nabla)$ determines an almost Hermitian 
	structure on $TM$ via the following formulas:

	\begin{alignat*}{5}\label{equation definition G, omega, etc.}
		g_{u}\big(\big(v,w\big),
			\big(\overline{v},
			\overline{w}\big)\big)\quad:=&\quad
			h_{p}\big(v,\overline{v}\big)+
			h_{p}\big(w,\overline{w}\big),&\textup{}\,\,\,\,\,\,\,\,\,
			(\textup{metric})&\nonumber\\
		\omega_{u}\big(\big(v,w\big),
			\big(\overline{v},
			\overline{w}\big)\big)\quad:=&\quad h_{p}\big(v,\overline{w}\big)-
			h_{p}\big(w,\overline{v}\big),&\textup{}\,\,\,\,\,\,\,\,\,
			(\textup{2-form})&\nonumber\\
		J_{u}\big(\big(v,w\big)\big)\quad:=&\quad
		(-w,v),&
		(\textup{almost complex structure})
	\end{alignat*}
	where $u,v,w,\overline{v},\overline{w}\;\in\; T_{p}M$.

	The tensors $g,J,\omega$ are smooth (this will follow from their coordinate representation, 
	see Proposition \ref{prop:4.1} below) and clearly, 
	$J^{2}=-Id$, $g(Ju,Jv)=g(u,v)$ and $\omega(u,v)=g(Ju,v)$ for all $u,v\in TM$ such that $\pi(u)=\pi(v)$. 
	Thus, $(TM,g,J)$ is an almost Hermitian manifold with fundamental form $\omega$. 
	This is \textit{Dombrowski's construction} \cite{Dombrowski}. 
	
	Note that $\pi:(TM,g)\to (M,h)$ is a Riemannian submersion. \\

	We now review the analytical properties of Dombrowski's construction. Recall that an almost Hermitian structure $(g,J,\omega)$ is K\"{a}hler 
	when the following two analytical conditions are met: (1) $J$ is integrable; (2) $d\omega=0$.

%
\begin{proposition}\label{prop:3.8} 
	Let $(h,\nabla,\nabla^{*})$ be a dualistic structure on $M$ and let $(g,J,\omega)$ be 
	the almost Hermitian structure on $TM$ associated to $(h,\nabla)$ via Dombrowski's construction. 
	The following are equivalent.
	\begin{enumerate}[(1)]
	\item $(TM,g,J,\omega)$ is a K\"{a}hler manifold.
	\item $(h,\nabla,\nabla^{*})$ is dually flat.
	\end{enumerate}
\end{proposition}
\begin{proof}
	See \cite{Dombrowski,Molitor-exponential}.
\end{proof}

	We now direct our attention to the coordinate expressions for $g,J$ and $\omega$. 

%

\begin{proposition}\label{prop:4.1}
	Let $(M,h,\nabla)$ be a dually flat manifold and let $(g, J,\omega)$ 
	be the K\"{a}hler structure on $TM$ associated to $(h,\nabla)$ via Dombrowski's construction. 
	Let $x=(x_{1},...,x_{n})$ be an affine coordinate system with respect to $\nabla$ on $U\subseteq M$, 
	and let $(q,r)=(q_{1},...,q_{n},r_{1},...,r_{n})$ denote the corresponding coordinates on $\pi^{-1}(U)$, 
	as described before Lemma \ref{nfkewefnknknk}. Then, in the coordinates $(q,r)$,
	\begin{equation}\label{eq:4.2}
		g=\begin{bmatrix}h_{ij} & 0\\
		0 & h_{ij}
		\end{bmatrix},\quad J=
		\begin{bmatrix}0 & -I_{n}\\
		I_{n} & 0
		\end{bmatrix},
		\quad\omega=
		\begin{bmatrix}0 & h_{ij}\\
		-h_{ij} & 0
		\end{bmatrix},
	\end{equation} 
	where $h_{ij}=h\big (\tfrac{\partial}{\partial x_{i}},\tfrac{\partial}{\partial x_{j}}\big)$,
	$i,j=1,...,n$. 
\end{proposition}
\begin{proof}
	See \cite{Molitor2014}.
\end{proof}

	In what follows, we will identify $T\mathbb{R}^{n}$ with $\mathbb{C}^{n}$ via the correspondence $T\mathbb{R}^{n}\to \mathbb{C}^{n}$
	defined by 
	\begin{eqnarray*}
		\sum_{k=1}^{n}b_{k}\dfrac{\partial}{\partial x_{k}}\bigg\vert_{a}\mapsto  a+ib, 
	\end{eqnarray*}
	where $a,b\in \mathbb{R}^{n}$, $b=(b_{1},...,b_{n})$ and where $(x_{1},...,x_{n})$ are standard coordinates on $\mathbb{R}^{n}$. 

	An immediate consequence of Proposition \ref{prop:4.1} is the following result.  

\begin{corollary}\label{nfednwkneknkn}
	Let $h$ be a Riemannian metric on $\mathbb{R}^{n}$ and let $\nabla^{\textup{flat}}$ be the canonical flat connection on $\mathbb{R}^{n}$. 
	Let $(g,J)$ be the almost Hermitian structure on $T\mathbb{R}^{n}=\mathbb{C}^{n}$ associated to 
	$(h,\nabla^{\textup{flat}})$ via Dombrowski's construction. Then, 
	\begin{enumerate}[(1)]
		\item $J$ is the canonical complex structure of $\mathbb{C}^{n}$ (= multiplication by $i$).
		\item $g_{x+iy}(a+ib,a'+ib')=h_{x}(a,a')+h_{x}(b,b')$ for all $x,y,a,a',b,b'\in \mathbb{R}^{n}$. 
	\end{enumerate}
\end{corollary}

	Now we turn our attention to affine maps between dually flat manifolds. Let $(M,h,\nabla)$ and $(M',h',\nabla')$ be connected dually flat 
	manifolds of dimension $n$ and $d$, respectively. We endow $TM$ and $TM'$ with their natural K\"{a}hler structures (coming from Dombrowski's construction).

	Recall that a smooth function $f:M\to M'$ is \textit{affine} if for every $p\in M$, there is a $\nabla$-affine coordinate system 
	$x:U\subseteq M\to \mathbb{R}^{n}$ and a $\nabla'$-affine coordinate system $x':U'\subseteq M'\to \mathbb{R}^{d}$ 
	such that $p\in U$, $f(U)\subseteq U'$ and $x'\circ f\circ x^{-1}:x(U)\to \mathbb{R}^{d}$ is the restriction of an affine map 
	(thus $(x'\circ f\circ x^{-1})(y)=Ay+B$ for all $y\in x(U)$, where $A$ is a $d\times n$ real matrix and $B\in \mathbb{R}^{d}$).

	When $f$ is a diffeomorphism, an easy verification shows that $f$ is affine if and only if $f^{*}\nabla'=\nabla.$

\begin{proposition}\label{nfeknwknekn}
	Let $f:M\to M'$ be a smooth map. 
	\begin{enumerate}[(1)]
	\item The derivative $f_{*}:TM\to TM'$ is holomorphic if and only if $f$ is affine. 
	\item Suppose $f$ affine. Then $f_{*}$ is isometric if and only if $f$ is isometric.
	\end{enumerate}
	Consequently, $f_{*}$ is a K\"{a}hler immersion if and only if $f$ is an isometric affine immersion. 
\end{proposition}
\begin{proof}
	Let $x=(x_{1},...,x_{n})$ be a $\nabla$-affine coordinate system on $U\subseteq M$ and let 
	$x'=(x_{1}',...,x_{d}')$ be a $\nabla'$-affine coordinate system on $U'\subseteq M'$. Suppose $U$ connected and $f(U)\subseteq U'$. 
	We denote by $(q,r)=(q_{1},...,q_{n},r_{1},...,r_{n})$ and $(q',r')=(q_{1}',...,q_{d}',r_{1}',...,r_{d}')$ 
	the corresponding coordinates on $TM$ and $TM'$, respectively (see before Lemma \ref{nfkewefnknknk}). 
	For simplicity, we will use the same symbols $``f"$ and $``f_{*}"$ for 
	the local expressions for $f$ and $f_{*}$, respectively. Thus we write $f(x)=(f^{1}(x),...,f^{d}(x))$ and $f_{*}(q,r)=(f(q),f_{*_{q}}r)$. 
	
	The derivative of $f_{*}$ at $(q,r)$ in the direction $(u,v)\in \mathbb{R}^{n}\times \mathbb{R}^{n}$ is given by 
	\begin{eqnarray*}	
		(f_{*})_{*_{(q,r)}}(u,v)=\big(f_{*_{q}}(u), A(q,u)r + f_{*_{q}}v\big), 
	\end{eqnarray*}
	where $A(q,u)$ is the $d\times n$ matrix whose $(i,j)$-entry is $\big(\tfrac{\partial f^{i}}{\partial x_{j}}\big)_{*_{q}}u
	=\sum_{k=1}^{n}u_{k}\tfrac{\partial^{2} f^{i}(q)}{\partial x_{k}\partial x_{j}}$. Let $J$ and $J'$ be the complex structures of 
	$M$ and $M'$, respectively. Locally, $J_{(q,r)}(u,v)=(-v,u)$ and $J'_{(q',r')}(u',v')=(-v',u')$ (see Proposition \ref{prop:4.1}). 
	Therefore $(f_{*})_{*}\circ J=J'\circ (f_{*})_{*}$ on $\pi^{-1}(U)$ if and only if 
	\begin{eqnarray*}
		&&\big(-f_{*_{q}}(v), A(q,-v)r + f_{*_{q}}u\big)
			=\big(-A(q,u)r - f_{*_{q}}v,f_{*_{q}}(u)\big)\,\,\,\,\,\,\,\,\,\forall\,q\in U,\,\,\forall\,u,v,r\in\mathbb{R}^{n}\\
		&\Leftrightarrow& A(q,u)r=A(q,v)r=0\,\,\,\,\,\,\,\,\,\forall\,q\in U,\,\,\forall\,u,v,r\in\mathbb{R}^{n}\\
		&\Leftrightarrow& \tfrac{\partial^{2} f^{i}}{\partial x_{k}\partial x_{j}}=0\,\,\,\textup{for all}\,\,k,i,j\,\,\textup{on}\,\,U. 
	\end{eqnarray*}
	Therefore $f_{*}$ is holomorphic on $\pi^{-1}(U)$ if and only if all partial derivatives $\tfrac{\partial^{2} f^{i}}{\partial x_{k}\partial x_{j}}$ vanish on $U$. 
	Since $U$ is connected, this is equivalent to the existence of a $d\times n$ real matrix $B$ and $C\in \mathbb{R}^{n}$ such that 
	$f(x)=Bx+C$ for all $x\in U$. It follows that $f_{*}$ is holomorphic on $TM$ if and only if $f$ is affine. This shows (1). 

	Assume now that $f$ is affine. Let $g$ and $g'$ be the Riemannian metrics on $TM$ and $TM'$, respectively. 
	With the same notation as above, the coordinate expression for $g$ and $g'$ are given by 
	$g(q,r)= 
	\left(\begin{smallmatrix}h_{ij}  &   0  
	\\  0   &   h_{ij}\end{smallmatrix}\right) $ 
	and 
	$g'(q',r')= 
	\left(\begin{smallmatrix}h'_{ij}  &   0  
	\\  0   &   h'_{ij}\end{smallmatrix}\right) $,
	where $h_{ij}=h(\tfrac{\partial}{\partial x_{i}},\tfrac{\partial}{\partial x_{j}})$ and $h'_{ij}
	=h'(\tfrac{\partial}{\partial x'_{i}},\tfrac{\partial}{\partial x'_{j}})$ (see Proposition \ref{prop:4.1}). 
	Given $(q,r)\in U\times \mathbb{R}^{n}$ and $(u,v),(u',v')\in \mathbb{R}^{n}\times \mathbb{R}^{n}$, we compute:
	\begin{eqnarray*}
		\lefteqn{\big((f_{*})^{*}g'\big)_{(q,r)}\big((u,v),(u',v')\big)
			=g'_{(f_{*})(q,r)}\big((f_{*})_{*}(u,v), (f_{*})_{*}(u',v')\big)}\\
	&=& g'_{(f_{*})(q,r)}\big[\big(f_{*_{q}}(u), A(q,u)r + f_{*_{q}}v\big),\big(f_{*_{q}}(u'), A(q,u')r + f_{*_{q}}v'\big)\big]\\
	&=&g'_{(f_{*})(q,r)}\big[\big(f_{*_{q}}(u),f_{*_{q}}v\big),\big(f_{*_{q}}(u'),f_{*_{q}}v'\big)\big]\\
	&=& h'_{f(q)}\big( f_{*_{q}}(u), f_{*_{q}}(u')\big)+h'_{f(q)}\big( f_{*_{q}}(v), f_{*_{q}}(v')\big)\\
	&=& (f^{*}h')_{q}(u,u')+(f^{*}h')_{q}(v,v'),
	\end{eqnarray*}
	where we have used the fact that $A(q,u)=0$, since $f$ is affine. It follows that $(f_{*})^{*}g'=g$ if and only 
	if $f^{*}h'=h$. This shows (2) and concludes the proof of the proposition.
\end{proof}

\section{Parallel lattices and torus actions}\label{nkdnksnkenkn}\label{nfekkknwdkdnk}

	Throughout this section $(M,\nabla)$ is a connected affine manifold of dimension $n$.
\begin{definition}\label{neknwdkneknk}
	A subset $L\subset TM$ is said to be a 
	\textit{parallel lattice} with respect to $\nabla$ if there are $n$ parallel vector fields 
	$X_{1},...,X_{n}$ on $M$ such that:
	\begin{enumerate}[(1)]
	\item $\{X_{1}(p),...,X_{n}(p)\}$ is a basis for $T_{p}M$ for every $p\in M$, 
	\item $L=\big\{k_{1}X_{1}(p)+...+k_{n}X_{n}(p)\,\,\big\vert\,\,p\in M,\,\,k_{1},...,k_{n}\in \mathbb{Z}\big\}$. 
	\end{enumerate}
	In that case we shall write $L=L_{X}$, where $X=(X_{1},...,X_{n})\in \mathfrak{X}(M)^{n}$, and call $X$ a \textit{generator} 
	for $L$. 
\end{definition}

	As a matter of notation, we will denote the set of all generators of $L$ by $\textup{gen}(L)$. Given $X\in \textup{gen}(L)$, we say that 
	$L$ is \textit{generated by} $X$. 

\begin{remark}
	Let $L\subset TM$ be a parallel lattice and $X=(X_{1},...,X_{n})\in \textup{gen}(L)$. 
	\begin{enumerate}[(1)]
		\item For every $p\in M$, the intersection $L\cap T_{p}M$ is a full rank lattice.
		\item $X=(X_{1},...,X_{n})$ is a global frame for $M$. Therefore $M$ is parallelizable. 
		\item $[X_{i},X_{j}]=0$ for all $i,j=1,...,n$. This comes from the fact that parallel vector fields 
			on an affine manifold commute. 
	\end{enumerate}
\end{remark}

	In what follows, let $L\subset TM$ be a fixed parallel lattice. If 
	$X=(X_{1},...,X_{n})\in \textup{gen}(L)$, then, for any $p\in M$ and any $n\times n$ real matrix 
	$A=(a_{ij})$, we let 
	\begin{eqnarray*}
		 &&X(p):=(X_{1}(p),...,X_{n}(p))\in (T_{p}M)^{n},\\
		 &&AX	
		:=\bigg(
		\begin{smallmatrix}
			a_{11} & \cdots & a_{1n}  \\
			\vdots &        & \vdots  \\
			a_{n1} & \cdots & a_{nn} 
		\end{smallmatrix}
		\bigg)
		\,
		\bigg(
		\begin{smallmatrix}
			X_{1}\\
			\vdots \\
			X_{n}
		\end{smallmatrix}
		\bigg)
		:=\bigg(\sum_{i=1}^{n}a_{1i}X_{i},....,
		\sum_{i=1}^{n}a_{ni}X_{i}\bigg)\in \mathfrak{X}(M)^{n}.
	\end{eqnarray*}

	Let $\textup{GL}(n,\mathbb{Z})$ denote the group of invertible $n\times n$ matrices with integer entries. 
\begin{lemma}\label{nkqndwknwfkn}
	Let $X,X'\in \mathfrak{X}(M)^{n}$ be frames. 
	\begin{enumerate}[(1)]
		\item If $X\in \textup{gen}(L$), then a vector field $Y$ on $M$ 
			is parallel if and only if there are real numbers 
			$\lambda_{1},...,\lambda_{n}$ such that $Y=\lambda_{1}X_{1}+...+\lambda_{n}X_{n}$. 
		\item If $X$ and $X'$ are both generators for $L$, then there exists $A\in \textup{GL}(n,\mathbb{Z})$ such that 
			$X=AX'$.
	\end{enumerate}
\end{lemma}
\begin{proof}
	(1) Let $Y$ be a parallel vector field on $M$. Fix $p_{0}\in M$. Since 
	$\{X_{1}(p_{0}),...,X_{n}(p_{0})\}$ is a basis for $T_{p_{0}}M$, there are real numbers 
	$\lambda_{1},...,\lambda_{n}$ such that 
	\begin{eqnarray*}
		 Y(p_{0})=\lambda_{1}X_{1}(p_{0})+...+\lambda_{n}X_{n}(p_{0}). 
	\end{eqnarray*}
	Let $p\in M$ be arbitary. Since $M$ is connected, there exists a piecewise smooth curve 
	$c:[0,1]\to M$ such that $c(0)=p_{0}$ and $c(1)=p$. Because $(X_{1},...,X_{n})$ is a global frame, 
	there are functions $f_{1},...,f_{n}:[0,1]\to \mathbb{R}$ such that 
	\begin{eqnarray*}
		 Y(c(t))=\sum_{k=1}^{n}f_{k}(t)X_{k}(c(t))
	\end{eqnarray*}
	for every $t\in [0,1]$. Note that $f_{k}(0)=\lambda_{k}$ for every $k=1,..,n.$ 
	Let $(X_{1}^{*},...,X_{n}^{*})$ denote the dual coframe of $(X_{1},...,X_{n})$. 
	Thus, by definition, $\langle X_{i}^{*}, X_{j}\rangle=\delta_{ij}$ for every $i,j$, 
	where $\langle\,,\,\rangle$ is the natural pairing between $TM$ and $T^{*}M$. It is immediate that 
	\begin{eqnarray*}
		 f_{i}(t)=\langle X_{i}^{*},Y\rangle(c(t))
	\end{eqnarray*}
	for every $t\in [0,1]$ and every $i=1,...,n$, where $\langle X_{i}^{*},Y\rangle:M\to \mathbb{R},$ 
	$q\mapsto \langle X_{i}^{*}(q),Y(q)\rangle$. Since the latter function is smooth and since $c$ is piecewise smooth, 
	each $f_{i}$ is continuous on $[0,1]$ and smooth wherever $c$ is. 
	
	Let $\tfrac{D}{\partial t}$ denote the covariant derivative operator 
	of vector fields along $c$ induced by $\nabla$. Because $Y$ and $X_{k}$ are parallel, we have for 
	all $t\in [0,1]$ where $c$ is smooth:
	\begin{eqnarray*}
		 0=\dfrac{D}{\partial t} Y(c(t))=\sum_{k=1}^{n}\bigg[f_{k}'(t)X_{k}(c(t))
		+f_{k}(t)\dfrac{D}{\partial t}X_{k}(c(t))\bigg]=\sum_{k=1}^{n}f'_{k}(t)X_{k}(c(t))
	\end{eqnarray*}
	and hence $f'_{k}(t)=0$ for all $k$ and all $t$, except for finitely many points. 
	It follows from this and the continuity of $f_{k}$ that $f_{k}$ is constant. Thus 
	\begin{eqnarray*}
		Y(p)=\sum_{k=1}^{n}f_{k}(1)X_{k}(p)=\sum_{k=1}^{n}f_{k}(0)X_{k}(p)=\sum_{k=1}^{n}\lambda_{k}X_{k}(p). 
	\end{eqnarray*}
	Since $p\in M$ is arbitrary, $Y=\lambda_{1}X_{1}+...+\lambda_{n}X_{n}$. 

	(2) By the first item there exists an invertible $n\times n$ real matrix $A$ such that $X=AX',$ Thus, 
	at a particular point $p\in M$, we have the formula $X(p)=AX'(p)$, which can be interpreted as a change of basis 
	for the full rank lattice $L\cap T_{p}M$. By the general theory of lattices, $A\in \textup{GL}(n,\mathbb{Z})$. 
\end{proof}

	Given $X=(X_{1},...,X_{n})\in \textup{gen}(L)$, we will denote by $\Gamma(L,X)$ the 
	group of transformations of $TM$ of the form 
	\begin{eqnarray*}
		 TM\to TM,\,\,\,u\mapsto u+k_{1}X_{1}+...+k_{n}X_{n}. \,\,\,\,\,\,\,\,(k_{1},...,k_{n}\in \mathbb{Z})
	\end{eqnarray*}
	The group $\Gamma(L,X)$ is obviously isomorphic to $\mathbb{Z}^{n}$.  
\begin{lemma}
	Suppose $X,X'\in \textup{gen}(L)$. Then $\Gamma(L,X)=\Gamma(L,X')$. 
\end{lemma}
\begin{proof}
	This follows immediately from the preceding lemma. 
\end{proof}

	It follows that $\Gamma(L,X)$ is independent of the choice of $X\in \textup{gen}(L)$. We shall thus write 
	$\Gamma(L)$ instead of $\Gamma(L,X)$. Note that:
	\begin{description}
		\item[$\bullet$] $\Gamma(L)$ is a subgroup of $\textup{Diff}(TM)$, the group of diffeomorphisms of $TM$.
		\item[$\bullet$] For every $\gamma\in \Gamma(L)$, $\pi\circ \gamma=\pi$, where $\pi:TM\to M$ 
			is the canonical projection.
		\item[$\bullet$] $\Gamma(L)$ characterizes $L$, for 
			$L=\{\gamma(0_{p})\,\,\vert\,\,p\in M,\,\,\gamma\in \Gamma(L)\},$ 
			where $0_{p}$ is the zero vector in $T_{p}M$. 
	\end{description}

	Because the action of $\Gamma(L)$ on $TM$ is free and proper, the orbit space 
	\begin{eqnarray*}
		 M_{L}:=TM/\Gamma(L)
	\end{eqnarray*}
	is a smooth manifold and the quotient map 
	\begin{eqnarray*}
		 q_{L}:TM\to M_{L}
	\end{eqnarray*}
	is a covering map whose Deck transformation group is $\Gamma(L)$. Moreover, the fact that 
	$\pi\circ \gamma=\pi$ for every $\gamma\in \Gamma(L)$ implies that there exists a surjective submersion 
	$\pi_{L}:M_{L}\to M$ such that the following diagram commutes: 
	\begin{eqnarray}\label{nwknkrgnekwnkdndkns}
	\begin{tikzcd}
		TM \arrow{r}{\displaystyle q_{L}} \arrow[swap]{d}{\displaystyle\pi} & M_{L} \arrow{d}{\displaystyle\pi_{L}} \\
		M \arrow[swap]{r}{\displaystyle\textup{Id}} & M
	\end{tikzcd}
	\end{eqnarray}
	Let $\mathbb{T}^{n}=\mathbb{R}^{n}/\mathbb{Z}^{n}$ denote the $n$-dimensional torus. 
	Given $t=(t_{1},...,t_{n})\in \mathbb{R}^{n}$, we will denote by $[t]=[t_{1},...,t_{n}]$ the corresponding 
	equivalence class in $\mathbb{R}^{n}/\mathbb{Z}^{n}$. 

	Given $X\in \textup{gen}(L)$, we will denote by
	\begin{eqnarray*}
		\Phi_{X}:\mathbb{T}^{n}\times M_{L}\to M_{L}
	\end{eqnarray*}
	the torus action defined by 
	\begin{eqnarray}\label{nndkfneknkw}
		\Phi_{X}\big([t],q_{L}(u)\big):=q_{L}(u+t_{1}X_{1}+...+t_{n}X_{n}), 
	\end{eqnarray}
	where $t=(t_{1},...,t_{n})\in \mathbb{R}^{n}$ and $u\in TM$. 

\begin{remark}
	A simple verification shows that:
	\begin{enumerate}[(1)]
	\item $\Phi$ is effective, that is, the map $\mathbb{T}^{n}\to \textup{Diff}(M_{L}),$ $a\mapsto (\Phi_{X})_{a}$ is injective, and 
	\item $\pi_{L}\circ (\Phi_{X})_{a}=\pi_{L}$ for every $a\in \mathbb{T}^{n}$.  
	\end{enumerate} 
\end{remark}
\begin{lemma}\label{nfekdnkknkn}
	The map $f:M_{L}\to \mathbb{T}^{n}\times M$ given by 
	\begin{eqnarray*}
		f\big(q_{L}(u_{1}X_{1}(p)+...+u_{n}X_{n}(p))\big):=\big([u_{1},...,u_{n}],p\big)
	\end{eqnarray*}
	is a $\mathbb{T}^{n}$-equivariant diffeomorphism ($\mathbb{T}^{n}$ acts on $\mathbb{T}^{n}\times M$ via 
	translations on the first factor). 
\end{lemma}
\begin{proof}
	By a direct verification. 
\end{proof}

\begin{lemma}\label{nkendwkddknk}
	Suppose $X,X'\in \textup{gen}(L)$. Let $A\in \textup{GL}(n,\mathbb{Z})$ be the unique 
	matrix satisfying $X=AX'.$ Then for every $[t]\in \mathbb{T}^{n}$ and every 
	$q_{L}(u)\in M_{L}$, 
	\begin{eqnarray*}
		\Phi_{X}\big([t],q_{L}(u)\big)=\Phi_{X'}\big(\rho_{A^{T}}([t]),q_{L}(u)\big), 
	\end{eqnarray*}
	where $\rho_{A^{T}}:\mathbb{T}^{n}\to \mathbb{T}^{n}$, $[t]\mapsto [A^{T}t]$ ($A^{T}$ is the transpose of $A$). 
\end{lemma}
\begin{proof}
	By a direct verification. 
\end{proof}
	To summarize, a parallel lattice $L$ on a connected affine manifold $(M,\nabla)$ induces an effective 
	torus action on the quotient space $M_{L}=TM/\Gamma(L)$, which is unique modulo $\textup{Aut}(\mathbb{T}^{n})\cong
	\textup{GL}(n,\mathbb{Z})$. 

\section{Analytic properties}\label{nmknfkenkfnk}

	Throughout this section $(M,h,\nabla)$ is a dually flat connected manifold and 
	$(g,J,\omega)$ is the K\"{a}hler structure on $TM$ associated to $(h,\nabla)$ via Dombrowski's construction. 

	Let $L\subset TM$ be a fixed parallel lattice with respect to $\nabla$ and let $\Gamma(L)$ be the corresponding 
	subgroup of $\textup{Diff}(TM)$. 
\begin{lemma}\label{nkwnkfenkwn}
	Suppose $X=(X_{1},...,X_{n})\in \textup{gen}(L)$. 
	For any $t=(t_{1},...,t_{n})\in \mathbb{R}^{n}$, the map $T_{t}:TM\to TM$ defined by 
	\begin{eqnarray*}
		 T_{t}(u):=u+t_{1}X_{1}+...+t_{n}X_{n}, 
	\end{eqnarray*}
	is a holomorphic isometry.
\end{lemma}
\begin{proof}
	Let $(U,\varphi)$ be an affine chart for $M$ with respect to $\nabla$, 
	with local coordinates $(x_{1},...,x_{n})$. We denote by $(q_{1},...,q_{n},r_{1},...,r_{n})$ the corresponding 
	coordinates on $\pi^{-1}(U)\subseteq TM$ as described 
	before Lemma \ref{nfkewefnknknk}, where $\pi:TM\to M$ is the canonical projection. 

	Since each $X_{i}$ is parallel and since parallel vector fields are constant in affine coordinates, 
	there exists an invertible real matrix $A=(a_{ij})\in \textup{GL}(n,\mathbb{R})$ such that 
	for every $p\in U$ and every $i=1,...,n$, 
	\begin{eqnarray*}
		 X_{i}(p)=\sum_{j=1}^{n}a_{ij}\dfrac{\partial}{\partial x_{j}}\bigg\vert_{p},
	\end{eqnarray*}
	so if $u\in \pi^{-1}(U)$, then 
	\begin{eqnarray*}
		 T_{t}(u)=u+\sum_{j=1}^{n}b_{j}\dfrac{\partial}{\partial x_{j}}\bigg\vert_{p}, 
	\end{eqnarray*}
	where $b_{j}=\sum_{i=1}^{n}t_{i}a_{ij}$. It follows that the local expression for $T_{t}$ in the coordinates 
	$(q_{1},...,q_{n},r_{1},...,r_{n})$ is given by 
	\begin{eqnarray*}
		T_{t}(q,r)=(q,r+b), 
	\end{eqnarray*}
	where $b=(b_{1},...,b_{n})\in \mathbb{R}^{n}$. Now, a simple calculation using the local expressions for $g$ and 
	$J$ in the coordinates $(q_{1},...,q_{n},r_{1},...,r_{n})$ shows that $T_{t}^{*}g=g$ and 
	$(T_{t})_{*}\circ J=J\circ (T_{t})_{*}$ on $\pi^{-1}(U)$. It follows that $T_{t}$ is isometric and holomorphic. 
\end{proof}

\begin{corollary}\label{nfkndkefnknk}
	Each $\gamma\in \Gamma(L)$ is a holomorphic and isometric map. 
\end{corollary}
	It follows from the corollary above that $M_{L}=TM/\Gamma(L)$ is a K\"{a}hler manifold for 
	which the quotient map 
	\begin{eqnarray*}
		q_{L}: TM\to M_{L}
	\end{eqnarray*}
	is a holomorphic and locally isometric covering map with Deck transformation group $\Gamma(L)$. 
	Moreover, it follows from the formula $\pi=\pi_{L}\circ q_{L}$ (see \eqref{nwknkrgnekwnkdndkns}) and the fact that 
	the canonical projection $\pi:TM\to M$ is a Riemannian submersion that $\pi_{L}:M_{L}\to M$ is a Riemannian submersion.

	Let $X\in \textup{gen}(L)$ be arbitrary and let $\Phi_{X}:\mathbb{T}^{n}\times 
	M_{L}\to M_{L}$ be the corresponding torus action as defined in \eqref{nndkfneknkw}. 

\begin{lemma}\label{nndknfkdnfkdnk}
	For every $[t]\in \mathbb{T}^{n}$, the map $M_{L}\to M_{L}$, $p\mapsto \Phi_{X}([t],p),$ 
	is holomorphic and isometric. 
\end{lemma}
\begin{proof}
	Let $\Phi: M_{L}\to M_{L},$ $p\mapsto \Phi_{X}([t],p).$ By definition of $\Phi_{X}$ 
	(see \eqref{nndkfneknkw}), we have 
	\begin{eqnarray*}
		\Phi\circ q_{L}=q_{L}\circ T_{t}, 
	\end{eqnarray*}
	and since $q_{L}$ is a holomorphic and locally isometric covering map, we see that $\Phi$ is holomorphic and isometric 
	if and only if $T_{t}$ is, which is the case by Lemma \ref{nkwnkfenkwn}. 
\end{proof}

	Therefore, a parallel lattice $L\subset TM$ on a dually flat manifold $M$ induces a holomorphic 
	and isometric torus action $\mathbb{T}^{n}\times M_{L}\to M_{L}$.  

\section{Momentum map}\label{ndnkneknkdnksk}

	We start with some definitions. Let $G$ be a Lie group with Lie algebra $\textup{Lie}(G)=\mathfrak{g}$. Given $g\in G$, 
	we denote by $\textup{Ad}_{g}:\mathfrak{g}\to \mathfrak{g}$ and $\textup{Ad}_{g}^{*}:\mathfrak{g}^{*}\to 
	\mathfrak{g}^{*}$ the adjoint and coadjoint representations of $G$, respectively; they are related as follows: 
	\begin{eqnarray*}
		 \langle \textup{Ad}_{g}^{*}\alpha, \xi\rangle=\langle \alpha,\textup{Ad}_{g^{-1}}\xi\rangle, 
	\end{eqnarray*}
	where $\xi\in \mathfrak{g}$, $\alpha\in \mathfrak{g}^{*}$ (the dual of $\mathfrak{g}$) and $\langle\,,\,\rangle$ 
	is the natural pairing between $\mathfrak{g}$ and $\mathfrak{g}^{*}$. 

	Let $\Phi:G\times M\to M$ be a Lie group action of $G$ on a manifold $M$. The \textit{fundamental vector field} 
	associated to $\xi\in \mathfrak{g}$ is the vector field on $M$, denoted by $\xi_{M}$, defined by 
	\begin{eqnarray*}
		 (\xi_{M})(p):=\dfrac{d}{dt}\bigg\vert_{0}\Phi(c(t),p), 
	\end{eqnarray*}
	where $p\in M$ and $c(t)$ is a smooth curve in $G$ satisfying $c(0)=e$ (neutral element) and $\dot{c}(0)=\xi$. 
	
	Given a map $\moment:M\to \mathfrak{g}^{*}$ and a vector $\xi\in \mathfrak{g}$, we will denote by $\moment^{\xi}:
	M\to \mathbb{R}$ the function given by 
	\begin{eqnarray*}
		 \moment^{\xi}(p):=\langle \moment(p),\xi\rangle. 
	\end{eqnarray*}
	We shall say that $\moment$ is \textit{$G$-equivariant} if for every $g\in G$, 
	\begin{eqnarray*}
		 \moment \circ \Phi_{g}=\textup{Ad}_{g}^{*}\circ \moment, 
	\end{eqnarray*}
	where $\Phi_{g}:M\to M$, $p\mapsto \Phi(g,p)$. 

	Finally, given a symplectic form $\omega$ on $M$, the say that $\Phi$ is \textit{symplectic} 
	if $(\Phi_{g})^{*}\omega=\omega$ for all $g\in G$. 

\begin{definition}
	Let $(M,\omega)$ be a symplectic manifold. A symplectic action $\Phi:G\times M\to M$ is said to be 
	\textit{Hamiltonian} if there exists a $G$-equivariant map $\moment:M\to \mathfrak{g}^{*}$, called 
	\textit{momentum map}, such that 
	\begin{eqnarray*}
		 \omega(\xi_{M},\,.\,)=d\moment^{\xi}(.)
	\end{eqnarray*}
	(equality of 1-forms) for all $\xi\in \mathfrak{g}$. 
\end{definition}

	When $G=\mathbb{T}^{n}=\mathbb{R}^{n}/\mathbb{Z}^{n}$ is a torus, it is convenient to identify
	\begin{description}
	\item[$\bullet$] the Lie algebra of the torus $\mathbb{T}^{n}$ with $\mathbb{R}^{n}$ via the derivative at 
		$0\in \mathbb{R}^{n}$ of the quotient map $\mathbb{R}^{n}\to \mathbb{R}^{n}/\mathbb{Z}^{n}$, 
	\item[$\bullet$] $\mathbb{R}^{n}$ and its dual $(\mathbb{R}^{n})^{*}$ via the Euclidean metric.
	\end{description}
	Upon these identifications, a momentum map for a Hamiltonian torus action 
	$\mathbb{T}^{n}\times M\to M$ can be regarded as a map $\moment:M\to \mathbb{R}^{n}$. 
	Moreover, since the coadjoint action of a commutative group is trivial, 
	the equivariance condition reduces to $\moment \circ \Phi_{g}=\moment$ for all $g\in \mathbb{T}^{n}$. 

	If the symplectic manifold $M$ is connected, then it is easy to see that two momentum maps 
	$\moment_{1}$ and $\moment_{2}$ for the same Hamiltonian torus action $\Phi:\mathbb{T}^{n}\times M\to M$ 
	differ by a constant, that is $\moment_{1}=\moment_{2}+c$, $c\in \mathbb{R}^{n}$. Reciprocally, if 
	$\moment:M\to \mathbb{R}^{n}$ is a moment map for a Hamiltonian torus action, then so is $\moment+c$, where 
	$c\in\mathbb{R}^{n}$ is any constant.\\

\begin{proposition}\label{nkdnenfkkn}
	Let $(M,h,\nabla)$ be a connected dually flat manifold endowed with a parallel lattice $L=L_{X}\subset TM$ 
	with respect to $\nabla,$ where $X=(X_{1},...,X_{n})\in \textup{gen}(L)$, and let 
	$\Phi_{X}:\mathbb{T}^{n}\times M_{L}\to M_{L}$ be the corresponding torus action. 
	If $(x,y)$ is a global pair of dual coordinate systems on $M$, then $\Phi_{X}$ is Hamiltonian with 
	momentum map $\moment_{X}:M_{L}\to \mathbb{R}^{n}$ given by 
	\begin{eqnarray*}
		 \moment_{X}=-A\circ y\circ \pi_{L}, 
	\end{eqnarray*}
	where $A=(a_{ij})\in \textup{GL}(n,\mathbb{R})$ is the matrix defined via the formula
	\begin{eqnarray*}
		 X_{i}=\sum_{j=1}^{n}a_{ij}\dfrac{\partial}{\partial x_{j}}, \,\,\,\,\,\,\,i=1,...,n. 
	\end{eqnarray*}
\end{proposition}
\begin{proof}
	Let $\omega^{TM}$ and $\omega^{M_{L}}$ be the symplectic forms on $TM$ and $M_{L}$, respectively. 
	Let $T:\mathbb{R}^{n}\times TM\to TM$ be the Lie group action of $\mathbb{R}^{n}$ on $TM$ given by 
	\begin{eqnarray*}
		 T(t,u):=u+t_{1}X_{1}+...+t_{n}X_{n},
	\end{eqnarray*}
	where $t=(t_{1},...,t_{n})\in \mathbb{R}^{n}$ and $u\in TM$. We claim that $T$ is Hamiltonian 
	with momentum map $\moment:TM\to \mathbb{R}^{n}\cong(\mathbb{R}^{n})^{*}$ given by 
	\begin{eqnarray*}
		 \moment=-A\circ y\circ \pi,
	\end{eqnarray*}
	where $\pi:TM\to M$ is the canonical projection. To see this, let $(x_{1},...,x_{n})$ be an affine coordinate 
	system on $M$ with respect to $\nabla$, and let $(q,r)=(q_{1},...,q_{n},r_{1},...,r_{n})$ be the corresponding 
	coordinates on $\pi^{-1}(U)\subseteq TM$ as described before Lemma \ref{nfkewefnknknk}. Since 
	$X_{i}=\sum_{j=1}^{n}a_{ij}\tfrac{\partial}{\partial x_{j}}$ by hypothesis, the local expression for 
	$T$ is the coordinates $(q,r)$ is given by 
	\begin{eqnarray*}
		 T(t,(q,r))=(q,r+A^{T}t), 
	\end{eqnarray*}
	and so, the fundamental vector field of $\xi\in \mathbb{R}^{n}\cong T_{0}\mathbb{R}^{n}$ is given by 
	\begin{eqnarray*}
		(\xi_{TM})(q,r)=\dfrac{d}{dt}\bigg\vert_{0}T(t\xi, (q,r))=\dfrac{d}{dt}\bigg\vert_{0}(q,r+tA^{T}\xi)=(0,A^{T}\xi). 
	\end{eqnarray*}
	It follows from this together with the coordinate expression for $\omega^{TM}$ (see Proposition \ref{prop:4.1}) that 
	\begin{eqnarray}\label{fnkndksnkk}
		 \omega^{TM}(\xi_{TM},\,.\,)=(0,A^{T}\xi)
		\begin{bmatrix}
			0 & h\\
			-h & 0
		\end{bmatrix}
		=(-\xi^{T}Ah,0),
	\end{eqnarray}
	where $h=(h_{ij})$ is the coordinate expression for the metric $h$. On the other hand, the derivative of the 
	map $\moment^{\xi}:TM\to \mathbb{R}$, $u\mapsto -\langle\xi, (A\circ y\circ \pi)(u) \rangle $ 
	is given in matrix notation by 
	\begin{eqnarray}\label{fdnkefkwnk}
		 d\moment^{\xi}=\bigg(\dfrac{\partial \moment^{\xi}}{\partial q_{i}}, 
		\dfrac{\partial \moment^{\xi}}{\partial r_{i}}\bigg)=
		\bigg(-\xi^{T}A\bigg(\dfrac{\partial y_{i}}{\partial x_{j}}\bigg)_{ij},0\bigg)=(-\xi^{T}Ah,0), 
	\end{eqnarray}
	where we have used $\pi(q,r)=q$ and $\tfrac{\partial y_{i}}{\partial x_{j}}=h_{ij}$ (see Proposition \ref{neknkneknekenfkenk}). 
	Comparing \eqref{fnkndksnkk} and \eqref{fdnkefkwnk} we find that $\omega^{TM}(\xi_{TM},\,.\,)=d\moment^{\xi}$ 
	for all $\xi\in \mathbb{R}^{n}$. Obviously $\moment$ is equivariant. Therefore $\moment$ is a momentum map for 
	the action $T$. This concludes the proof of the claim. 

	Next we show that $\moment_{X}:M_{L}\to \mathbb{R}^{n}$ is a momentum map for the action $\Phi_{X}$. Let 
	$\xi\in \mathbb{R}^{n}\cong \textup{Lie}(\mathbb{R}^{n})\cong \textup{Lie}(\mathbb{T}^{n})$ be arbitrary. By inspection 
	of the definition of $\Phi_{X}$, we see that $q_{L}:TM\to M_{L}$ is equivariant, that is (with obvious notation):
	\begin{eqnarray*}
		 (\Phi_{X})_{[t]}\circ q_{L}=q_{L}\circ T_{t}
	\end{eqnarray*}
	for every $t\in \mathbb{R}^{n}$. Taking the derivative along the curve $t\xi$ at $t=0$ we find that 
	\begin{eqnarray*}
		 \xi_{M_{L}}\circ q_{L}=(q_{L})_{*}\circ \xi_{TM}.
	\end{eqnarray*}
	Because $q_{L}:TM\to M_{L}$ is holomorphic and locally isometric, we have that 
	$q_{L}^{*}\omega^{M_{L}}=\omega^{TM}$. Thus for any 
	$A\in TTM$, 
	\begin{eqnarray*}
		\omega^{M_{L}}(\xi_{M_{L}}\circ q_{L},(q_{L})_{*}A)&=& \omega^{M_{L}}((q_{L})_{*}\xi_{TM},(q_{L})_{*}A)\\
		&=& (q_{L}^{*}\omega^{M_{L}})(\xi_{TM},A)\\
		&=& \omega^{TM}(\xi_{TM},A)\\
		&=& d\moment^{\xi}(A)\\
		&=& d\moment_{X}^{\xi}\big( (q_{L})_{*}A\big), 
	\end{eqnarray*}
	where in the last equality we have used $\moment=\moment_{X}\circ q_{L}$. It follows that 
	$\omega^{M_{L}}(\xi_{M_{L}},\,.\,)=d\moment_{X}^{\xi}$ for every $\xi \in \textup{Lie}(\mathbb{T}^{n})$. 
	Obviously $\moment_{X}$ is equivariant. Therefore $\moment_{X}$ is a momentum map for 
	the action $\Phi_{X}$. This concludes the proof of the proposition.
\end{proof}

%
%

\begin{lemma}
	Let the hypotheses be as in Proposition \ref{nkdnenfkkn}. Suppose $X,X'\in \textup{gen}(L)$. 
	Let $B\in \textup{GL}(n,\mathbb{Z})$ be the unique matrix such that $X=BX'$, as in Lemma \ref{nkqndwknwfkn}. Then 
	\begin{eqnarray*}
		 \moment_{X}=B\circ \moment_{X'}. 
	\end{eqnarray*}
\end{lemma}
\begin{proof}
	By a direct verification.
\end{proof}

\section{Torification}\label{nfeknkwdnkk}

	Throughout this section,
	\begin{description}
	\item[$\bullet$] $(M,h,\nabla,\nabla^{*})$ is a connected dually flat manifold of dimension $n$ and
	\item[$\bullet$] $(N,g,J,\omega)$ is a connected K\"{a}hler manifold of real dimension $2n$, 
		equipped with an effective holomorphic and isometric torus action $\Phi:\mathbb{T}^{n}\times N\to N$.
	\end{description}


	We will denote by $N^{\circ}$ the set of points $p\in N$ where the action $\Phi$ is free, that is, 
	\begin{eqnarray*}
		N^{\circ} = \{p\in N\,\,\vert\,\, \Phi(a,p) = p\,\,\,\Rightarrow\,\,\,a = e\}. 
	\end{eqnarray*}
	Then $N^{\circ}$ is a $\mathbb{T}^{n}$-invariant connected open dense subset of $N$\footnote{This follows from the following result 
	(see \cite{Guillemin}, Corollary B.48.). Let $\Phi:G\times M\to M$ be a proper effective Lie group action of a commutative 
	Lie group $G$ on a connected manifold $M$. Then the set $M^{\circ}$ 
	of points where the action is free is open and dense in $M$. If in addition $M$ 
	is orientable and $G$ is connected, then $M^{\circ}$ is connected. }.\\

	Recall the notation $\Phi_{X}$ defined in \eqref{nndkfneknkw} for a generator $X$ of a parallel lattice $L\subset TM$.

\begin{lemma}\label{nfeknkwnkwnwknk}
	Let $L\subset TM$ be a parallel lattice with respect to $\nabla$, $U\subseteq N$ a $\mathbb{T}^{n}$-invariant set and $F:M_{L}\to U$ a map.
	The following are equivalent:
	\begin{enumerate}[(a)]
		\item There exists $X\in \textup{gen}(L)$ such that $F\circ (\Phi_{X})_{[t]}=\Phi_{[t]}\circ F$ for all $[t]\in \mathbb{T}^{n}$. 
		\item For every $X\in \textup{gen}(L)$, there exists $A\in \textup{GL}(n,\mathbb{Z})$ such that 
			$F\circ (\Phi_{X})_{[t]}=\Phi_{[At]}\circ F$ for all $t\in \mathbb{R}^{n}$.  
	\end{enumerate}
\end{lemma}
\begin{proof}
	Use Lemma \ref{nkendwkddknk}. 
\end{proof}
\begin{definition}
	Let $L\subset TM$ be a parallel lattice with respect to $\nabla$, $U\subseteq N$ a $\mathbb{T}^{n}$-invariant set and $F:M_{L}\to U$ a map.
	We shall say that $F$ is \textit{equivariant with respect to} $L$, or 
	simply $L$-\textit{equivariant}, if any of the conditions in Lemma \ref{nfeknkwnkwnwknk} holds. 
\end{definition}

	Now we can define the main concept of this paper. 

\begin{definition}\label{nksdnkfnksdn}
	We shall say that $N$ is a \textit{torification} 
	of $M$ if there exist a parallel lattice $L\subset TM$ with respect to $\nabla$
	and a $L$-equivariant holomorphic and isometric diffeomorphism $F:M_{L}\to N^{\circ }$. 
\end{definition}

	By abuse of language, we will often say that the torus action $\Phi:\mathbb{T}^{n}\times N\to N$ is a torification of $M$.

\begin{remark}
	\textbf{}
	\begin{enumerate}[(1)]
	\item If $N$ is a torification of $M$, then so does $N^{\circ}$. Therefore torifications are not unique in general.
	\item Let $A\in \textup{GL}(n,\mathbb{Z})$. If $\Phi:\mathbb{T}^{n}\times N\to N$ is a torification of $M$, then so does 
		$\widetilde{\Phi}:\mathbb{T}^{n}\times N\to N$, $([t],p)\mapsto\Phi([At],p)$. 
	\item If $L\subset TM$ is a parallel lattice with respect to $\nabla$ with generator $X$, then $\Phi_{X}:\mathbb{T}^{n}\times M_{L}\to M_{L}$ is
		trivially a torification of $M$. 	
	\end{enumerate}
\end{remark}

	Below is an alternative definition, in terms of covering maps. 

\begin{proposition}\label{nkwdnkkfenknk}
	Let $M$ and $N$ be as defined in the beginning of this section. The following are equivalent:   
	\begin{enumerate}[(1)]
	\item $N$ is a torification of $M$. 
	\item There exist a holomorphic and isometric covering map $\tau: TM\to N^{\circ}$, a parallel lattice $L\subset TM$ 
		with respect to $\nabla$ and $X=(X_{1},...,X_{n})\in \textup{gen}(L)$ such that:
			\begin{enumerate}[(i)]
			\item $\Gamma(L)=\textup{Deck}(\tau)$ (= Deck transformation group of $\tau$).
			\item $\tau\circ T_{t}	= \Phi_{[t]}\circ \tau$
				for every $t\in \mathbb{R}^{n}$, where $T:\mathbb{R}^{n}\times TM\to TM$ is the Lie group 
				action of $\mathbb{R}^{n}$ on $TM$ given by $T(t,u)=u+t_{1}X_{1}+...+t_{n}X_{n},$
				where $t=(t_{1},...,t_{n})\in \mathbb{R}^{n}$ and $u\in TM$. 
			\end{enumerate}
	\end{enumerate}
\end{proposition}
\begin{proof}[Sketch of proof]
	$(1)\Rightarrow (2)$. If $N$ is a torification of $M$, then there exist a parallel lattice $L\subset TM$ with respect to $\nabla$, 
	$X\in \textup{gen}(L)$ and a holomorphic and isometric diffeomorphism $F:M_{L}\to N^{\circ }$ such that $F\circ (\Phi_{X})_{a}=\Phi_{a}\circ F$ 
	for all $a\in \mathbb{T}^{n}$. Then it is easy to check that the map $\tau:TM\to N^{\circ}$ defined by $\tau(u)=(F\circ q_{L})(u)$ has 
	the required properties. 

	\noindent $(2)\Rightarrow (1)$. Let $\tau: TM\to N^{\circ}$, $L\subset TM$ and $X\in \textup{gen}(L)$ be as in the second item of the proposition. 
	Because $\tau$ is $\Gamma(L)$-invariant, it descends to a K\"{a}hler isomorphism $F:M_{L}\to N^{\circ}$, and a straightforward computation using 
	$\tau\circ T_{t}= \Phi_{[t]}\circ \tau$ shows that $F$ is equivariant in the sense that $F\circ (\Phi_{X})_{a}=\Phi_{a}\circ F$ for all $a\in \mathbb{T}^{n}$. 
	It follows that $N$ is a torification of $M$. 
%
\end{proof}

	Before proceeding, we introduce some terminology.

\begin{definition}
	Suppose $\Phi:\mathbb{T}^{n}\times N\to N$ is a torification of $(M,h,\nabla)$. 
	\begin{enumerate}[(1)]
	\item A \textit{toric parametrization} is a triple
		$(L,X,F)$, where 
		\begin{itemize}
			\item $L\subset TM$ is parallel lattice with respect to $\nabla$, generated by $X$, 
		\item $F:M_{L}\to N^{\circ}$ is a holomorphic and isometric diffeomorphism that is equivariant in the sense that
			$F\circ (\Phi_{X})_{a}=\Phi_{a}\circ F$ for all $a\in \mathbb{T}^{n}$.
		\end{itemize}
	\item  Let $\tau:TM\to N^{\circ}$ and $\kappa:N^{\circ}\to M$ be smooth maps. We say 
	that the pair $(\tau,\kappa)$ is a \textit{toric factorization} if there exists a toric parametrization $(L,X,F)$ that makes the following diagram 
	commutative: 
	\begin{eqnarray*} 
	\begin{tikzcd}
		TM \arrow[dd,bend right=70,swap,"\displaystyle\pi"] \arrow{rd}{\displaystyle \tau} \arrow[swap]{d}{\displaystyle q_{{L}}} & \\
		M_{{L}}   \arrow[swap]{d}{\displaystyle \pi_{{L}}} \arrow[swap]{r}{\displaystyle F} & N^{\circ} \arrow{ld}{\displaystyle \kappa} \\
		M  &   
	\end{tikzcd}.
	\end{eqnarray*}
	In this case, we say that $(\tau, \kappa)$ is \textit{induced by the toric parametrization} $(L,X,F)$. 
\item We say that $\kappa:N^{\circ}\to M$ is the \textit{compatible projection induced by the toric parametrization} $(L,X,F)$ 
	if there exists a map $\tau:TM\to N^{\circ}$ such that $(\tau,\kappa)$ is the toric factorization induced by $(L,X,F)$. 
	When it is not necessary to mention $(L,X,F)$ explicitly, we just say 
	that $\kappa$ is a \textit{compatible projection}. Analogously, one defines a \textit{compatible covering map} $\tau:TM\to N^{\circ}$. 

	\end{enumerate}
\end{definition}

	By abuse of language, we will often say that the formula $\pi=\kappa\circ \tau$ is a toric factorization. If $\pi=\kappa\circ \tau$ is a toric factorization, 
	then $\tau$ is a K\"{a}hler covering map whose Deck transformation group is $\Gamma(L)$, and $\kappa$ is naturally a principal $\mathbb{T}^{n}$-bundle and 
	a Riemannian submersion. 

\begin{proposition}\label{ndknqkefnkwnkndk}
	Let $(M,h,\nabla)$ and $(M',h',\nabla')$ be connected dually flat spaces and $f:M\to M'$ an 
	isomorphism of dually flat spaces. Suppose $\Phi:\mathbb{T}^{n}\times N\to N$ is a torification of $M$, with toric factorization $\pi=\kappa\circ \tau:TM\to M$. 
	Then $\Phi:\mathbb{T}^{n}\times N\to N$ is a torification of $M'$, with toric factorization $\pi'=\kappa'\circ \tau':TM'\to M'$, where $\tau'=\tau\circ (f_{*})^{-1}$ and 
	$\kappa'=f\circ \kappa.$ 
\end{proposition}
\begin{proof}[Sketch of proof]
	By Proposition \ref{nfeknwknekn}, $f_{*}:TM\to TM'$ is a K\"{a}hler isomorphism, and thus $\tau'=\tau\circ (f_{*})^{-1}:TM'\to N^{\circ}$ is a K\"{a}hler covering map. 
	Now apply Proposition \ref{nkwdnkkfenknk}. 
\end{proof}

	Now we focus our attention on torifications whose torus action is Hamiltonian. 

\begin{proposition}\label{unfenkwneknknk}
	Let $\Phi:\mathbb{T}^{n}\times N\to N$ be a torification of $(M,h,\nabla)$. Suppose that 
	$\Phi$ is Hamiltonian with momentum map $\moment:N\to \mathbb{R}^{n}$ and that $(x,y)$ are global pair of dual coordinate systems on $M$.
	\begin{enumerate}[(1)]
	\item Let $\kappa:N^{\circ}\to M$ be the compatible projection induced by a toric parametrization $(L,X,F)$. 
		Then there exists a constant $C\in \mathbb{R}^{n}$ such that on $N^{\circ}$,
		\begin{eqnarray*}
			\moment=-A\circ y\circ \kappa+C,
		\end{eqnarray*}
		where $A=(a_{ij})\in \textup{GL}(n,\mathbb{R})$ is the matrix defined via the formula $X_{i}=\sum_{j=1}^{n}a_{ij}\tfrac{\partial}{\partial x_{j}}$, $i=1,...,n$ 
		(here $X=(X_{1},...,X_{n})$ and $x=(x_{1},...,x_{n})$). In particular, $\moment(N^{\circ})$ is an open subset of $\mathbb{R}^{n}$. 
	\item If $x(M)=\mathbb{R}^{n},$ then $\moment(N^{\circ})$ is a convex subset of $\mathbb{R}^{n}$. 
		If in addition $\moment$ is proper, that is, if $\moment^{-1}(K)$ is compact whenever $K\subseteq \mathbb{R}^{n}$ is compact, 
		then $\moment(N)\subseteq \mathbb{R}^{n}$ is convex. 
	\item $\moment:N^{\circ}\to \moment(N^{\circ})$ is naturally a principal $\mathbb{T}^{n}$-bundle. In particular, there exists a unique Riemannian metric 
		$k$ on $\moment(N^{\circ})$ that makes $\moment:N^{\circ}\to \moment(N^{\circ})$ a Riemannian submersion. 
	\end{enumerate} 
\end{proposition}
\begin{proof} (1) By Proposition \ref{nkdnenfkkn}, $\Phi_{X}:\mathbb{T}^{n}\times M_{L}\to M_{L}$ 
	is Hamiltonian with momentum map $\moment':M_{L}\to \mathbb{R}^{n}$ given by $\moment'=-A\circ y\circ \pi_{L}$.
	Since $F:M_{L}\to N^{\circ}$ is an equivariant symplectomorphism, $\moment'\circ F^{-1}:N^{\circ}\to \mathbb{R}^{n}$ 
	is a momentum map with respect to the action $\Phi:\mathbb{T}^{n}\times N^{\circ}\to N^{\circ}$. 
	But then $\moment'\circ F^{-1}$ and $\moment$ are two momentum maps for the same torus action 
	on the connected set $N^{\circ}$. Therefore there is a constant $C\in \mathbb{R}^{n}$ such that 
	$\moment=\moment'\circ F^{-1}+C$ on $N^{\circ}$. The lemma follows from this and the fact that $\moment'\circ F^{-1}=-A\circ y\circ \pi_{L}\circ F^{-1}=
	-A\circ y\circ \kappa.$

	\noindent (2) By (1), there are an invertible matrix $A=(a_{ij})$ and a constant $C\in \mathbb{R}^{n}$ such that 
	$\moment(N^{\circ})=-A(y(M))+C$. Therefore $\moment(N^{\circ})$ is convex if and only if 
	$y(M)$ is convex. By Proposition \ref{neknkneknekenfkenk}, there is 
	a smooth function $\psi:M \to \mathbb{R}$ such that $y=\textup{grad}(\psi)=(\tfrac{\partial \psi}{\partial x_{1}},...,\tfrac{\partial\psi}{\partial x_{n}})$
	and $\tfrac{\partial^{2}\psi}{\partial x_{i}\partial x_{j}}=h_{ij}=h(\tfrac{\partial}{\partial x_{i}},\tfrac{\partial}{\partial x_{j}})$ for all $i,j=1,...,n$ on $M$. 
	Thus $y$ can be regarded as a strictly convex function on the convex set $x(M)=\mathbb{R}^{n}$. 
	By the general properties of the Legendre transform (see Theorem \ref{neknknfkfneknk}), 
	$y(\mathbb{R}^{n})$ is convex. If $\moment$ is proper, then a simple exercice in topology shows that 
	$\moment(N)=\overline{\moment(A)}$ for any dense subset $A$ of $N$, where $\overline{\moment(A)}$ denotes the closure of 
	$\moment(A)$ in $\mathbb{R}^{n}$. Since $N^{\circ}$ is dense in $N$, $\moment(N)=\overline{\moment(N^{\circ})}$. 
	It follows from this and the fact that the closure of a convex set is convex that $\moment(N)$ is convex. 

	\noindent (3) By (1), the following diagram is commutative:
	\begin{eqnarray*} 
	\begin{tikzcd}
		&   N^{\circ}  \arrow{rd}{\displaystyle \moment}\arrow[swap]{ld}{\displaystyle \kappa}  & \\
		M   \arrow[swap]{rr}{\displaystyle -A\circ y+C}  & &   \moment(N^{\circ})
	\end{tikzcd}
	\end{eqnarray*}
	Since $\kappa$ is a principal $\mathbb{T}^{n}$-bundle and $-A\circ y+C$ is a diffeomorphism, 
	$\moment:N^{\circ}\to \moment(N^{\circ})$ is a principal $\mathbb{T}^{n}$-bundle.
\end{proof}

\begin{lemma}\label{neekwnkvnkdnsk}
	Let $\mathcal{O}\subset N$ be the $\mathbb{T}^{n}$-orbit of $p\in N^{\circ}$. Then $T_{p}N=T_{p}\mathcal{O}\oplus J\big(T_{p}\mathcal{O}\big)$, 
	and the direct sum is orthogonal relative to the K\"{a}hler metric $g$ of $N$. 
\end{lemma}
\begin{proof}
	Because $\Phi$ is free at $p$, $\textup{dim}(\mathcal{O})=\textup{dim}(\mathbb{T}^{n})=n$ and hence 
	$\textup{dim}(T_{p}\mathcal{O})=\textup{dim}(JT_{p}\mathcal{O})=n$. Thus, to prove that $T_{p}N=T_{p}\mathcal{O}\oplus J\big(T_{p}\mathcal{O}\big)$, it suffices 
	to show that $T_{p}\mathcal{O}\cap J\big(T_{p}\mathcal{O}\big)$ is trivial (recall that $\textup{dim}(N)=2n$ by hypothesis). So let 
	$u=Jv\in T_{p}\mathcal{O}\cap J\big(T_{p}\mathcal{O}\big)$ be arbitrary, where $u,v\in T_{p}\mathcal{O}$. 
	Since $T_{p}\mathcal{O}=\{\xi_{N}(p)\,\,\vert\,\,\xi\in \textup{Lie}(\mathbb{T}^{n})\}$, where $\xi_{N}$ denotes the fundamental vector field on $N$ associated 
	to $\xi\in \textup{Lie}(\mathbb{T}^{n})$, there are 
	$\xi,\eta\in \mathbb{R}^{n}=\textup{Lie}(\mathbb{T}^{n})$ such that $u=\xi_{N}(p)$ and $v=\eta_{N}(p)$. From  this it follows that 
	\begin{eqnarray*}
		g_{p}(u,u)=g_{p}(Jv,u)=\omega_{p}(v,u)=\omega_{p}(\eta_{N},\xi_{N})=(d\moment^{\eta})_{p}(\xi_{N})=\dfrac{d}{dt}\bigg\vert_{0}\moment^{\eta}(\Phi(\textup{exp}(t\xi),p))=0,
	\end{eqnarray*}
	where we have used the following facts: (1) $\moment^{\eta}$ is constant along $\mathbb{T}^{n}$-orbits and (2) $(t,p)\mapsto 
	\Phi(\textup{exp}(t\xi),p)$ is the flow of $\xi_{N}$. Thus $u=0$. It follows that $T_{p}N=T_{p}\mathcal{O}\oplus J\big(T_{p}\mathcal{O}\big)$. 
	The computation above also shows that this direct sum is orthogonal. 
\end{proof}

\begin{proposition}
	Let the hypotheses be as in Proposition \ref{unfenkwneknknk}. 
	Let $k$ be the Riemannian metric on $\moment(N^{\circ})$ described in Proposition \ref{unfenkwneknknk}. 
	Given $p\in N^{\circ}$ and $\xi,\eta\in \textup{Lie}(\mathbb{T}^{n})$, we have 
	\begin{eqnarray}\label{nfekwnkenfkw}
		k_{\moment(p)}\big(\moment_{*_{p}}J\xi_{N},\moment_{*_{p}}J\eta_{N}\big)=g_{p}(\xi_{N},\eta_{N}). 
	\end{eqnarray}
\end{proposition}
\begin{proof}
	This follows from the following facts: (1) $\moment:N^{\circ}\to \moment(N^{\circ})$ is a Riemannian submersion, (2) the fibers of $\moment\vert_{N^{\circ}}$ 
	are $\mathbb{T}^{n}$-orbits and (3) the orthogonal complement of $T_{p}\mathcal{O}$ is $J(T_{p}\mathcal{O})$, by Lemma \ref{neekwnkvnkdnsk}. 
\end{proof}

	Let $\nabla^{\textup{flat}}$ denote the canonical flat connection on $\mathbb{R}^{n}$, or any open subset of it (see \eqref{nkwndkddnkwnneknfnk}). 

\begin{definition}\label{nkndknefkfnkwn}
	Let the hypotheses be as in Proposition \ref{unfenkwneknknk}. Let $\nabla^{k}$ be the dual connection of $\nabla^{\textup{flat}}$ with respect to 
	the Riemannian metric $k$ on $\moment(N^{\circ})$. We call $(k,\nabla^{k},\nabla^{\textup{flat}})$ the \textit{canonical dualistic structure} of 
	$\moment(N^{\circ})$ and $(\moment(N^{\circ}),k,\nabla^{k})$ the \textit{canonical dually flat space} associated to $\Phi:\mathbb{T}^{n}\times N\to N.$
\end{definition}

	The terminology is justified by the following result.

\begin{proposition}\label{nfekwnkefnkwnk}
	Let the hypotheses be as in Proposition \ref{unfenkwneknknk}. 
	Let $(k,\nabla^{k},\nabla^{\textup{flat}})$ be the canonical dualistic structure of $\moment(N^{\circ})$. The following hold. 
	\begin{enumerate}[(1)]
		\item $(k,\nabla^{k},\nabla^{\textup{flat}})$ is dually flat.
		\item $\Phi:\mathbb{T}^{n}\times N\to N$ is a torification of $(\moment(N^{\circ}),k,\nabla^{k})$. 
		\item Given a compatible projection $\kappa:N^{\circ}\to M$, there exists an isomorphism of dually flat spaces $f$ from $(M,h,\nabla)$ to 
			$(\moment(N^{\circ}),k,\nabla^{k})$ such that $f\circ\kappa=\moment$ on $N^{\circ}$. 
	\end{enumerate}
\end{proposition}
\begin{proof}
	Let $\kappa:N^{\circ}\to M$ be a compatible projection. By Proposition \ref{unfenkwneknknk}(1), 
	there exist a constant $C\in \mathbb{R}^{n}$ and an invertible matrix $A=(a_{ij})\in \textup{GL}(n,\mathbb{R})$ such that 
	$\moment=-A\circ y\circ \kappa+C$ on $N^{\circ}$. In particular, the map $f=-A\circ y+C$ is a diffeomorphism from 
	$M$ to $\moment(N^{\circ})$ satisfying $f\circ \kappa=\moment$ on $N^{\circ}$. We claim that $f$ is an isometry.  Indeed, since $\kappa$ is a surjective submersion, we have
	\begin{eqnarray*}
		f^{*}k=h \,\,\,\,\,\,\,\,\,\,\Leftrightarrow \,\,\,\,\,\,\,\,\,\,\kappa^{*}f^{*}k=\kappa^{*}h \,\,\,\,\,\,\,\,\,\,
		\Leftrightarrow& \,\,\,\,\,\,\,\,\,\,\moment^{*}k=\kappa^{*}h\,\,\,\textup{on}\,\,\moment(N^{\circ}).
	\end{eqnarray*}
	Therefore it suffices to show that $\moment^{*}k=\kappa^{*}h$ on $N^{\circ}$. Let $p\in N^{\circ}$ and $u,v\in T_{p}N$ be arbitrary. Let $\mathcal{O}\subset N$ denotes the 
	$\mathbb{T}^{n}$-orbit of $p$ and let $T_{p}\mathcal{O}^{\perp}=J\big(T_{p}\mathcal{O}\big)$ denotes the orthogonal complement of $T_{p}\mathcal{O}$ in $T_{p}N$ with respect to the 
	K\"{a}hler metric $g$. Write $u=u_{1}+u_{2}$ and $v=v_{1}+v_{2}$, where $u_{1},v_{1}\in T_{p}\mathcal{O}$ and $u_{2},v_{2}\in T_{p}\mathcal{O}^{\perp}$. Because 
	$\moment:N^{\circ}\to \moment(N^{\circ})$ is a Riemannian submersion with fiber $\moment^{-1}(\moment(p))=\mathcal{O}$, we have $\moment_{*_{p}}u_{1}=\moment_{*_{p}}v_{1}=0$ and 
	$k_{\moment(p)}(\moment_{*_{p}}u_{2},\moment_{*_{p}}v_{2})=g_{p}(u_{2},v_{2})$. Thus 
	\begin{eqnarray*}
		(\moment^{*}k)_{p}(u,v)= k_{\moment(p)}(\moment_{*_{p}}u_{2},\moment_{*_{p}}v_{2})= g_{p}(u_{2},v_{2}). 
	\end{eqnarray*}
	The exact same argument applied to $\kappa$ shows that $(\kappa^{*}h)_{p}(u,v)=g_{p}(u_{2},v_{2})$. It follows that $(\moment^{*}k)_{p}(u,v)=(\kappa^{*}h)_{p}(u,v)$
	and concludes the proof of the claim. Let $\nabla^{*}$ be the dual connection of $\nabla$ with respect to $\nabla.$
	The map $f$ is trivially affine from $(M,\nabla^{*})$ to $(\moment(N^{\circ}),\nabla^{\textup{flat}})$ (since $f$ is an affine function 
	of $y$ and $y$ is a $\nabla^{*}$-affine coordinate system on $M$). By Lemma \ref{ncekndwkneknku} and the claim, $f$ is also affine 
	from $(M,\nabla)$ to $(\moment(N^{\circ}),\nabla^{k})$. This forces $\nabla^{k}$ to be flat. It follows that 
	$(k,\nabla^{k},\nabla^{\textup{flat}})$ is dually flat and that $f$ is an isomorphism of dually flat spaces. 
	This shows (1) and (3). (2) is a consequence of $(3)$ and Proposition \ref{ndknqkefnkwnkndk}. 
\end{proof}

\begin{proposition}\label{nenwknkfwnknk}
	Given $i=1,2,$ let $(M_{i},h_{i},\nabla_{i})$ be a connected dually flat space that has a global pair of dual coordinate systems. 
	Suppose that $\Phi_{i}:\mathbb{T}^{n}\times N_{i}\to N_{i}$ is a torification of $(M_{i},h_{i},\nabla_{i})$, $i=1,2$, and that 
	$\Phi_{i}$ is Hamiltonian. If there exists a K\"{a}hler isomorphism $G:N_{1}\to N_{2}$ and a Lie group isomorphism $\rho:\mathbb{T}^{n}\to \mathbb{T}^{n}$ 
	such that $G\circ (\Phi_{1})_{a}=(\Phi_{2})_{\rho(a)}\circ G$ for all $a\in \mathbb{T}^{n}$, then 
	there exists an isomorphism of dually flat spaces between $(M_{1},h_{1},\nabla_{1})$ and $(M_{2},h_{2},\nabla_{2})$. 
\end{proposition}
\begin{proof}
	Let $\moment_{1}:N_{1}\to \mathbb{R}^{n}$ be a momentum map and $L=\rho_{*_{e}}:\mathbb{R}^{n}\to \mathbb{R}^{n}=\textup{Lie}(\mathbb{T}^{n})$. 
	Because  $G$ is equivariant relative to $\rho$, the fundamental vector fields of $\xi\in \textup{Lie}(\mathbb{T}^{n})$ and $L(\xi)$ are related by 
	\begin{eqnarray}\label{nekdnwknknkn}
		G_{*_{p}}\xi_{N_{1}}(p)=(L(\xi))_{N_{2}}(G(p))
	\end{eqnarray}
	for all $p\in N_{1}$. From this, a straightforward verification shows that 
	$\moment_{2}=(L^{-1})^{*}\circ \moment_{1}\circ G^{-1}$ is a momentum map with respect to $\Phi_{2}$, where $(L^{-1})^{*}$ denotes the adjoint of $L^{-1}$ 
	with respect to the Euclidean scalar product (thus $\langle L^{-1}(x),y\rangle=\langle x,(L^{-1})^{*}(y)\rangle$ for 
	all $x,y\in \mathbb{R}^{n}$). Let $(\moment_{1}(N_{1}^{\circ}),k_{1},\nabla^{k_{1}})$ and $(\moment_{2}(N_{2}^{\circ}),k_{2},\nabla^{k_{2}})$ be the corresponding 
	canonical dually flat manifolds (see Definition \ref{nkndknefkfnkwn}). Given $p\in N_{i}$, $i=1,2,$ the K\"{a}hler metric $g_{i}$ induces an orthogonal decomposition 
	$T_{p}N_{i}=T_{p}\mathcal{O}_{i}\oplus (T_{p}\mathcal{O}_{i})^{\perp}$, where $\mathcal{O}_{i}\subset N_{i}$ is the $\mathbb{T}^{n}${-}orbit of $p$. 
	We shall denote by $u^{\perp}\in (T_{p}\mathcal{O}_{i})^{\perp}$ the orthogonal projection of $u\in T_{p}N_{i}$ on 
	$(T_{p}\mathcal{O}_{i})^{\perp}$. We claim that:
	\begin{enumerate}[(a)]
		\item $(k_{i})_{\moment_{i}(p)}((\moment_{i})_{*_{p}}u,(\moment_{i})_{*_{p}}v)=(g_{i})_{p}(u^{\perp},v^{\perp})$ for all $p\in N_{i}^{\circ}$ and all $u,v\in T_{p}N_{i}$.
		\item $\big(G_{*_{p}}(u)\big)^{\perp}=G_{*_{p}}(u^{\perp})$ for all $p\in N_{1}^{\circ}$ and all $u\in T_{p}N_{1}$.  
	\end{enumerate}
	The fist item is a consequence of the fact that $\moment_{i}:N_{i}^{\circ}\to \moment_{i}(N_{i}^{\circ})$ is a Riemannian submersion whose fibers are $\mathbb{T}^{n}$-orbits. To see (b), 
	let $u\in T_{p}N_{1}^{\circ}$ be arbitrary. Since the tangent space of an orbit at a given point is spanned by the fundamental vector fields at that point, 
	there exists $\xi\in \textup{Lie}(\mathbb{T}^{n})$ such that $u=\xi_{N_{1}}(p)+u^{\perp}$. It follows from this and \eqref{nekdnwknknkn} that 
	\begin{eqnarray*}
		G_{*_{p}}(u)=(L(\xi))_{N_{2}}(G(p))+G_{*_{p}}(u^{\perp}).
	\end{eqnarray*}
	Projecting onto $\big(T_{G(p)}\mathcal{O}\big)^{\perp}$, where $\mathcal{O}$ is the $\mathbb{T}^{n}$-orbit of $G(p)$, 
	we get $(G_{*_{p}}(u))^{\perp}=\big(G_{*_{p}}(u^{\perp})\big)^{\perp}$. Thus, it suffices to show that 
	$\big(G_{*_{p}}(u^{\perp})\big)^{\perp}=G_{*_{p}}(u^{\perp})$, or equivalently, 
	that $G_{*_{p}}(u^{\perp})$ is orthogonal to $T_{G(p)}\mathcal{O}$. Since $\rho:\mathbb{T}^{n}\to \mathbb{T}^{n}$ 
	is a Lie group isomorphism, $L=\rho_{*_{e}}$ is a linear bijection and hence $T_{G(p)}\mathcal{O}$ is spanned by elements of the form $(L(\xi))_{N_{2}}(G(p))$, 
	where $\xi\in \mathbb{R}^{n}=\textup{Lie}(\mathbb{T}^{n})$. Given $\xi\in \textup{Lie}(\mathbb{T}^{n})$, we compute:
	\begin{eqnarray*}
		\lefteqn{(g_{2})_{G(p)}((L(\xi))_{N_{2}}(G(p)),G_{*_{p}}(u^{\perp}))}\\
		&=&(g_{2})_{G(p)}(G_{*_{p}}\xi_{N_{1}}(G(p)),G_{*_{p}}(u^{\perp}))\\
		&=&  (G^{*}g_{2})_{p}(\xi_{N_{1}},u^{\perp}) =(g_{1})_{p}(\xi_{N_{1}},u^{\perp})=0,
	\end{eqnarray*}
	where we have used \eqref{nekdnwknknkn}. This shows that $G_{*_{p}}(u^{\perp})\in \big(T_{G(p)}\mathcal{O}\big)^{\perp}$ and concludes the proof of the claim.

	Next we prove that $f=(L^{-1})^{*}:\moment_{1}(N_{1}^{\circ})\to \moment_{2}(N_{2}^{\circ})$ is an isomorphism of dually flat spaces. The injectivity of $f$ is 
	a consequence of the injectivity of $(L^{-1})^{*}$. The fact that $f$ is surjective is a consequence of the formula 
	$\moment_{2}=(L^{-1})^{*}\circ \moment_{1}\circ G^{-1}$. Thus $f$ is a bijection from $\moment_{1}(N_{1}^{\circ})$ to $\moment_{2}(N_{2}^{\circ})$. 
	Let $p\in N_{1}^{\circ}$ and $u,v\in T_{p}N_{1}$. We compute:
	\begin{eqnarray*}
		\lefteqn{(f^{*}k_{2})_{\moment_{1}(p)}\big((\moment_{1})_{*_{p}}(u), (\moment_{1})_{*_{p}}(v)\big)}\\
		&=& (k_{2})_{(L^{-1})^{*}(\moment_{1}(p))}\big((L^{-1})^{*}(\moment_{1})_{*_{p}}(u), (L^{-1})^{*}(\moment_{1})_{*_{p}}(v)\big)\\
		&=& (k_{2})_{(\moment_{2}\circ G)(p)}\big((\moment_{2}\circ G)_{*_{p}}(u), (\moment_{2}\circ G)_{*_{p}}(v)\big)\\
		&=& (g_{2})_{G(p)}\big(G_{*_{p}}(u^{\perp}), G_{*_{p}}(v^{\perp})\big) \,\,\,\,\,\,\,(\textup{see (a) and (b)})\\ 
		&=& (G^{*}g_{2})_{p}(u^{\perp},v^{\perp})=(g_{1})_{p}(u^{\perp},v^{\perp})\\
		&=& (k_{1})_{\moment_{1}(p)}((\moment_{1})_{*_{p}}(u),(\moment_{1})_{*_{p}}(v)) \,\,\,\,\,\,\,(\textup{see (a)})
	\end{eqnarray*}
	where we have used $\moment_{2}\circ G=(L^{-1})^{*}\circ \moment_{1}$. 
	This shows that $f^{*}k_{2}=k_{1}$, that is, $f$ is an isometry. Clearly, $f$ is affine from $(\moment(N_{1}^{\circ}),\nabla^{\textup{flat}})$ to 
	$(\moment(N_{2}^{\circ}),\nabla^{\textup{flat}})$, since it is the restriction of a linear map. 
	By Lemma \ref{ncekndwkneknku}, it is also affine from $(\moment(N_{1}^{\circ}),\nabla^{k_{1}})$ to $(\moment(N_{2}^{\circ}),\nabla^{k_{2}})$. 
	It follows that $f$ is an isomorphism of dually flat spaces. 

	By Proposition \ref{nfekwnkefnkwnk}, there are isomorphisms of dually flat spaces $\phi_{1}:M_{1}\to \moment(N_{1}^{\circ})$ and 
	$\phi_{2}:M_{2}\to \moment(N_{2}^{\circ})$. Therefore $\phi_{2}^{-1}\circ f\circ \varphi_{1}$ is an isomorphism of dually flat spaces between 
	$(M_{1},h_{1},\nabla_{1})$ and $(M_{2},h_{2},\nabla_{2})$.  
\end{proof}

\section{The canonical example (complex point of view)}\label{nekdwknknwkdnknkn}


	In this section, we re-examine some well-known results of symplectic toric geometry (complex and symplectic point of views, Legendre transform) 
	by using systematically the point of view of torification and the language of dually flat spaces. 

	We continue to use the notation $x=(x_{1},...,x_{n})$ to denote the standard coordinates on $\mathbb{R}^{n}$. The flat connection on $\mathbb{R}^{n}$ 
	is denoted by $\nabla^{\textup{flat}}$ (see \eqref{nkwndkddnkwnneknfnk}).

	The theorem below is the main result of this section. 
\begin{theorem}\label{nckdnwknknednenkenk}
	Let $(N,g,J)$ be a connected K\"{a}hler manifold of real dimension $2n$ and let 
	$\Phi:\mathbb{T}^{n}\times N\to N$ be an effective, Hamiltonian and holomorphic torus action. 
	Suppose that for every $\xi\in \textup{Lie}(\mathbb{T}^{n})$, the vector field $-J\xi_{N}$ 
	is complete, where $\xi_{N}$ is the fundamental vector field of $\xi$; let $\varphi^{\xi}:\mathbb{R}\times N\to N$ be the corresponding flow, that is, 
	$\varphi^{\xi}(0,q)=q$ and $\tfrac{d}{dt}\varphi^{\xi}(t,q)=-J\xi_{N}(\varphi^{\xi}(t,q))$ for all 
	$(t,q)\in \mathbb{R}\times N$. Fix $p\in N^{\circ}$ and consider the correspondence $h$ that associates to $x\in \mathbb{R}^{n}$ 
	the bilinear form $h_{x}$ on $T_{x}\mathbb{R}^{n}=\mathbb{R}^{n}$ defined by 
	\begin{eqnarray*}	
		h_{x}(u,v)=g_{\varphi^{x}(1,p)}(u_{N},v_{N}),  \,\,\,\,\,\,\,\,\,\,(x,u,v\in \mathbb{R}^{n}=\textup{Lie}(\mathbb{T}^{n})).
	\end{eqnarray*}
	Then,
	\begin{enumerate}[(1)]
		\item $(\mathbb{R}^{n},h,\nabla^{\textup{flat}})$ is a dually flat manifold.
		\item $\Phi:\mathbb{T}^{n}\times N\to N$ is a torification of $(\mathbb{R}^{n},h,\nabla^{\textup{flat}})$. 
		\item Given a momentum map $\moment:N\to \mathbb{R}^{n}$, there are coordinates $y=(y_{1},...,y_{n})$ on $\mathbb{R}^{n}$ such that 
			\begin{enumerate}[(a)]
				\item $(x,y)$ is a global pair of dual coordinate systems on $(\mathbb{R}^{n},h,\nabla^{\textup{flat}})$
					(in particular, $y$ is affine with respect to the dual connection of $\nabla^{\textup{flat}}$ with respect to $h$).
				\item $-y$ is an isomorphism of dually flat spaces between $(\mathbb{R}^{n},h,\nabla^{\textup{flat}})$ and the 
					canonical dually flat space $(\moment(N^{\circ}),k,\nabla^{k})$. 
			\end{enumerate}
			If in addition there is no $C\in \mathbb{R}^{n}-\{0\}$ such that $\moment(N^{\circ})=\moment(N^{\circ})+C$ (for example if $N$ is compact), 
			then $y$ satisfying (a) and (b) is unique. 
	\end{enumerate} 
	\end{theorem}

	The proof of Theorem \ref{nckdnwknknednenkenk} proceeds in a series of lemmas. We begin with a discussion of the 
	holomorphic extension of a torus action. This is a standard construction, which, to our knowledge, was first discribed in \cite{Guillemin-Geometric}. 
	Our presentation follows closely \cite{Ishida}. 

	Let $\Phi:\mathbb{T}^{n}\times N\to N$ be a Lie group action of the torus $\mathbb{T}^{n}=\mathbb{R}^{n}/\mathbb{Z}^{n}$ 
	on a complex manifold $N$ by holomorphic diffeomorphisms. We will identify the Lie algebra of the torus $\mathbb{T}^{n}$ with 
	$\mathbb{R}^{n}$ via the derivative at $0\in \mathbb{R}^{n}$ of the quotient map $\mathbb{R}^{n}\to \mathbb{R}^{n}/\mathbb{Z}^{n}$, $t\mapsto [t]$. 
	Under this identification, the exponential map $\textup{exp}:\textup{Lie}(\mathbb{T}^{n})\to \mathbb{T}^{n}$ is just the quotient map 
	$\mathbb{R}^{n}\to \mathbb{R}^{n}/\mathbb{Z}^{n}.$ Given $\xi\in \mathbb{R}^{n}=\textup{Lie}(\mathbb{T}^{n})$, 
	we will denote by $\xi_{N}\in \mathfrak{X}(N)$ the corresponding fundamental vector field on $N$. 
	We will assume that $J\xi_{N}$ is complete for all $\xi\in \textup{Lie}(\mathbb{T}^{n})$, where $J:TN\to TN$ is the complex structure of $N$.  

	As a matter of notation, we will denote by $\varphi^{X}_{t}$ the flow of a vector field $X$. Thus $\tfrac{d}{dt}\varphi^{X}_{t}(p)=
	X(\varphi^{X}_{t}(p))$. For later use, recall that if $X$ and $Y$ are complete commuting vector fields, then 
	$\varphi_{t}^{X}=\varphi_{1}^{tX}$ and $\varphi_{t}^{X}\circ \varphi_{t}^{Y}=\varphi^{X+Y}_{t}$ for all $t\in \mathbb{R}$ (in particular, $X+Y$ is complete). 
	
	Let $\{e_{1},...,e_{k}\}$ be the canonical basis of $\mathbb{R}^{n}=\textup{Lie}(\mathbb{T}^{n})$. Define $T:\mathbb{C}^{n}\times N\to N$ by 
	\begin{eqnarray}\label{nnksnkenksn}
		T(z,p)=\big(\varphi^{-J(e_{1})_{N}}_{x_{1}}\circ \cdots \circ \varphi_{x_{n}}^{-J(e_{n})_{N}}\circ \varphi_{y_{1}}^{(e_{1})_{N}}
		\circ \cdots \circ \varphi_{y_{n}}^{(e_{n})_{N}}\big)(p),
	\end{eqnarray}
	where $z=(z_{1},...,z_{n})\in \mathbb{C}^{n}$, $z_{k}=x_{k}+iy_{k}$, $x_{k},y_{k}\in \mathbb{R}$, $k=1,...,n$ and $p\in N.$

\begin{lemma}\label{nefknwdknefknk}
	\textbf{}
	\begin{enumerate}[(1)]
	\item $T$ is a Lie group action of $\mathbb{C}^{n}$ on $N$. 
	\item $T(tz,p)=\varphi_{t}^{-Jx_{N}+y_{N}}(p)=\varphi^{-Jx_{N}}_{t}(\Phi([y],p))$ for all $t\in \mathbb{R}$, $z=x+iy\in \mathbb{R}^{n}+i\mathbb{R}^{n}$, $p\in N$. 
	\item The derivative of $T$ at $(z,p)\in \mathbb{C}^{n}\times N$ in the direction $(w,u)\in \mathbb{C}^{n}\times T_{p}N$ is given by
		\begin{eqnarray*}
			T_{*_{(z,p)}}(w,u)= (-Ja_{N}+b_{N})\big(T(z,p)\big)+(T_{z})_{*_{p}}u, 
		\end{eqnarray*}
			where $w=a+ib$, $a,b\in \mathbb{R}^{n}=\textup{Lie}(\mathbb{T}^{n})$ and $T_{z}:N\to N$, $p\mapsto T(z,p)$. 
	\item The map $T:\mathbb{C}^{n}\times N\to N$ is holomorphic. 
	\end{enumerate}
\end{lemma}
\begin{proof}
	(1) Obviously, $T(0,p)=p$ for every $p\in N$. To prove the formula $T(z,T(z',p))=T(z+z',p)$, 
	it suffices to show that all the flows in \eqref{nnksnkenksn} commute, which is equivalent to show that for any pair 
	$\xi,\eta\in \mathbb{R}^{n}=\textup{Lie}(\mathbb{T}^{n})$, 
	\begin{eqnarray}\label{nfkwndkdnk}
		[\xi_{N},\eta_{N}]=[J\xi_{N},\eta_{N}]=[J\xi_{N},J\eta_{N}]=0.
	\end{eqnarray}	
	The vanishing of $[\xi_{N},\eta_{N}]$ and $[J\xi_{N},\eta_{N}]$ is a consequence of Lemma \ref{fneknkenkfnk} (see below) together with the fact that 
	$\mathbb{R}^{n}\to \mathfrak{X}(N)$, $\xi\mapsto \xi_{N}$, is an antihomomorphism of commutative Lie algebras. The vanishing of $[J\xi_{N},J\eta_{N}]$ 
	follows from the vanishing of the other two Lie brackets together with the vanishing of the Nijenhuis tensor:
	\begin{eqnarray*}
		0=[\xi_{N},\eta_{N}]+J[J\xi_{N},\eta_{N}]+J[\xi_{N},J\eta_{N}]-[J\xi_{N},J\eta_{N}]=-[J\xi_{N},J\eta_{N}]. 
	\end{eqnarray*}
	It follows that $T$ is a well defined Lie group action. 

	(2) This follows easily from the fact that all the flows in \eqref{nnksnkenksn} commute (see above). 

	(3) By a direct calculation. 

	(4) First we show that for every $s\in \mathbb{R}$ and every $k=1,...,n$, the diffeomorphism $\varphi^{-J(e_{k})_{N}}_{s}:N\to N$ is holomorphic. 
	In view of Lemma \ref{fneknkenkfnk}, 
	this is equivalent to show that $[J(e_{k})_{N},JY]=J[J(e_{k})_{N},Y]$ for all vector fields $Y$ on $N$. So let $Y$ be an arbitrary vector field on $N$. 
	Because the flow of $(e_{k})_{N}$ consists of holomorphic transformations, $[(e_{k})_{N},JY]=J[(e_{k})_{N},Y]$. 
	It follows from this and the vanishing of the Nijenhuis tensor that 
	\begin{eqnarray*}
		0&=& [J(e_{k})_{N},JY]-J[(e_{k})_{N},JY]-J[J(e_{k})_{N},Y]-[(e_{k})_{N},Y]\\
		&=&[J(e_{k})_{N},JY]+[(e_{k})_{N},Y]-J[J(e_{k})_{N},Y]-[(e_{k})_{N},Y]\\
		&=&[J(e_{k})_{N},JY]-J[J(e_{k})_{N},Y].  
	\end{eqnarray*}
	Thus $\varphi^{-J(e_{k})_{N}}_{s}$ is holomorphic. Now, a straightforward computation using (3) and the fact that the flows 
	$\varphi_{x_{k}}^{-J(e_{k})_{N}}$ and $\varphi_{y_{k}}^{(e_{k})_{N}}$ are all holomorphic transformations of $N$, shows that 
	$T_{*(z,p)}(iw,Ju)=JT_{*_{(z,p)}}(w,u)$, which means that $T$ is holomorphic.  
%
%
\end{proof}

\begin{lemma}\label{fneknkenkfnk}
	Let $X$ be a vector field on a complex manifold $W.$ The following are equivalent.
	\begin{enumerate}[(1)]
		\item The flow of $X$ consists of holomorphic transformations of $W$. 
		\item $[X,JY]=J[X,Y]$ for all vector field $Y$ on $W$. 
	\end{enumerate}
\end{lemma}
\begin{proof}
	See \cite{moroianu_2007}, Lemma 2.7. 
\end{proof}

	Let $(\mathbb{C}^{*})^{n}$ be the algebraic torus and let $\varepsilon:\mathbb{C}^{n}\to (\mathbb{C}^{*})^{n}$ be the map defined by 
	\begin{eqnarray*}
		\varepsilon(z)=(e^{2\pi z_{1}},...,e^{2\pi z_{n}}), 
	\end{eqnarray*}
	where $z=(z_{1},...,z_{n})\in \mathbb{C}^{n}$. Then $\varepsilon$ is a surjective Lie group homomorphism and a holomorphic covering map 
	with Deck transformation group $i\mathbb{Z}^{n}$. If $\varepsilon(z)=\varepsilon(w)$ for some $z,w\in \mathbb{C}^{n}$, 
	then there exists $k\in \mathbb{Z}^{n}$ such that $z=w+ik$, and so 
	\begin{eqnarray*}
		T(z,p)=T(w+ik,p)=T(w,T(ik,p))=T\big(w,\varphi_{1}^{-J(0)_{N}}(\Phi([k],p))\big)=T(w,p),
	\end{eqnarray*}
	where we have used Proposition \ref{nefknwdknefknk}(2). Therefore there exists a unique map $\Phi^{\mathbb{C}}:(\mathbb{C}^{*})^{n}\times N\to N$ such that 
	\begin{eqnarray*}
		\Phi^{\mathbb{C}}(\varepsilon(z),p)=T(z,p)
	\end{eqnarray*}
	for all $z\in \mathbb{C}^{n}$ and $p\in N$. 
	It follows from the properties of $\varepsilon$ and $T$ that $\Phi^{\mathbb{C}}:(\mathbb{C}^{*})^{n}\times N\to N$ is a Lie group action and a holomorphic map.
	Moreover, $\Phi^{\mathbb{C}}((e^{2i\pi t_{1}},...,e^{2i\pi t_{n}}),p)=\Phi^{\mathbb{C}}(\varepsilon(it),p)=\Phi([t],p)$ for all $t=(t_{1},...,t_{n})\in \mathbb{R}^{n}$ and all $p\in N.$
	Therefore $\Phi^{\mathbb{C}}$ is an extension of $\Phi$, provided $\mathbb{T}^{n}$ is identified with 
	$\{(e^{i2\pi t_{1}},...,e^{i2\pi t_{n}})\in \mathbb{C}^{n}\,\,\vert\,\,t_{1},...,t_{n}\in \mathbb{R}\}\subset \mathbb{C}^{n}$. \\

	We now specialize to the case when $N$ is a connected K\"{a}hler manifold of real dimension $2n$ and 
	$\Phi:\mathbb{T}^{n}\times N\to N$ is effective, Hamiltonian and holomorphic. We continue to use the notation 
	$\Phi^{\mathbb{C}}$ to denote the holomorphic extension of $\Phi$. 

	Let $N^{\circ}$ be the dense connected open subset of $N$ of points $p\in N$ where the action $\Phi$ is free. 

\begin{lemma}\label{nfekwndknfeknkdnk}
	The following hold. 
	\begin{enumerate}[(1)]
	\item The action $\Phi$ is free at $p\in N$ if and only if $\Phi^{\mathbb{C}}$ is free at $p$. 
	\item $N^{\circ}$ is a $(\mathbb{C}^{*})^{n}$-orbit. 
	\item Given $p\in N^{\circ}$, the orbit map $\Phi_{p}^{\mathbb{C}}:(\mathbb{C}^{*})^{n}\to N^{\circ}$ is a holomorphic diffeomorphism, which 
		is equivariant in the sense that $\Phi_{p}^{\mathbb{C}}(zw)=(\Phi^{\mathbb{C}}_{z}\circ \Phi_{p}^{\mathbb{C}})(w)$ 
			for all $z,w\in (\mathbb{C}^{*})^{n}$. 
	\end{enumerate}
\end{lemma}
\begin{proof}
	(1) Let $p\in N$ be arbitrary. Let $K_{1}\subseteq \mathbb{T}^{n}$ and $K_{2}\subseteq (\mathbb{C}^{*})^{n}$ be the corresponding 
	stabilizers. We must show that $K_{1}$ is trivial if and only if $K_{2}$ is trivial. Suppose $K_{2}$ trivial. 
	Because $\Phi^{\mathbb{C}}$ is an extension of $\Phi$, we have 
	$K_{1}=\mathbb{T}^{n}\cap K_{2}\subset K_{2}=\{e\}$ and hence $K_{1}$ is trivial. Conversely, suppose $K_{1}$ trivial. Let 
	$\varepsilon(z)=\varepsilon(x+iy)\in K_{2}$ be arbitrary, where $x,y\in \mathbb{R}^{n}$. 
	Consider the curve $\alpha:\mathbb{R}\to N,$ $t\mapsto \Phi^{\mathbb{C}}(\varepsilon(tx),p)=T(tx,p)$. Note that $\alpha(0)$ and 
	$\alpha(1)$ belong to the same $\mathbb{T}^{n}$-orbit. Indeed, 
	\begin{eqnarray*}	
		\alpha(0)=p=\Phi_{\varepsilon(z)}^{\mathbb{C}}(p)=T(z,p)=T(iy,T(x,p))=\Phi_{[y]}(\alpha(1)), 
	\end{eqnarray*}
	where we have used Lemma \ref{nefknwdknefknk}(2). By Lemma \ref{nefknwdknefknk}(2), $\alpha$ is an integral curve of $-Jx_{N}$. 
	Let $\omega$ be the symplectic form on $N$ and let $\moment:N\to \mathbb{R}^{n}$ be a momentum map for the action $\Phi$. We denote 
	by $\moment^{x}:N\to \mathbb{R}$ the map defined by $\moment^{x}(q)=\langle\moment(q),x\rangle$, where $\langle\,,\,\rangle$ is the 
	Euclidean pairing on $\mathbb{R}^{n}$.  By definition of the momentum map, 
	the derivative of $\moment^{x}$ along $\alpha$ is given by 
	\begin{eqnarray*}
		\dfrac{d}{dt}\moment^{x}\big(\alpha(t)\big)=\omega\big(x_{N},\dot{\alpha}(t)\big)=-\omega_{\alpha(t)}(x_{N},Jx_{N})=
		-g_{\alpha(t)}(x_{N},x_{N})\leq 0.
	\end{eqnarray*}
	Thus $\moment^{x}$ is nonincreasing along $\alpha$. 
	Since $\moment^{x}$ is $\mathbb{T}^{n}$-invariant and $\alpha(0),\alpha(1)$ belong to the same $\mathbb{T}^{n}$-orbit, 
	it follows that $\moment^{x}\circ \alpha$ is constant on $[0,1]$. In particular, $0=\tfrac{d}{dt}\big\vert_{0}\moment^{x}(\alpha(t))=-g_{p}(x_{N},x_{N})$ and hence
	$x_{N}(p)=0$. Therefore $x\in \textup{Lie}(K_{1})=\{\xi\in \textup{Lie}(\mathbb{T}^{n})\,\,\vert\,\,\xi_{N}(p)=0\}$. Since $K_{1}$ is trivial, 
	$x=0$. It follows that $\varepsilon(z)=\varepsilon(iy)=(e^{2i\pi y_{1}},...,e^{2i\pi y_{n}})\in \mathbb{T}^{n}$. 
	Thus $\varepsilon(z)\in \mathbb{T}^{n}\cap K_{2}=K_{1}=\{e\}$. It follows that $K_{2}$ is trivial and concludes the proof of (1). 

	(2) Let $\mathcal{O}^{\mathbb{C}}$ be the $(\mathbb{C}^{*})^{n}$-orbit of an arbitrary point $p\in N^{\circ}$. By (1), 
	$\Phi^{\mathbb{C}}$ is free at $p$. Since $(\mathbb{C}^{*})^{n}$ is commutative, the stabilizer subgroups associated to $\Phi^{\mathbb{C}}$ 
	are constant along $(\mathbb{C}^{*})^{n}$-orbits. Thus $\Phi^{\mathbb{C}}$ is free at every point $q\in \mathcal{O}^{\mathbb{C}}$. By (1), this implies 
	$\mathcal{O}^{\mathbb{C}}\subseteq N^{\circ}$. To prove that $\mathcal{O}^{\mathbb{C}}=N^{\circ}$, it suffices to show that $\mathcal{O}^{\mathbb{C}}$ is open 
	in $N^{\circ}$, because then $N^{\circ}$ would be a disjoint union of open $(\mathbb{C}^{*})^{n}$-orbits and $N^{\circ}$ is connected. To see that 
	$\mathcal{O}^{\mathbb{C}}$ is open, recall that $\mathcal{O}^{\mathbb{C}}$ is the image of the orbit map $\Phi_{p}^{\mathbb{C}}:(\mathbb{C}^{*})^{n}\to N$. 
	Since $\Phi^{\mathbb{C}}_{p}\circ \varepsilon=T_{p}$ and $\varepsilon$ is a submersion, the image of $(\Phi^{\mathbb{C}}_{p})_{*_{e}}$ 
	coincides with the image of $(T_{p})_{*_{0}}$. By Lemma \ref{nefknwdknefknk}(3), the image of $(T_{p})_{*_{0}}$ is $T_{p}\mathcal{O}+JT_{p}\mathcal{O}$, where 
	$\mathcal{O}$ is the $\mathbb{T}^{n}$-orbit of $p$. By Lemma \ref{neekwnkvnkdnsk}, $T_{p}\mathcal{O}+JT_{p}\mathcal{O}=T_{p}N$. Thus $(T_{p})_{*_{0}}$ is surjective. It 
	follows that $\Phi_{p}^{\mathbb{C}}$ is a submersion at $e.$ Because orbit maps have constant rank, $\Phi^{\mathbb{C}}_{p}$ is a submersion, which implies that 
	its image is open. This shows (2).  

	(3) We already know that $\Phi_{p}^{\mathbb{C}}:(\mathbb{C}^{*})^{n}\to N$ is an holomorphic map and a submersion (see above). 
	Because $\textup{dim}((\mathbb{C}^{*})^{n})=\textup{dim}(N)$, $\Phi_{p}^{\mathbb{C}}$ is a local diffeomorphism. Since $\Phi^{\mathbb{C}}$ is free at $p$, 
	it is also injective. Thus $\Phi^{\mathbb{C}}_{p}$ is an holomorphic diffeomorphism. The equivariance property comes from the fact the $\Phi_{p}^{\mathbb{C}}$ is an 
	orbit map. 
\end{proof}

	Let $p\in N^{\circ}$ be fixed. Consider the following commutative diagram: 
\begin{figure}[h!]
                \begin{minipage}[c]{0.33\linewidth}
		\begin{eqnarray*}
		\begin{tikzcd}
		& \mathbb{C}^{n} \arrow[dd,bend right=70,swap,"\pi"] \arrow{d}{ \varepsilon} \arrow{dr}{ T_{p}}&   \\
		&  (\mathbb{C}^{*})^{n} \arrow{d}{ \sigma} \arrow[swap]{r}{ \Phi^{\mathbb{C}}_{p}}  & N^{\circ}\\
		&\mathbb{R}^{n} &
		\end{tikzcd}
		\end{eqnarray*} 
		\end{minipage}
                \begin{minipage}[c]{.67\linewidth}
	\begin{itemize}
	\item $\varepsilon(z_{1},...,z_{n})=(e^{2\pi z_{1}},...,e^{2\pi z_{n}})$, 
	\item $\sigma(z_{1},...,z_{n})=\big(\tfrac{\ln(|z_{1}|)}{2\pi},...,\tfrac{\ln(|z_{n}|)}{2\pi}\big)$,
	\item $\pi(z_{1},...,z_{n})=\big(\textup{Real}(z_{1}),...,\textup{Real}(z_{n})\big)$.
	\end{itemize}
		\end{minipage}
\end{figure}

	Let $g$ be the K\"{a}hler metric on $N$. Because $\varepsilon$ is a covering map and $\Phi_{p}^{\mathbb{C}}$ is a diffeomorphism, $T_{p}$ is 
	a covering map, and so the pullback $\widetilde{g}=(T_{p})^{*}g$ is a metric on $\mathbb{C}^{n}$. Let $J$ be the canonical complex structure on $\mathbb{C}^{n}$ 
	($J=$ multiplication by $i$). Given $x\in \mathbb{R}^{n}$, let $h_{x}$ be the bilinear form on $T_{x}\mathbb{R}^{n}=\mathbb{R}^{n}$ 
	defined by $h_{x}(u,v)=g_{T_{p}(x)}(u_{N},v_{n})$. Note that $T_{p}(x)=\varphi^{-Jx_{N}}(1,p)$ by Lemma \ref{nefknwdknefknk}(2). Thus
	\begin{eqnarray*}
		h_{x}(u,v)=g_{\varphi^{-Jx_{N}}(1,p)}(u_{N},v_{N}).
	\end{eqnarray*}
\begin{lemma}\label{nfeknkenfknd}
	\textbf{}
	\begin{enumerate}[(a)]
		\item $(\mathbb{C}^{n},\widetilde{g},J)$ is a K\"{a}hler manifold. 
		\item Given $t\in \mathbb{R}^{n}$, the map $f_{t}:\mathbb{C}^{n}\to \mathbb{C}^{n}$, $z\mapsto z+it$, is an isometry, that is, $f_{t}^{*}\widetilde{g}=\widetilde{g}$. 
		\item $\widetilde{g}_{z}(w,w')=\widetilde{g}_{\pi(z)}(w,w')$ for all $z,w,w'\in \mathbb{C}^{n}$,
		\item $\widetilde{g_{z}}(ia,b)=0$ for all $z\in \mathbb{C}^{n}$ and all $a,b\in \mathbb{R}^{n}$. 
		\item $\widetilde{g_{z}}(ia,ib)=\widetilde{g_{z}}(a,b)$ for all $z\in \mathbb{C}^{n}$ and all $a,b\in \mathbb{R}^{n}$. 
		\item The correspondence $x\in \mathbb{R}^{n}$, $x\mapsto h_{x}$ is a Riemannian metric on $\mathbb{R}^{n}$. 
		\item $\widetilde{g}_{x+iy}(a+ib,a'+ib')=h_{x}(a,a')+h_{x}(b,b')$ for all $x,y,a,a',b,b'\in \mathbb{R}^{n}$. 
	\end{enumerate}
\end{lemma}
\begin{proof}
	(a) By Lemma \ref{nefknwdknefknk}(4), $T_{p}:\mathbb{C}^{n}\to N$ is holomorphic, and since $T_{p}$ is an isometric covering map, it is locally a holomorphic 
	isometry. This forces $(\mathbb{C}^{n},\widetilde{g},J)$ to be a K\"{a}hler manifold. 

	\noindent (b) A simple diagram chase shows that $T_{p}\circ f_{t}=\Phi_{[t]}\circ T_{p}$. Since $T_{p}$ is a K\"{a}hler covering, this implies that $f_{t}$ is isometric 
	if and only if $\Phi_{[t]}$ is isometric, which is the case. 

	\noindent (c) Follows from (b). To see (d), we use Lemma \ref{nefknwdknefknk}(3):
	
	\begin{eqnarray*}
		\widetilde{g}_{z}(ia,b)=g_{T_{p}(z)}((T_{p})_{*}ia, (T_{p})_{*}b)=g_{T_{p}(z)}(a_{N},-Jb_{N}).
	\end{eqnarray*}
	By Lemma \ref{nfekwndknfeknkdnk}(2), the point $T_{p}(z)=\Phi^{\mathbb{C}}(\varepsilon(z),p)$ belongs to $N^{\circ}$ and hence 
	the decomposition $T_{T_{p}(z)}\mathcal{O}\oplus J(T_{T_{p}(z)}\mathcal{O})$ is orthogonal 
	relative to the K\"{a}hler metric $g$, where $\mathcal{O}$ denotes the 
	$\mathbb{T}^{n}$-orbit of $T_{p}(z)$ in $N$ (see Lemma \ref{neekwnkvnkdnsk}). It follows that $g_{T_{p}(z)}(a_{N},-Jb_{N})=0$, and thus $\widetilde{g}_{z}(ia,b)=0$. 
	This shows (d). Analogously, one shows (e). (f) Let $j:\mathbb{R}^{n}\to \mathbb{C}^{n}$ be the canonical injection. Given $x,u,v\in \mathbb{R}^{n}$, we compute:
	\begin{eqnarray*}
		(j^{*}\widetilde{g})_{x}(u,v)&=&\widetilde{g}_{x}(u,v)=\big((T_{p})^{*}g\big)_{x}(u,v)=g_{T_{p}(x)}\big((T_{p})_{*}u,(T_{p})_{*}v\big)\\
		&=& g_{T_{p}(x)}(-Ju_{N},-Jv_{N})=g_{T_{p}(x)}(u_{N},v_{N})=h_{x}(u,v),
	\end{eqnarray*}
	where we have used Lemma \ref{nefknwdknefknk}(3). Thus $h=j^{*}\widetilde{g}$. This shows that $h$ is a Riemannian metric on $\mathbb{R}^{n}$. 
	(g) By a direct calculation using (c),(d) and $(e)$.  This concludes the proof.  
\end{proof}

\begin{proof}[Proof of Theorem \ref{nckdnwknknednenkenk}]
	\noindent (1) To prove that $(\mathbb{R}^{n},h,\nabla^{\textup{flat}})$ is dually flat, it suffices to show that the almost Hermitian manifold associated to 
	$(h,\nabla^{\text{flat}})$ via Dombrowski's construction is K\"{a}hler (see Proposition \ref{prop:3.8}). By Corollary \ref{nfednwkneknkn} and Lemma \ref{nfeknkenfknd}(g), 
	$(\mathbb{C}^{n}, \widetilde{g},J)$ is the almost Hermitian manifold associated to $(h,\nabla^{\textup{flat}})$ via Dombrowski's construction. By Lemma \ref{nfeknkenfknd}(1), 
	it is a K\"{a}hler manifold. 
	
	\noindent (2) Apply Proposition \ref{nkwdnkkfenknk} to the covering map $T_{p}:\mathbb{C}^{n}\to N^{\circ}$ and the parallel lattice $L\subset T\mathbb{R}^{n}$ 
	generated by the vector fields $\tfrac{\partial}{\partial x_{1}},...,\tfrac{\partial}{\partial x_{n}}$.  

	\noindent (3) Let $\moment:N\to \mathbb{R}^{n}$ be a momentum map. Since $\moment\circ\Phi_{p}^{\mathbb{C}}$ is 
	$\mathbb{T}^{n}$-invariant and $\sigma$ is a principal $\mathbb{T}^{n}$-bundle, 
	there is a smooth map $f=(f^{1},...,f^{n}):\mathbb{R}^{n}\to \mathbb{R}^{n}$ such that $f\circ \sigma=\moment\circ \Phi_{p}^{\mathbb{C}}$. 
	We claim that $\tfrac{\partial f^{i}}{\partial x_{j}}(x)=-h_{ij}(x):=-h_{x}(\tfrac{\partial}{\partial x_{i}},\tfrac{\partial}{\partial x_{j}})$ for every $x\in \mathbb{R}^{n}$ and every 
	$i,j=1,...,n$. To see this, let $x\in \mathbb{R}^{n}$ be arbitrary. Since the coordinate functions of the curve $t\mapsto x+te_{j}$ are real-valued, we have 
	$x+te_{j}=\pi(x+te_{j})=(\sigma\circ \varepsilon)(x+te_{j})$ (see the diagram before Lemma \ref{nfeknkenfknd}), and so 
	\begin{eqnarray*}
		\dfrac{\partial f^{i}}{\partial x_{j}}(x)&=&\dfrac{d}{dt}\bigg\vert_{0}f^{i}(x+te_{j})=
		\dfrac{d}{dt}\bigg\vert_{0}(f^{i}\circ\sigma\circ\varepsilon)(x+te_{j})\\
		&=&\dfrac{d}{dt}\bigg\vert_{0} (\moment^{e_{i}}\circ \Phi_{p}^{\mathbb{C}}\circ \varepsilon\big)(x+te_{j})\big)
		=\dfrac{d}{dt}\bigg\vert_{0} (\moment^{e_{i}}\circ T_{p})(x+te_{j}),
	\end{eqnarray*}
	where we have used $\Phi_{p}^{\mathbb{C}}\circ \varepsilon=T_{p}$. Since $(T_{p})_{*_{x}}(e_{j})=-J(e_{j})_{N}(T_{p}(x))$ by Lemma \ref{nefknwdknefknk}(3), 
	we get 
	\begin{eqnarray*}
		\dfrac{\partial f^{i}}{\partial x_{j}}(x)&=&-(\moment^{e_{i}})_{*_{T_{p}(x)}}J(e_{j})_{N}=-\omega_{T_{p}(x)}\big((e_{i})_{N},J(e_{j})_{N}\big)\\
		&=&  -g_{T_{p}(x)}\big((e_{i})_{N},(e_{j})_{N}\big)  =-h_{x}(e_{i},e_{j})=-h_{ij}(x). 
	\end{eqnarray*}
	This concludes the proof of the claim. Consider the 1-form $\theta$ on $\mathbb{R}^{n}$ defined by 
	$\theta=\sum_{k=1}^{n}f^{k}dx_{k}.$ The form $\theta$ is closed ($d\theta=0$) if and only if $\tfrac{\partial f^{i}}{\partial x_{j}}-
	\tfrac{\partial f^{j}}{\partial x_{i}}=0$ on $\mathbb{R}^{n}$ for all $i,j=1,...,n$, which is the case since $\tfrac{\partial f^{j}}{\partial x_{i}}=-h_{ij}$. Therefore there exists 
	a smooth function $\psi:\mathbb{R}^{n}\to \mathbb{R}$ such that $\theta=-d\psi$. Clearly  $\tfrac{\partial \psi}{\partial x_{i}}=-f^{i}$ and hence 
	$\tfrac{\partial^{2} \psi}{\partial x_{i} \partial x_{j}}=-\tfrac{\partial f^{i}}{\partial x_{j}}=h_{ij}$ for all $i,j=1,...,n$. Thus $h$ is the Hessian of $\psi$. 
	Let $y=\textup{grad}(\psi)=\big(\tfrac{\partial \psi}{\partial x_{1}},...,\tfrac{\partial \psi}{\partial x_{n}}\big)$ be the gradient map of $\psi.$
	By Lemma \ref{nefkwnknkenknk}, $y=(y_{1},...,y_{n})$ are coordinates on $\mathbb{R}^{n}$ and $(x,y)$ is a global pair of dual coordinate
	systems on $(\mathbb{R}^{n},h,\nabla^{\textup{flat}}).$ This shows (a). Next we prove (b). To see that $-y:\mathbb{R}^{n}\to \moment(N^{\circ})$ is an isometry,
	let $\pi=\sigma\circ \varepsilon:\mathbb{C}^{n}\to \mathbb{R}^{n}$ be the canonical projection of the tangent bundle 
	$T\mathbb{R}^{n}=\mathbb{C}^{n}$. Since $-y=f$ and $f\circ \sigma=\moment\circ \Phi_{p}^{\mathbb{C}}$, we have 
	\begin{eqnarray*}
		-y\circ \pi= f\circ \sigma\circ \varepsilon= \moment\circ \Phi_{p}^{\mathbb{C}}\circ \varepsilon= \moment\circ T_{p}
	\end{eqnarray*}
	and hence 
	\begin{eqnarray*}
		(-y)^{*}k=h \,\,\,\,\,\,\,\,\,\,&\Leftrightarrow& \,\,\,\,\,\,\,\,\,\,\pi^{*}(-y)^{*}k=\pi^{*}h\\
		&\Leftrightarrow& \,\,\,\,\,\,\,\,\,\,\big((-y)\circ \pi\big)^{*}k=\pi^{*}h\\
		&\Leftrightarrow& \,\,\,\,\,\,\,\,\,\,\big(\moment \circ T_{p})^{*}k=\pi^{*}h.
	\end{eqnarray*}
	Thus it suffices to show that $\big(\moment \circ T_{p})^{*}k=\pi^{*}h$. Let $z=x+iy,w=a+ib,w'=a'+ib'\in \mathbb{C}^{n}$ be arbitrary, 
	where $x,y,a,a',b,b'\in \mathbb{R}^{n}$. By Lemma \ref{nefknwdknefknk}(3) and the fact that $\moment$ is constant along 
	$\mathbb{T}^{n}$-orbits, we have $\moment_{*}(T_{p})_{*_{z}}w=-\moment_{*}Ja_{N}(T_{p}(z))$, and so 
	\begin{eqnarray*}
		\Big(\big(\moment \circ T_{p})^{*}k\Big)_{z}(w,w')&=& k_{(\moment\circ T_{p})(z)}(\moment_{*}(T_{p})_{*_{z}}w,\moment_{*}(T_{p})_{*_{z}}w')\\
		&=&k_{(\moment\circ T_{p})(z)}(\moment_{*}Ja_{N}(T_{p}(z)),\moment_{*}Ja'_{N}(T_{p}(z)))\\
		&=& g_{T_{p}(z)}(a_{N}(T_{p}(z)),a'_{N}(T_{p}(z))). 
	\end{eqnarray*}
	Using again Lemma \ref{nefknwdknefknk}(3), we see that $a_{N}(T_{p}(z))=(T_{p})_{*_{z}}ia$, and so
	\begin{eqnarray*}
		\Big(\big(\moment \circ T_{p})^{*}k\Big)_{z}(w,w')= g_{T_{p}(z)}((T_{p})_{*_{z}}ia, (T_{p})_{*_{z}}ia')=\widetilde{g}_{z}(ia,ia'), 
	\end{eqnarray*}
	where $\widetilde{g}=(T_{p})^{*}g$. It follows from this and Lemma \ref{nfeknkenfknd}(g), that 
	\begin{eqnarray}\label{nckndkwenknk}
		\Big(\big(\moment \circ T_{p})^{*}k\Big)_{z}(w,w')=h_{x}(a,a').
	\end{eqnarray}

	On the other hand, it follows from the linearity of the map $\pi:\mathbb{C}^{n}\to \mathbb{R}^{n},$ $(z_{1},...,z_{n})\mapsto 
	(\textup{Real}(z_{1}),...,\textup{Real}(z_{n})) $ that 
	\begin{eqnarray}\label{nekenknckn}
		(\pi^{*}h)_{z}(w,w')=h_{\pi(z)}(\pi(w),\pi(w'))=h_{x}(a,a'). 
	\end{eqnarray}
	Comparing \eqref{nckndkwenknk} and \eqref{nekenknckn} we get $\big(\moment \circ T_{p})^{*}k=\pi^{*}h$. This concludes the proof that $-y$ is an isometry. 
	Obviously $-y$ is affine from $(\mathbb{R}^{n},(\nabla^{\textup{flat}})^{*})$ to $(\moment(N^{\circ}),\nabla^{\textup{flat}})$ (since $y=(y_{1},...,y_{n})$ are 
	$(\nabla^{\textup{flat}})^{*}$-affine coordinates). The fact that it is affine from 
	$(\mathbb{R}^{n},\nabla^{\textup{flat}})$ to $(\moment(N^{\circ}),\nabla^{k})$ follows from Lemma \ref{ncekndwkneknku}. This concludes the proof of (b). 
	Finally, suppose there is no $C\in \mathbb{R}^{n}-\{0\}$ such that $\moment(N^{\circ})=\moment(N^{\circ})+C$. 
	Let $y'=(y'_{1},...,y_{n}')$ be another system of coordinates on $\mathbb{R}^{n}$ satisfying (a) and (b). By Proposition 
	\ref{neknkneknekenfkenk}, there is a smooth function $\psi':\mathbb{R}^{n}\to \mathbb{R}$ such that $y'=\textup{grad}(\psi')=(\tfrac{\partial \psi'}{\partial x_{1}},...,
	\tfrac{\partial \psi'}{\partial x_{n}})$ and $\tfrac{\partial^{2}\psi'}{\partial x_{i}\partial x_{j}}= 
	h(\tfrac{\partial}{\partial x_{i}},\tfrac{\partial}{\partial x_{j}})$ for all $i,j=1,...,n$. Thus 
	$\tfrac{\partial^{2}}{\partial x_{i}\partial x_{j}}(\psi-\psi')= 0$ for all $i,j=1,...,n$, where $\psi$ is the function defined above, in the proof of (a). Thus 
	there are real numbers $c_{0},...,c_{n}$ such that 
	$\psi(x)=\psi'(x)+c_{0}+c_{1}x_{1}+...+c_{n}x_{n}$ for all $x=(x_{1},...,x_{n})\in \mathbb{R}^{n}$. Taking the gradient yields $y=y'+C$, 
	where $C=(c_{1},...,c_{n})$. It follows from this and (b) that $\moment(N^{\circ})=\moment(N^{\circ})+C$. By our assumption, $C=0.$ Thus $y=y'$. 
\end{proof}

	Combining the results of Proposition \ref{unfenkwneknknk} and Theorem \ref{nckdnwknknednenkenk}, we obtain the following corollary. 

\begin{corollary}\label{nefknknkfnknkn}
	Let the hypotheses be as in Theorem \ref{nckdnwknknednenkenk}. If $\moment:N\to \mathbb{R}^{n}$ is a momentum map, then 
	$\moment(N^{\circ})$ is an open convex subset of $\mathbb{R}^{n}$. If in addition $\moment$ is proper, 
	that is, if $\moment^{-1}(K)$ is compact whenever $K\subseteq \mathbb{R}^{n}$ is compact, then $\moment(N)\subset \mathbb{R}^{n}$ is convex.  
\end{corollary}

\begin{remark}
	Stronger and/or more general results on the convexity properties of the momentum map are available in the litterature 
	\cite{Atiyah, Ratiu2,Ratiu1,Molino,Guillemin82, Neeb,Kirwan,Sjamaar}. With the notable exception of \cite{Sjamaar}, they are all based 
	on either Morse theoretical techniques \cite{Atiyah,Guillemin82,Kirwan} or some ``Lokal-global-Prinzip" \cite{Ratiu2,Ratiu1,Molino,Neeb}.
	Our approach is interesting because it relies only on classical convexity theory and Legendre transform (see the appendix).
\end{remark}

	The following corollary is immediate.
\begin{corollary}\label{nknkwkkdnkn}
	Let the hypotheses be as in Theorem \ref{nckdnwknknednenkenk}. Then $\Phi:\mathbb{T}^{n}\times N\to N$ is a torification of $(\moment(N^{\circ}),k,\nabla^{k}).$ 
\end{corollary}

\begin{proposition}[\textbf{Complement of Theorem \ref{nckdnwknknednenkenk}}]\label{ncekndwkneknk}
	Let the hypotheses be as in Theorem \ref{nckdnwknknednenkenk}. 	Let $L\subset T\mathbb{R}^{n}$ be the parallel lattice generated by 
	$X=(\tfrac{\partial}{\partial x_{1}},...,\tfrac{\partial}{\partial x_{n}})$. Under the identifications $(\mathbb{C}^{*})^{n}=\mathbb{C}^{n}/i\mathbb{Z}^{n}=T\mathbb{R}^{n}/\Gamma(L)$, 
	the triple $(L,X,\Phi^{\mathbb{C}}_{p})$ is a toric parametrization. The induced 
	toric factorization is $\big(T_{p},\sigma\circ (\Phi_{p}^{\mathbb{C}})^{-1}\big) $. 
\end{proposition}
\begin{proof}
	By inspection of the proof of Theorem \ref{nckdnwknknednenkenk}.
\end{proof}

	Given a smooth function $f:U\subseteq \mathbb{R}^{n}\to \mathbb{R}$, we will denote by $\textup{grad}(f):U\to \mathbb{R}^{n}$ its gradient map. 
	Thus $\textup{grad}(f)(x)=\big(\tfrac{\partial f}{\partial x_{1}}(x),...,\tfrac{\partial f}{\partial x_{n}}(x)\big)$, where $(x_{1},...,x_{n})$ are standard 
	coordinates on $\mathbb{R}^{n}$. 

\begin{definition}
	In the situation of Theorem \ref{nckdnwknknednenkenk}, we say that a smooth function $\psi:\mathbb{R}^{n}\to \mathbb{R}$ is a \textit{potential} 
	if $\tfrac{\partial^{2}\psi}{\partial x_{i}\partial x_{j}}=h_{ij}=h(\tfrac{\partial}{\partial x_{i}},\tfrac{\partial}{\partial x_{j}})$ for all $i,j=1,...,n$. 
	We sall say that a potential $\psi:\mathbb{R}^{n}\to \mathbb{R}$ is \textit{compatible} with the momentum map $\moment:N\to \mathbb{R}^{n}$ 
	if the gradient map $y=\textup{grad}(\psi)$ satisfies (a) and (b) of Theorem \ref{nckdnwknknednenkenk}.  
\end{definition}

	Note that (see Proposition \ref{neknkneknekenfkenk} and Lemma \ref{nefkwnknkenknk}):
	\begin{itemize} 
	\item If $y=(y_{1},...,y_{n})$ is a system of coordinates on $\mathbb{R}^{n}$ satisfying (a) of Theorem \ref{nckdnwknknednenkenk}, 
		then $y=\textup{grad}(\psi)$ for some smooth function $\psi:\mathbb{R}^{n}\to \mathbb{R}$ and $\psi$ is automatically a potential. 
	\item Given a momentum map $\moment:N\to \mathbb{R}^{n}$, there exists a potential $\psi$ compatible with $\moment$. 
	\item If $\psi:\mathbb{R}^{n}\to \mathbb{R}$ is a potential, then $(x,\textup{grad}(\psi))$ is a global pair of dual coordinate systems on 
		$(\mathbb{R}^{n},h,\nabla^{\textup{flat}})$, but $-y=-\textup{grad}(\psi)$ may fail to satisfy (b) of Theorem \ref{nckdnwknknednenkenk}. 
	\end{itemize}

\begin{lemma}\label{nfekwdnkenfkfnk}
	Let the hypotheses be as in Theorem \ref{nckdnwknknednenkenk}. Let $\moment:N\to \mathbb{R}^{n}$ be a momentum map.
	Suppose there is no $C\in \mathbb{R}^{n}-\{0\}$ such that $\moment(N^{\circ})=\moment(N^{\circ})+C$. 
	If $\psi:\mathbb{R}^{n}\to \mathbb{R}$ a potential satisfying $(-\textup{grad}(\psi))(\mathbb{R}^{n})=\moment(N^{\circ})$, 
	then $\psi$ is compatible with $\moment$. 
\end{lemma}
\begin{proof}
	This follows from Proposition \ref{neknkneknekenfkenk} and the uniqueness part of Theorem \ref{nckdnwknknednenkenk}(3).
\end{proof}

\begin{proposition}[\textbf{Holomorphic versus symplectic}]\label{nekfnkwdnkneknknk}
	Let the hypotheses be as in Theorem \ref{nckdnwknknednenkenk}. Let $\psi:\mathbb{R}^{n}\to \mathbb{R}$ be a potential compatible 
	with the momentum map $\moment:N\to \mathbb{R}^{n}$. Define $\phi:\moment(N^{\circ})\to \mathbb{R}$ by 
	\begin{eqnarray}\label{nfknwdknfknknk}
		\phi(x)=-\big\langle x,(-\textup{grad}(\psi))^{-1}(x)\big\rangle-\psi\big((-\textup{grad}(\psi))^{-1}(x)\big),
	\end{eqnarray}
	where $\langle\,,\,\rangle$ is the Euclidean pairing. Then,
	\begin{enumerate}[(1)]
	\item $(\textup{grad}(\phi),x)$ is a global pair of dual coordinate systems on $(\moment(N^{\circ}),k,\nabla^{k})$. 
	\item $-\textup{grad}(\phi)$ is an isomorphism of dually flat spaces from
		$(\moment(N^{\circ}),k,\nabla^{k})$ to $(\mathbb{R}^{n},h,\nabla^{\textup{flat}})$, 
		with inverse $(-\textup{grad}(\phi))^{-1}=-\textup{grad}(\psi)$. 
	\item $k$ is the Hessian of $\phi$, that is, 
		$\tfrac{\partial^{2} \phi}{\partial x_{i}\partial x_{j}}=k(\tfrac{\partial}{\partial x_{i}},\tfrac{\partial}{\partial x_{j}})$ 
		for all $i,j=1,...,n$.  
	\end{enumerate}
%
\end{proposition}
\begin{proof}
	First we prove that $-\textup{grad}(\phi)=(-\textup{grad}(\psi))^{-1}$. Let $b=-\textup{grad}(\psi)(a)\in \moment(N^{\circ})$ be arbitrary. 
	By definition of $\phi$, we have 
	\begin{eqnarray*}
		\phi((-\textup{grad}(\psi))(a))=-\big\langle (-\textup{grad}(\psi))(a),a\big\rangle-\psi(a).
	\end{eqnarray*}
	Taking the derivative with respect to $a$ in the direction $\tfrac{\partial}{\partial x_{i}}$, we get 
	\begin{eqnarray*}
		\big\langle \big(\textup{grad}(\phi)\circ (-\textup{grad}(\psi))\big)(a), -H_{i}\big\rangle=\big\langle H_{i}, a\big\rangle 
		+\Big\langle \textup{grad}(\psi),\dfrac{\partial}{\partial x_{i}}\Big\rangle -\dfrac{\partial \psi}{\partial x_{i}}=\big\langle H_{i}, a\big\rangle,
	\end{eqnarray*}
	where $H_{i}=\tfrac{\partial}{\partial x_{i}}\textup{grad}(\psi)=
	(\tfrac{\partial^{2}\psi}{\partial x_{i}\partial x_{1}},...,\tfrac{\partial^{2}\psi}{\partial x_{i}\partial x_{n}})=(h_{1i},...,h_{ni})$. 
	Thus 
	\begin{eqnarray*}
		\big\langle \big((-\textup{grad}(\phi))\circ (-\textup{grad}(\psi))\big)(a) -a , H_{i}\big\rangle=0
	\end{eqnarray*}
	for all $a\in \mathbb{R}^{n}$ and all $i=1,...,n$. Since $[h_{ij}]$ is invertible, this implies $\big((-\textup{grad}(\phi))\circ (-\textup{grad}(\psi))\big)(a)=a$ 
	for all $a\in \mathbb{R}^{n}$. Therefore $-\textup{grad}(\phi)=(-\textup{grad}(\psi))^{-1}$. Since $-\textup{grad}(\psi)$ is an isomorphism of dually flat spaces, 
	so does $-\textup{grad}(\phi)$. This shows (2). Because $(x,y)$ is a global pair of dual coordinates systems on $\mathbb{R}^{n}$ and $-y$ is 
	an isomorphism of dually flat spaces, $(x\circ (-y)^{-1},y\circ (-y)^{-1})=(-\textup{grad}(\phi),-x)$ is a global pair of dual coordinate systems 
	on $\moment(N^{\circ})$. Obviously, the same is true for $(\textup{grad}(\phi),x)$. This shows (1). Finally, (3) is a consequence of (1) and Proposition 
	\ref{neknkneknekenfkenk}. 
\end{proof}

\begin{remark}
	The function $\phi$ defined in \eqref{nfknwdknfknknk} satisfies $\phi(x)=\psi^{*}(-x)$ for all $x\in \moment(N^{\circ})$, where 
	$\psi^{*}$ is the Legendre transform of $\psi$ (see Definition \ref{jefkwjkjkejk}). 
\end{remark}

\begin{definition}
	In the situation of the preceding proposition, we say that $\phi$ is the \textit{dual} of $\psi$ (see \eqref{nfknwdknfknknk}).
\end{definition}

\begin{example}\label{ncenknkenk}
	Let $\Phi:\mathbb{T}^{n}\times \mathbb{C}^{n}\to \mathbb{C}^{n}$ be defined by $\Phi([t],z)=(e^{2i\pi t_{1}}z_{1},...,e^{2i\pi t_{n}}z_{n})$. 
	We endow $\mathbb{C}^{n}$ with the standard flat K\"{a}hler structure. Then $\Phi$ is isometric and Hamiltonian, with momentum map 
	$\moment:\mathbb{C}^{n}\to \mathbb{R}^{n}$ given by $\moment(z)=-\pi (|z_{1}|^{2}, ..., |z_{n}|^{2})$. 

	Given $\xi\in \mathbb{R}^{n}=\textup{Lie}(\mathbb{T}^{n})$, the fundamental vector field $\xi_{\mathbb{C}^{n}}$ at $z\in \mathbb{C}^{n}$ is 
	given by
	\begin{eqnarray*}
		\xi_{\mathbb{C}^{n}}(z)=\tfrac{d}{dt}\big\vert_{0}\Phi(\textup{exp}(t\xi),x)=2i\pi (\xi_{1}z_{1},...,\xi_{n}z_{n}). 
	\end{eqnarray*}
	Integral curves of $-i\xi_{\mathbb{C}^{n}}$ are of the form $\alpha(t)=(\lambda_{1}e^{2\pi \xi_{1}t},...,\lambda_{k}e^{2\pi\xi_{n}t})$, 
	where $\lambda_{1},...,\lambda_{n}\in \mathbb{C}$. Therefore $-i\xi_{\mathbb{C}^{n}}$ is complete and its flow $\varphi^{\xi}(t,z)$ is given by 
	$\varphi^{\xi}(t,z)=(z_{1}e^{2\pi \xi_{1}t},...,z_{n}e^{2\pi \xi_{n}t})$. Let $p=(1,1,...,1)\in \mathbb{C}^{n}$ and let $h$ 
	be the metric on $\mathbb{R}^{n}$ defined by $h_{x}(u,v)=g_{\varphi^{x}(1,p)}(u_{\mathbb{C}^{n}},v_{\mathbb{C}^{n}})$, $x,u,v\in \mathbb{R}^{n}$, and where $g$ 
	is the Euclidean metric on $\mathbb{C}^{n}$. Let $\langle z,w\rangle=\overline{z}_{1}w_{1}+...+\overline{z}_{n}w_{n}$ be the Hermitian product on $\mathbb{C}^{n}$. 
	We have $g_{z}(w,w')=\textup{Real}\langle w,w\rangle$ for all $z,w,'w\in \mathbb{C}^{n}$, and so 
	\begin{eqnarray*}
		h_{x}(u,v)&=& \textup{Real}\big\langle u_{\mathbb{C}^{n}}(\varphi^{x}(1,p)),v_{\mathbb{C}^{n}}(\varphi^{x}(1,p))\big\rangle\\
		&=& \textup{Real} \big\langle 2i\pi (u_{1}e^{2\pi x_{1}},...,u_{n}e^{2i\pi x_{n}}), 2i\pi (v_{1}e^{2\pi x_{1}},...,v_{n}e^{2i\pi x_{n}}) \big\rangle\\
		&=& 4\pi^{2}\big(u_{1}v_{1}e^{4\pi x_{1}}+...+u_{n}v_{n}e^{4\pi x_{n}}\big). 
	\end{eqnarray*}
	Therefore the matrix representation of $h$ at $x\in \mathbb{R}^{n}$ is given by 
	\begin{eqnarray*}
		4\pi^{2}
		\begin{bmatrix}
			e^{4\pi x_{1}}	 & \cdots  & 0 \\
					 & \ddots &   \\
			0                &         & e^{4\pi x_{n}}
		\end{bmatrix}.
	\end{eqnarray*}
	It is the Hessian of the function $\psi:\mathbb{R}^{n}\to \mathbb{R}$, $x\mapsto \tfrac{1}{4}(e^{4\pi x_{1}}+...+e^{4\pi x_{n}})$. The image of 
	$-\textup{grad}(\psi)=-\pi (e^{4\pi x_{1}},...,e^{4\pi x_{n}})$ is the ``negative quadrant" 
	$Q=\{(x_{1},...,x_{n})\in \mathbb{R}^{n}\,\,\big\vert\,\,x_{k}<0\,\,\textup{for all}\,\,k=1\}$. It is immediate to verify that 
	$Q=\moment((\mathbb{C}^{n})^{\circ})$ and hence the potential $\psi$ is compatible with $\moment$. 
	Let $\phi:Q\to \mathbb{R}$ be the dual of $\psi$. A direct computation shows that 
	\begin{eqnarray*}
		\phi(x_{1},...,x_{n})=-\dfrac{1}{4\pi}\sum_{k=1}^{n}\Big[x_{k}\ln\Big(\dfrac{-x_{k}}{\pi}\Big)-x_{k}\Big]. 
	\end{eqnarray*}
	Taking the Hessian yields the matrix representation of the Riemannian metric $k$ induced on $Q$ at $(x_{1},...,x_{n})\in Q$:
	\begin{eqnarray*}
		-\dfrac{1}{4\pi}
		\begin{bmatrix}
			\tfrac{1}{x_{1}}  &  \cdots & 0 \\
					  &  \ddots &   \\
			0                 &         & \tfrac{1}{x_{n}}
		\end{bmatrix}.
	\end{eqnarray*}
\end{example}

	A slightly more involved example is the complex projective space. Let $\mathbb{P}_{n}(c)$ denote the complex projective space of complex dimension $n$, 
	endowed with the Fubini-Study metric $g_{c}$ normalized in such a way that the holomorphic sectional curvature is $c$.

	Let $\Phi:\mathbb{T}^{n}\times \mathbb{P}_{n}(c)\to \mathbb{P}_{n}(c)$ be the torus action defined by 
	\begin{eqnarray*}
		\Phi([t],[z])= \big[e^{2i\pi t_{1}}z_{1},...,e^{2i\pi t_{n}}z_{n},z_{n+1}\big],
	\end{eqnarray*}
	where $[t]=[t_{1},...,t_{n}]\in \mathbb{T}^{n}$ and 
	$[z]=[z_{1},...,z_{n+1}]\in \mathbb{P}_{n}(c)$ (homogeneous coordinates). Then $\Phi$ is isometric and Hamiltonian, with momentum map 
	$\moment_{c}:\mathbb{P}_{n}(c)\to \mathbb{R}^{n}$ given by 
	\begin{eqnarray*}
		\moment_{c}([z])=-\dfrac{4\pi}{c} \,\bigg(\dfrac{|z_{1}|^{2}}{\langle z,z \rangle}, ..., \dfrac{|z_{n}|^{2}}{\langle z,z \rangle}\bigg),
	\end{eqnarray*}
	where $\langle z,w\rangle= \overline{z}_{1}w_{1}+...+\overline{z}_{n+1}w_{n+1}$ is the standard Hermitian product on $\mathbb{C}^{n+1}$. 
	The image of $\mathbb{P}_{n}(c)^{\circ}=\{[z_{1},...,z_{n+1}]\,\,\vert\,\,z_{k}\neq 0\,\,\forall k=1,...,n+1\}$ under $\moment_{c}$ is $-\tfrac{4\pi}{c}S_{n}$, where 
	$S_{n}\subset \mathbb{R}^{n}$ is the set $\{(x_{1},...,x_{n})\in \mathbb{R}^{n}\,\,\vert\,\,x_{k}>0\,\,
	\textup{for all}\,\,k\,\, \textup{and}\,\,\,\sum_{k=1}^{n}x_{k}< 1\big\}$.

\begin{proposition}\label{nfknwkdnkdnkn}
	Let $\Phi:\mathbb{T}^{n}\times \mathbb{P}_{n}(c)\to \mathbb{P}_{n}(c)$ and $\moment_{c}:\mathbb{P}_{n}(c)\to \mathbb{R}^{n}$ be as defined above.  
	Let $h_{c}$ be the Riemannian metric on $\mathbb{R}^{n}$ associated to $p=[1,...,1]\in \mathbb{P}_{n}(c)^{\circ}$. 
	Then the function $\psi_{c}:\mathbb{R}^{n}\to \mathbb{R}$, defined by
	\begin{eqnarray*}
		\psi_{c}(x)=\dfrac{1}{c}\,\ln\big(1+e^{4\pi x_{1}}+...+e^{4\pi x_{n}}\big),
	\end{eqnarray*}
	is a potential compatible with $\moment_{c}.$ Its dual 
	$\phi_{c}:\moment(\mathbb{P}_{n}(c)^{\circ})\to \mathbb{R}$ is given by: 
	\begin{eqnarray*}
		\phi_{c}(x)=-\dfrac{1}{4\pi }\sum_{k=1}^{n}x_{k}\ln(-x_{k})+\dfrac{1}{4\pi}\bigg(\dfrac{4\pi}{c}+\sum_{k=1}^{n}x_{k}\bigg)
		\ln\bigg(\dfrac{4\pi}{c}+\sum_{k=1}^{n}x_{k}\bigg)-\dfrac{1}{c}\,\ln\bigg(\dfrac{4\pi}{c}\bigg). 
	\end{eqnarray*}
\end{proposition}
%
\begin{proof}
	Let $\Phi'$ denote the action of $\mathbb{T}^{n}$ on $\mathbb{C}^{n+1}-\{0\}$ defined by $\Phi'([t],z)= (e^{2i \pi t_{1}}z_{1},...,e^{2i\pi t_{n}}z_{n},z_{n+1})$.
	Let $f:\mathbb{C}^{n+1}-\{0\}\to \mathbb{P}_{n}(c)$, $z\mapsto [z]$. Because $f$ is $\mathbb{T}^{n}$-equivariant and holomorphic, we have 
	\begin{itemize}
		\item $\xi_{\mathbb{P}_{n}(c)}([z])= f_{*_{z}}\xi_{\mathbb{C}^{n+1}-\{0\}}(z)$ and 
		\item $J\xi_{\mathbb{P}_{n}(c)}([z])=f_{*_{z}}i\xi_{\mathbb{C}^{n+1}-\{0\}}(z)$
	\end{itemize}
	for all $\xi\in \mathbb{R}^{n}=\textup{Lie}(\mathbb{T}^{n})$ and all $z\in \mathbb{C}^{n+1}-\{0\}$, where $J$ denotes the complex structure on $\mathbb{P}_{n}(c)$.
	Clearly
	\begin{eqnarray*}
		\xi_{\mathbb{C}^{n+1}-\{0\}}(z)=\dfrac{d}{dt}\bigg\vert_{0}\Phi'([t\xi],z)= 2i\pi (\xi_{1}z_{1},..., \xi_{n}z_{n},0),
	\end{eqnarray*}
	and so $-J\xi_{\mathbb{P}_{n}(c)}([z])= 2\pi f_{*_{z}} (\xi_{1}z_{1},..., \xi_{n}z_{n},0).$ It is then easy to verify that 
	\begin{eqnarray*}
		\varphi_{t}^{-J\xi_{\mathbb{P}_{n}(c)}}([z])=\big[e^{2\pi t\xi_{1}}z_{1},...,e^{2\pi t\xi_{n}}z_{n},z_{n+1}\big]
	\end{eqnarray*}
	is the flow of $-J\xi_{\mathbb{P}_{n}(c)}$. 

	Let $h_{c}$ be the Riemannian metric on $\mathbb{R}^{n}$ associated to $p=[1,...,1]\in\mathbb{P}_{n}(c)^{\circ}$. Thus 
	$(h_{c})_{x}(u,v)=(g_{c})_{\varphi_{1}^{-Jx_{\mathbb{P}_{n}(c)}}(p)}(u_{\mathbb{P}_{n}(c)},v_{\mathbb{P}_{n}(c)})$, where 
	$x,u,v\in \mathbb{R}^{n}=\textup{Lie}(\mathbb{T}^{n})$. To compute the matrix representation of $h_{c}$, 
	we use the following facts:
	\begin{enumerate}[(1)]
		\item The Hopf fibration $\pi_{H}=f\vert_{S^{2n+1}}:S^{2n+1}\to \mathbb{P}_{n}(4)$ is a Riemannian submersion and $g_{c}=\tfrac{4}{c}g_{4}$. 
		\item $u_{\mathbb{C}^{n+1}-\{0\}}$ is tangent to $S^{2n+1}$ and for every $z\in S^{2n+1}$, $(\pi_{H})_{*_{z}}u_{\mathbb{C}^{n+1}-\{0\}}(z)
			=u_{\mathbb{P}_{n}(c)}([z])$.
		\item Given $z\in S^{2n+1}$, $T_{z}S^{2n+1}=\mathbb{R}iz\oplus z^{\perp}$, where 
			$z^{\perp}=\{w\in \mathbb{C}^{n+1}\big\vert\,\langle w,z\rangle=0 \}$. In this decomposition, $\mathbb{R}iz$ is the tangent space 
			of the fiber $\pi_{H}^{-1}(\pi_{H}(z))$ and $z^{\perp}$ is its orthogonal complement in $T_{z}S^{2n+1}$. 
		\end{enumerate}
	Given $w\in T_{z}S^{2n+1}$, we will denote by $w^{\perp}$ the unique element in $z^{\perp}$ such that $w-w^{\perp}\in \mathbb{R}iz$. Clearly, 
	$w^{\perp}=w+\langle w,z\rangle z$. It follows from the facts above that
	\begin{eqnarray}\label{nfwknwkenkwnk}
		(g_{c})_{[z]}(u_{\mathbb{P}_{n}(c)},v_{\mathbb{P}_{n}(c)})=\dfrac{4}{c}\,\textup{Real}\Big\langle \Big(u_{\mathbb{C}^{n+1}-\{0\}}(z)\Big)^{\perp},
		\Big(v_{\mathbb{C}^{n+1}-\{0\}}(z)\Big)^{\perp}\Big\rangle, 
	\end{eqnarray}
	where $z\in S^{2n+1}$ and $u,v\in \mathbb{R}^{n}=\textup{Lie}(\mathbb{T}^{n})$. A direct calculation using \eqref{nfwknwkenkwnk} then shows that 
	\begin{eqnarray*}
		(g_{c})_{[z]}(u_{\mathbb{P}_{n}(c)},v_{\mathbb{P}_{n}(c)})=\dfrac{(4\pi)^{2}}{c}\,\bigg[\sum_{k=1}^{n}u_{k}v_{k}|z_{k}|^{2}-
		\bigg(\sum_{a=1}^{n}u_{a}|z_{a}|^{2}\bigg)\bigg(\sum_{b=1}^{n}v_{b}|z_{b}|^{2}\bigg)\bigg],
	\end{eqnarray*}
	where $z\in S^{2n+1}$. Taking $[z]=\varphi_{1}^{-Jx_{\mathbb{P}_{n}(c)}}(p) =[e^{2\pi x_{1}},...,e^{2\pi x_{n}},1]$ and normalizing 
	appropriately, one finds the matrix representation $[(h_{c})_{ij}]$ of $h_{c}$ at $x\in \mathbb{R}^{n}$: 
	\begin{eqnarray*}
		(h_{c})_{ij}= \dfrac{(4\pi)^{2}}{c}\,\bigg(\dfrac{\delta_{ij}e^{4\pi x_{i}}}{1+e^{4\pi x_{1}}+...+e^{4\pi x_{n}}}
		-\dfrac{e^{4\pi(x_{i}+x_{j})}}{(1+e^{4\pi x_{1}}+...+e^{4\pi x_{n}})^{2}}\bigg). 
	\end{eqnarray*}
	It is the Hessian of the function $\psi_{c}(x)=\tfrac{1}{c}\ln\big(1+e^{4\pi x_{1}}+...+e^{4\pi x_{n}}\big)$. 
	A simple verification shows that the image of $-\textup{grad}(\psi_{c})$ is $-\tfrac{4\pi}{c}S_{n}=\moment_{c}(\mathbb{P}_{n}(c)^{\circ})$ 
	and hence $\psi_{c}$ is compatible with $\moment_{c}$ (see Lemma \ref{nfekwdnkenfkfnk}). The dual of $\psi_{c}$ is obtained by a direct 
	calculation. 
\end{proof}

\section{Examples from Information Theory}\label{nfknwkdenknwkn}

	For an introduction to information geometry, see for example \cite{Jost2,Amari-Nagaoka,Murray}.
\begin{definition}\label{def:5.1}
	A \textit{statistical manifold} is a pair $(S,j)$, where $S$ is a manifold and where $j$ 
	is an injective map from $S$ to the space of all probability density functions $p$ 
	defined on a fixed measure space $(\Omega,dx)$: 
	\begin{eqnarray*}
		j:S\hookrightarrow\Bigl\{ p:\Omega\to\mathbb{R}\;\Bigl|\; p\:\textup{is measurable, }p\geq 
		0\textup{ and }\int_{\Omega}p(x)dx=1\Bigr\}.
	\end{eqnarray*}
\end{definition}

	If $\xi=(\xi_{1},...,\xi_{n})$ is a coordinate system on a statistical manifold $S$, then we shall 
	indistinctly write $p(x;\xi)$ or $p_{\xi}(x)$ for the probability density function determined by $\xi$.

	Given a ``reasonable" statistical manifold $S$, it is possible to define a metric $h_{F}$ and a 
	family of connections $\nabla^{(\alpha)}$ on $S$ $(\alpha\in\mathbb{R})$ in the following way: 
	for a chart $\xi=(\xi_{1},...,\xi_{n})$ of $S$, define
	\begin{alignat*}{1}
		\bigl(h_F\bigr)_{\xi}\bigl(\partial_{i},\partial_{j}\bigr) & :=
		\mathbb{E}_{p_{\xi}}\bigl(\partial_{i}\ln\bigl(p_{\xi}\bigr)\cdotp\partial_{j}\ln\bigl(p_{\xi}\bigr)\bigr),
		\nonumber\\
		\Gamma_{ij,k}^{(\alpha)}\bigl(\xi\bigr) & :=
		\mathbb{E}_{p_{\xi}}\bigl[\bigl(\partial_{i}\partial_{j}\ln\bigl(p_{\xi}\bigr)
		+\tfrac{1-\alpha}{2}\partial_{i}\ln\bigl(p_{\xi}\bigr)\cdotp\partial_{j}\ln\bigl(p_{\xi}\bigr)\bigr),
		\partial_{k}\ln\bigl(p_{\xi}\bigr)\bigr],\label{eq:48}\nonumber
	\end{alignat*}
	where $\mathbb{E}_{p_{\xi}}$ denotes the mean, or expectation, with respect to the probability 
	$p_{\xi}dx$, and where $\partial_{i}$ is a shorthand for $\tfrac{\partial}{\partial\xi_{i}}$. 
	It can be shown that if the above expressions are defined and smooth for every chart of $S$, 
	then $h_F$ is a well defined metric on $S$ called the \textit{Fisher metric}, and that the 
	$\Gamma_{ij,k}^{(\alpha)}$'s define a connection $\nabla^{(\alpha)}$ via the formula 
	$\Gamma_{ij,k}^{(\alpha)}(\xi)=(h_F)_{\xi}(\nabla_{\partial_{i}}^{(\alpha)}\partial_{i},\partial_{j})$, 
	which is called the \textit{$\alpha$-connection}.

	Among the $\alpha$-connections, the $(\pm1)$-connections are particularly important; the 1-connection is
	usually referred to as the \textit{exponential connection}, also denoted by $\nabla^{(e)}$, while the 
	$(−1)$-connection is referred to as the \textit{mixture connection}, denoted by $\nabla^{(m)}$.

	In this paper, we will only consider statistical manifolds $S$ for which the Fisher metric and
	$\alpha$-connections are well defined.

\begin{proposition}\label{prop:5.2}
	Let $S$ be a statistical manifold. Then, $(h_F,\nabla^{(\alpha)},\nabla^{(-\alpha)})$ is a dualistic structure on $S$. 
	In particular, $\nabla^{(-\alpha)}$ is the dual connection of $\nabla^{(\alpha)}$.
\end{proposition}
\begin{proof}
	See \cite{Amari-Nagaoka}.
\end{proof}

	We now recall the definition of an exponential family. 
\begin{definition}\label{def:5.3} 
	An \textit{exponential family} $\mathcal{E}$ on a measure space $(\Omega,dx)$ is a set of probability 
	density functions $p(x;\theta)$ of the form
	\begin{eqnarray*}
		p(x;\theta)=\exp\biggl\{ C(x)+\sum_{i=1}^{n}\theta_{i}F_{i}(x)-\psi(\theta)\biggr\},\label{eq:49}
	\end{eqnarray*}
	where $C,F_1...,F_n$ are measurable functions on $\Omega$, $\theta=(\theta_{1},...,\theta_{n})$ 
	is a vector varying in an open subset $\Theta$ of $\mathbb{R}^{n}$ and where $\psi$ is a function defined on $\Theta$.
\end{definition}
	In the above definition it is assumed that the family of functions $\{1,F_1,...,$ $F_n\}$ 
	is linearly independent, so that the map $p(x,\theta)\mapsto\theta$ becomes a bijection, hence defining a global 
	chart for $\mathcal{E}$. The parameters $\theta_{1},...,\theta_{n}$ are called the 
	\textit{natural} or \textit{canonical parameters} of the exponential family $\mathcal{E}$. 
	The function $\psi$ is called \textit{cumulant generating function}. 

\begin{example}[\textbf{Poisson distribution}]\label{nfeknwdkenfknk}
	A Poisson distribution is a distribution over $\Omega=\mathbb{N}=\{0,1,...\}$ of the form 
	\begin{eqnarray*}
		 p(k;\lambda)=e^{-\lambda}\dfrac{\lambda^{k}}{k!}, 
	\end{eqnarray*}
	where $k\in \mathbb{N}$ and $\lambda>0$. Let $\mathscr{P}$ denote the set of all Poisson distributions 
	$p(\,.\,;\lambda)$, $\lambda>0$. The set $\mathscr{P}$ is an exponential family, because 
	$p(k,\lambda)=\textup{exp}\big(C(k)+F(k)\theta-\psi(\theta)\big),$ where 
	\begin{eqnarray*}
		 C(k)=-\ln(k!), \,\,\,\,\,\,\,F(k)=k, \,\,\,\,\,\,\,\theta=\ln(\lambda), \,\,\,\,\,\,\,
		\psi(\theta)=\lambda=e^{\theta}.
	\end{eqnarray*}
\end{example}

\begin{example}[\textbf{Categorical distribution}]\label{exa:5.5}
	Given a finite set $\Omega=\{ x_{1},...,x_{n}\}$, define
 	\begin{eqnarray*}
		\mathcal{P}_{n}^{\times}\:=\biggl\{ p:\Omega\to\mathbb{R}\:\Bigl|\: p(x)>0 
		\textup{ for all }x\in\Omega\textup{ and }\sum_{k=1}^{n}p(x_{k})=1\biggr\}.
	\end{eqnarray*}
	Elements of $\mathcal{P}_{n}^{\times}$ can be parametrized as follows: 
	$p(x;\theta)=\exp\big\{\sum_{i=1}^{n-1}\theta_{i}F_{i}(x)-\psi(\theta)\big\}$,
	where 
	\begin{eqnarray*}
		&&\theta=(\theta_{1},...,\theta_{n-1})\in\mathbb{R}^{n-1},\,\,\,\,\,\,
		F_{i}(x_{j})=\delta_{ij}\,\,\,\,\,\,\textup{(Kronecker delta)},\\
		&& \psi(\theta)=\ln\bigg(1+\sum_{i=1}^{n-1}e^{\theta_{i}}\bigg). 
	\end{eqnarray*}
	Therefore $\mathcal{P}_{n}^{\times}$ is an exponential family of dimension $n-1$. 
\end{example}
\begin{example}[\textbf{Binomial distribution}]\label{exa:5.6}
	The set of binomial distributions defined over $\Omega:=\{0,...,n\}$,
	\begin{eqnarray*}
		p(k)=\binom{n}{k}q^{k}\bigl(1-q\bigr)^{n-k},\quad (k\;\in\;\Omega,\;q\in(0,1)),
	\end{eqnarray*}
	where $\binom{n}{k}=\tfrac{n!}{(n-k)!k!}$, is a $1$-dimensional statistical manifold, 
	denoted by $\mathcal{B}(n)$, parametrized 
	by $q\in\bigl(0,1\bigr)$. It is an exponential family, because 
	$p(k)=\exp\big\{ C(k)+\theta F(k)-\psi(\theta)\big\}$,
	where
	\begin{eqnarray*}
		&&\theta=\ln\Bigl(\frac{q}{1-q}\Bigr),\quad C\bigl(k\bigr)=
		\ln\binom{n}{k},\quad F(k)= k,\quad\\
		&& \psi(\theta)= n\ln\bigl(1+e^{\theta}\bigr).
	\end{eqnarray*}
\end{example}
\begin{example}[\textbf{Multinomial distribution}]\label{jeknkkdefjdk}
	Let $A_{m,n}=\{(x_{1},...,x_{m})\in \mathbb{R}^{m}\,\,\vert\,\,x_{i}\geq 0\,\,\textup{for all}\, 
	i=1,...,m\,\,\textup{and}\,\,\sum_{i=1}^{m}x_{i}=n\}$. Let $\Omega_{m,n}=A_{m,n}\cap \mathbb{Z}^{m}$. 
	Given $k\in \Omega_{m,n}$ and $\pi\in A_{m,1}$, let 
	\begin{eqnarray*}
		p(k;\pi)=\dfrac{n!}{k_{1}!...k_{m}!}\pi_{1}^{k_{1}}...\, \pi_{m}^{k_{m}},
	\end{eqnarray*}
	where $k=(k_{1},...,k_{m})$ and $\pi=(\pi_{1},...,\pi_{m})$. Let $\mathcal{M}(m,n)$ be the set of all maps 
	$\Omega_{m,n}\to \mathbb{R}, k\mapsto p(k,\pi)$, where $\pi\in A_{m,1}$. Each element in $\mathcal{M}(m,n)$ is 
	called a \textit{multinomial distribution.} They form an exponential familly, because $p(k,\pi)=\textup{exp}(C(k)+\sum_{i=1}^{m-1}F_{i}(k)\theta_{i}-\psi(\theta))$, 
	where 
	\begin{eqnarray*}
		&&\theta_{i}=\ln\Bigl(\frac{\pi_{i}}{1-\sum_{j=1}^{m-1}\pi_{j}}\Bigr),\quad C\bigl(k\bigr)=
		\ln\bigg(\dfrac{n!}{k_{1}!... k_{m}!}\bigg),\quad F_{i}(k)= k_{i},\quad\,\,\,\,\,i=1,...,m-1,\\
		&& \psi(\theta)= n\ln\bigl(1+e^{\theta_{1}}+...+e^{\theta_{m-1}}\bigr).
	\end{eqnarray*}
\end{example}

\begin{example}[\textbf{Negative Binomial distribution}]
	Let $\Omega=\mathbb{N}=\{0,1,....\}$ and $r\in \Omega$, $r\geq 1$. Let $\mathcal{NB}(r)$ denote the set of functions $p:\Omega\to \mathbb{R}$ of the form
	\begin{eqnarray*}
		p(k;q)=\binom{k+r-1}{r-1}(1-q)^{k}q^{r}, \,\,\,\,\,\,\,\,\,\,\,\,\,\,\,\,\,\,(k=0,1,2,...)
	\end{eqnarray*}
	where $q\in (0,1)$ and $\binom{k+r-1}{r-1}=\tfrac{(k+r-1)!}{k!(r-1)!}$. Using the fact that $\tfrac{1}{(1-t)^{r}}=\sum_{k\geq 0}\binom{k+r-1}{r-1}t^{k}$ for $|t|<1$, 
	one sees that $\sum_{k\geq 0}p(k;q)=1$. Each element of $\mathcal{NB}(r)$ is a called a \textit{negative Binomial distribution}.  
	The set $\mathcal{NB}(r)$ is an exponential family, because $p(k;q)=p(k;\theta)=\exp\big\{ C(k)+\theta F(k)-\psi(\theta)\big\}$,
	where
	\begin{eqnarray*}
		&&\theta=\ln(1-q)\in (-\infty,0),\quad C\bigl(k\bigr)=
		\ln\binom{k+r-1}{r-1},\quad F(k)= k,\quad\\
		&& \psi(\theta)= -r\ln\bigl(1-e^{\theta}\bigr).
	\end{eqnarray*}
\end{example}

\begin{proposition}\label{prop:5.7}
	Let $\mathcal{E}$ be an exponential family such as in Definition \ref{def:5.3}. Then 
	$(\mathcal{E},h_F,\nabla^{(e)})$ is a dually flat manifold and the natural parameters 
	$\theta=(\theta_{1},...,\theta_{n})$ form a global $\nabla^{(e)}$-affine coordinate system on $\mathcal{E}$.  
\end{proposition}
\begin{proof}
	See \cite{Amari-Nagaoka}. 
\end{proof}

	Given a smooth function $\psi:\mathbb{R}^{n}\to \mathbb{R}$, we will use the notation $\textup{Hess}(\psi)$ 
	to denote the Hessian $\big[\tfrac{\partial^{2}\psi}{\partial x_{i}\partial x_{j}}\big]$ of $\psi.$ When $\textup{Hess}(\psi)$ 
	is positive definite at each point of $\mathbb{R}^{n}$, we regard $\textup{Hess}(\psi)$ as a Riemannian metric on $\mathbb{R}^{n}$.

	Under mild assumptions (that we will always assume in this paper), it can be shown that the Hessian of 
	the cumulent generating function $\psi:\Theta\to \mathbb{R}$ of an exponential family $\mathcal{E}$ is the coordinate expression 
	for the Fisher metric $h_{F}$ in the natural parameters $\theta_{i}$ (see \cite{Amari-Nagaoka}). 
	Therefore the natural parameters form an isomorphism of dually flat spaces: 
	\begin{eqnarray*}
		(\mathcal{E},h_{F},\nabla^{(e)}) \overset{(\theta_{1},...,\theta_{n})}{\longrightarrow} (\Theta,\textup{Hess}(\psi),\nabla^{\textup{flat}}),
	\end{eqnarray*}
	where $\nabla^{\textup{flat}}$ is the canonical flat connection on $\mathbb{R}^{n}$.

\begin{lemma}\label{nfekwnwkwfnknk}
	Let $\psi:\mathbb{R}^{n}\to \mathbb{R}$ be a smooth function whose Hessian is positive definite at each point of $\mathbb{R}^{n}$ 
	and let $f:\mathbb{R}^{n}\to \mathbb{R}^{n}$ be an invertible affine map (thus $f(x)=Ax+B$, where $A$ is an invertible $n\times n$ matrix and $B\in \mathbb{R}^{n}$). 
	Then 
	\begin{eqnarray*}
		f^{*}\textup{Hess}(\psi)=\textup{Hess}(\psi\circ f). 
	\end{eqnarray*}
\end{lemma}
\begin{proof}
	By a direct calculation.
\end{proof}

	Let $\mathbb{P}_{n}$ and $\mathbb{D}$ denote the complex projective space of complex dimension $n$ and unit disk in $\mathbb{C}$, respectively (both regarded 
	as complex manifolds). We denote by 
	\begin{itemize}
	\item $\mathbb{P}_{n}(c)$ the complex projective space endowed with the Fubini-Study metric normalized in such a way that the holomorphic sectional curvature is $c>0$.
	\item $\mathbb{D}(c)$ the disk $\mathbb{D}$ endowed with the Hyperbolic metric $ds^{2}=-\tfrac{4}{c}\tfrac{dx^{2}+dy^{2}}{(1-x^{2}-y^{2})^{2}}$ 
		of constant holomorphic sectional curvature $c<0$.
	\end{itemize}

	Let $\Phi_{n}$ be the action of $\mathbb{T}^{n}$ on $\mathbb{P}_{n}$ defined by 
	\begin{eqnarray*}
		\Phi_{n}([t],[z])= [e^{2i\pi t_{1}}z_{1},...,e^{2i\pi t_{n}}z_{n},z_{n+1}]. 
	\end{eqnarray*}
\begin{proposition}\label{nfeknwknwknk}
	\textbf{}
	\begin{enumerate}[(1)]
		\item $\mathbb{T}^{1}\times \mathbb{C}\to \mathbb{C}$, $([t],z)\mapsto e^{2i\pi t}z$ is a torification of $\mathscr{Poisson}$. 
		\item $\Phi_{n}:\mathbb{T}^{n}\times \mathbb{P}_{n}(1)\to \mathbb{P}_{n}(1)$ is a torification of $\mathcal{P}_{n+1}^{\times}$.
		\item $\Phi_{1}:\mathbb{T}^{1}\times \mathbb{P}_{1}(\tfrac{1}{n})\to \mathbb{P}_{1}(\tfrac{1}{n})$ is a torification of $\mathcal{B}(n)$.
		\item $\Phi_{m}:\mathbb{T}^{m}\times \mathbb{P}_{m}(\tfrac{1}{n})\to \mathbb{P}_{m}(\tfrac{1}{n})$ is a torification of $\mathcal{M}(m+1,n)$.
		\item $\mathbb{T}^{1}\times \mathbb{D}(-\tfrac{1}{r})\to \mathbb{D}(-\tfrac{1}{r})$, $([t],z)\mapsto e^{2i\pi t}z$ is a torification of $\mathcal{NB}(r)$. 
	\end{enumerate}
\end{proposition}
\begin{proof}
	\noindent (1) Let $\psi_{1},\psi_{2}:\mathbb{R}\to \mathbb{R}$ be defined by $\psi_{1}=\tfrac{1}{4}e^{4\pi x}$ and $\psi_{2}(x)=e^{x}$, respectively. 
	Then $\mathbb{T}^{1}\times \mathbb{C}\to \mathbb{C}$ is a torification of $(\mathbb{R},\textup{Hess}(\psi_{1}),\nabla^{\textup{flat}})$ 
	(see Example \ref{ncenknkenk}) and $(\mathbb{R},\textup{Hess}(\psi_{2}),\nabla^{\textup{flat}})$ is isomorphic to $(\mathscr{P},h_{F},\nabla^{(e)})$ 
	(see Example \ref{nfeknwdkenfknk}). Let $f:\mathbb{R}\to \mathbb{R}$, $x\mapsto 4\pi x-\ln(4)$. Since $\psi_{1}=\psi_{2}\circ f$, 
	Lemma \ref{nfekwnwkwfnknk} implies that $f^{*}\textup{Hess}(\psi_{2})=\textup{Hess}(\psi_{1})$. It follows that $f$ can be regarded as 
	an isomorphism of dually flat spaces from $(\mathbb{R},\textup{Hess}(\psi_{1}),\nabla^{\textup{flat}})$ to $(\mathscr{P},h_{F},\nabla^{(e)})$. 
	By Proposition \ref{ndknqkefnkwnkndk}, $\mathbb{T}^{1}\times \mathbb{C}\to \mathbb{C}$ is a torification of $(\mathscr{P},h_{F},\nabla^{(e)})$. 
	This concludes the proof of (1). 

	\noindent (2)--(4) The proof is entirely analogous: for each exponential family, just compare the cumulant generating function to the potential 
	$\psi_{c}$ described in Proposition \ref{nfknwkdnkdnkn}, and use Lemma \ref{nfekwnwkwfnknk}.

	\noindent (5) Apply Proposition \ref{nkwdnkkfenknk} to the K\"{a}hler covering map $T\mathcal{NB}(r)\cong \{z\in \mathbb{C}\,\,\vert\,\,\textup{Real}(z)<0\}
	\to \mathbb{D}(-\tfrac{1}{r})^{\circ}$, $z\mapsto e^{z/2}$. 
\end{proof}

	In what follows, we will identify the tangent bundle of $\mathcal{E}=\mathscr{P},\mathcal{P}_{n+1}^{\times},\mathcal{B}(n),\mathcal{M}(m+1,n),\mathcal{NB}(r)$ with 
	$\mathbb{C}^{\textup{dim}(\mathcal{E})}$ via the map $T\mathcal{E}\to \mathbb{C}^{\textup{dim}(\mathcal{E})}$, 
	$\sum\dot{\theta}_{k}\tfrac{\partial}{\partial \theta_{k}}\big\vert_{p(.,\theta)}\mapsto (z_{1},...,z_{\textup{dim}(\mathcal{E})})$, where the $\theta_{k}$'s are natural parameters
	and $z_{k}=\theta_{k}+i\dot{\theta}_{k}$.

\begin{proposition}[\textbf{Complement of Proposition \ref{nfeknwknwknk}}]\label{njewndknknknk}
	In each case below, the pair $(\tau,\kappa)$ is a toric factorization of the indicated exponential family  $\mathcal{E}$. 
	\begin{enumerate}
		\item $\mathcal{E}=\mathscr{P},$ \,\,\,\,\,\,  $\mathbb{C}=T\mathscr{P}\overset{\tau}{\longrightarrow} \mathbb{C}^{*}\overset{\kappa}{\longrightarrow} \mathscr{P}$,
			\begin{eqnarray*}
				\tau(z)	= 2e^{z/2}, \,\,\,\,\,\,\,\,\,\,\,\,\,\,\,\,\,\,\,\,
				\kappa(z)(k)=e^{-|z|^{2}/4}\dfrac{1}{k!}\bigg(\dfrac{|z|^{2}}{4}\bigg)^{k},\,\,\,k=0,1,2...  
			\end{eqnarray*}
		\item $\mathcal{E}=\mathcal{P}_{n+1}^{\times},$ \,\,\,\,\,\, $\mathbb{C}^{n}=T\mathcal{P}_{n+1}^{\times}\overset{\tau}{\longrightarrow} \mathbb{P}_{n}(1)^{\circ}
			\overset{\kappa}{\longrightarrow} \mathcal{P}_{n+1}^{\times},$
			\begin{eqnarray*}
				\tau(z)=\big[e^{z_{1}/2},...,e^{z_{n}/2},1\,\big], 
				\,\,\,\,\,\,\,\,\,\,\,\,\,\,\,\,\,\,\,\,\kappa([z])(x_{k})=\dfrac{|z_{k}|^{2}}{|z_{1}|^{2}+...+|z_{n+1}|^{2}},\,\,\,k=1,2...,n+1.
			\end{eqnarray*}
		\item $\mathcal{E}=\mathcal{B}(n),$ \,\,\,\,\,\, $\mathbb{C}=T\mathcal{B}(n)\overset{\tau}{\longrightarrow} \mathbb{P}_{1}(\tfrac{1}{n})^{\circ}
			\overset{\kappa}{\longrightarrow}\mathcal{B}(n),$
			\begin{eqnarray*}
				\tau(z)	= \big[e^{z/2},1\big], \,\,\,\,\,\,\,\,\,\,\,\,\,\,\,\,\,\,\,\,
				\kappa([z_{1},z_{2}])(k)=\binom{n}{k}\dfrac{|z_{1}^{k}z_{2}^{n-k}|^{2}}{(|z_{1}|^{2}+|z_{2}|^{2})^{n}},\,\,\,k=0,1,2...,n.
			\end{eqnarray*}
		\item $\mathcal{E}=\mathcal{M}(m+1,n),$ \,\,\,\,\,\, $\mathbb{C}^{m}=T\mathcal{M}(m+1,n)\overset{\tau}{\longrightarrow}\mathbb{P}_{m}(\tfrac{1}{n})^{\circ}
			\overset{\kappa}{\longrightarrow}\mathcal{M}(m+1,n),$
			\begin{eqnarray*}
				\tau(z)=\big[e^{z_{1}/2},...,e^{z_{m}/2},1\,\big], 
				\,\,\,\,\,\,\,\,\,\,\,\,\,\,\,\,\,\,\,\,
				\kappa([z])(k)=\left\lbrace
				\begin{array}{lll}
				& \dfrac{n!}{k_{1}!...k_{m+1}!}\,\dfrac{|z_{1}^{k_{1}}...\,z_{m+1}^{k_{m+1}}|^{2}}{(|z_{1}|^{2}+...+|z_{m+1}|^{2})^{n}},\\[1.3em]
				& k=(k_{1},...,k_{m+1})\in \mathbb{N}^{m+1},\\[0.5em] 
				& k_{1}+...+k_{m+1}=n.
				\end{array}
				\right.
			\end{eqnarray*}
		\item $\mathcal{E}=\mathcal{NB}(r),$ \,\,\,\,\,\,  $\mathbb{C}=T\mathcal{NB}(r)\overset{\tau}{\longrightarrow} \mathbb{D}(-\tfrac{1}{r})^{\circ}
			\overset{\kappa}{\longrightarrow} \mathcal{NB}(r)$,
			\begin{eqnarray*}
				\tau(z)	= e^{z/2}, \,\,\,\,\,\,\,\,\,\,\,\,\,\,\,\,\,\,\,\,
				\kappa(z)(k)=\binom{k+r-1}{r-1}|z|^{2k}(1-|z|^{2})^{r},\,\,\,k=0,1,2...  
			\end{eqnarray*}

	\end{enumerate}
\end{proposition}
\begin{proof}[Sketch of proof]
	(1)  Let $\Phi:\mathbb{T}^{1}\times \mathbb{C}\to \mathbb{C}$, $([t],z)\mapsto e^{2i\pi t}z$. Recall that the dually flat manifold associated to $\Phi$ and the point 
	$p=1\in \mathbb{C}$ is $(\mathbb{R},\textup{Hess}(\psi_{1}),\nabla^{\textup{flat}})$, where $\psi_{1}:\mathbb{R}\to \mathbb{R}$, $x\mapsto\tfrac{1}{4}e^{4\pi x}$ 
	(see Example \ref{ncenknkenk}). Let $\Phi^{\mathbb{C}}:\mathbb{C}^{*}\times \mathbb{C}\to \mathbb{C}$ be the holomorphic extension of $\Phi$. A direct 
	calculation shows that $\Phi^{\mathbb{C}}(z,w)=zw$. Then, using Proposition \ref{ncekndwkneknk}, it is easy to check that 
	$\tau:\mathbb{C}\to \mathbb{C}^{*},$ $z\mapsto e^{2 \pi z}$ and $\kappa:\mathbb{C}^{*}\to \mathbb{R}$, $z\mapsto \tfrac{\ln(|z|)}{2\pi}$ form 
	a toric factorization.
	Now let $\psi_{2}:\mathbb{R}\to \mathbb{R},$ $x\mapsto e^{x}$ and $f:\mathbb{R}\to \mathbb{R}$, $x\mapsto 4\pi x -\ln(4)$. 
	As we saw in the proof of Proposition \ref{nfeknwknwknk}, $f$ is an isomorphism of dually flat spaces from $(\mathbb{R}, \textup{Hess}(\psi_{1}),\nabla^{\textup{flat}})$ 
	to $(\mathbb{R},\textup{Hess}(\psi_{2}),\nabla^{\textup{flat}})\cong (\mathscr{P},h_{F},\nabla^{(e)})$. By Proposition \ref{ndknqkefnkwnkndk}, 
	$(\tau',\kappa')=(\tau\circ (f_{*})^{-1},f\circ \kappa)$ is a toric factorization of $(\mathscr{P},h_{F},\nabla^{(e)})$. Under the identification 
	$T\mathscr{P}=\mathbb{C}$, we have $f_{*}z=4\pi z-\ln(4)$ and hence $\tau'(z)=\tau(\tfrac{1}{4\pi}(z+\ln(4)))=2e^{z/2}$ and 
	$\kappa'(z)=(f\circ \kappa)(z)=f(\ln(|z|)/2\pi)=\ln(|z|^{2})-\ln(4)=:\theta$. The corresponding distribution $p(.;\theta)$ 
	is a Poisson distribution with parameter $\lambda=e^{\theta}=e^{\ln(|z|^{2})-\ln(4)}=\tfrac{1}{4}|z|^{2}$ (see Example \ref{nfeknwdkenfknk}). 
	This concludes the proof of (1). The other cases are analogous, or simpler.
\end{proof}

\section{Lifting isometric affine maps}\label{nfwkwnkenkwndk}

	Let $\Phi:\mathbb{T}^{n}\times N\to N$ and $\Phi':\mathbb{T}^{d}\times N'\to N'$ be torifications of the dually flat manifolds 
	$(M,h,\nabla)$ and $(M',h',\nabla')$, respectively. 

\begin{definition}
	Let $f:M\to M'$ and $m:N\to N'$ be smooth maps. We say that $m$ is a \textit{lift of} $f$ 
	if there are compatible covering maps $\tau:TM\to N^{\circ}$ and $\tau':TM'\to (N')^{\circ}$ such that 
	$m\circ \tau= \tau'\circ f_{*}$.  
	In this case, we say that $m$ \textit{is a lift of} $f$ \textit{with respect to} $\tau$ and $\tau'$.  
\end{definition}

	The terminology is motivated by the following lemma.

\begin{lemma}
	Let $\pi=\kappa\circ \tau:TM\to M$ and $\pi'=\kappa'\circ \tau':TM'\to M'$ be toric factorizations. 
	If $m$ is a lift of $f:M\to M'$ with respect to $\tau$ and $\tau'$, then $m(N^{\circ})\subseteq (N')^{\circ}$ and $\kappa'\circ m=f\circ\kappa$.
\end{lemma}
\begin{proof}
	The inclusion $m(N^{\circ})\subseteq (N')^{\circ}$ is an immediate consequence of the formula $m\circ \tau=\tau'\circ f_{*}$. 
	It follows from the same formula and the definition of a toric factorization that 
	\begin{eqnarray*}
		\kappa'\circ m\circ \tau=\kappa'\circ\tau'\circ f_{*} =\pi'\circ f_{*}=f\circ \pi=f\circ \kappa\circ \tau,
	\end{eqnarray*}
	and thus $\kappa'\circ m\circ \tau=f\circ \kappa\circ \tau$. This implies $\kappa'\circ m=f\circ \kappa$. 
\end{proof}
\begin{eqnarray*}
	\begin{tikzcd}
		TM \arrow{rrr}{\displaystyle f_{*}} \arrow[dd,swap,"\displaystyle\pi"] \arrow{rd}{\displaystyle \tau} &    &   & 
			TM' \arrow[swap]{ld}{\displaystyle \tau'} \arrow[dd,"\displaystyle\pi'"]  \\
		& N^{\circ} \arrow{r}{\displaystyle m} \arrow{ld}{\displaystyle \kappa}  
			& (N')^{\circ} \arrow[swap]{dr}{\displaystyle \kappa'}  & \\
		M \arrow[swap]{rrr}{\displaystyle f}  &     && M'
	\end{tikzcd}
\end{eqnarray*}

\begin{proposition}\label{ncekndkvndknk}
	Let $m:N\to N'$ be a lift of $f:M\to M'$. 
	\begin{enumerate}[(1)]
	\item The map  $m$ is a K\"{a}hler immersion if and only if $f$ is an isometric affine immersion. 
	\item If $f$ is an isometric affine immersion, then there exists a unique Lie group homomorphism 
		$\rho:\mathbb{T}^{n}\to \mathbb{T}^{d}$ with finite kernel such that $m\circ \Phi_{a}=\Phi_{\rho(a)}'\circ m$ 
			for all $a\in \mathbb{T}^{n}$.  
	\end{enumerate}
\end{proposition}

	The proof of Proposition \ref{ncekndkvndknk} proceeds in a series of lemmas.
	Let $\pi=\kappa\circ \tau:TM\to M$ and $\pi'=\kappa'\circ \tau':TM'\to M'$ be fixed toric factorizations. 

\begin{lemma}\label{nwfkndkwnk}
	Let $f:M\to M'$ be an isometric affine immersion. Given $i=1,2,$ let $m_{i}:N\to N'$ be a K\"{a}hler immersion
	satisfying $m_{i}(N^{\circ})\subseteq (N')^{\circ}$ and $\kappa'\circ m_{i}=f\circ \kappa$ on $N^{\circ}$. Then 
	there exists a unique $a\in \mathbb{T}^{d}$ such that $m_{1}=\Phi'_{a}\circ m_{2}$.  
\end{lemma}
\begin{proof}
	Because $\kappa'\circ m_{1}=\kappa'\circ m_{2}$ on $N^{\circ}$ and $\kappa':(N')^{\circ}\to M'$ is a trivial principal 
	fiber bundle with respect to the torus action $\Phi'$, there exists a unique smooth function $\phi:N^{\circ}\to \mathbb{T}^{d}$ 
	such that $m_{1}(p)=\Phi'(\phi(p),m_{2}(p))$ for all $p\in N^{\circ}$. Taking the derivative with respect to $p\in N^{\circ}$ 
	in the direction $u\in T_{p}N$ we obtain 
	\begin{eqnarray*}
		(m_{1})_{*_{p}}u=\big(\Phi'_{m_{2}(p)}\big)_{*}\phi_{*_{p}}u+\big(\Phi'_{\phi(p)}\big)_{*}(m_{2})_{*_{p}}u.
	\end{eqnarray*}
	The same formula with $Ju$ in place of $u$ yields 
	\begin{eqnarray}\label{nefnwdkkenk}
		(m_{1})_{*_{p}}Ju=\big(\Phi'_{m_{2}(p)}\big)_{*}\phi_{*_{p}}Ju+J'\big(\Phi'_{\phi(p)}\big)_{*}(m_{2})_{*_{p}}u,
	\end{eqnarray}
	where we have used the fact that $m_{2}$ and $\Phi'_{\phi(p)}$ are holomorphic. Since $m_{1}$ is holomorphic, we also have 
	\begin{eqnarray}\label{nfkenknkwnk}
		(m_{1})_{*_{p}}Ju= J'(m_{1})_{*_{p}}u=J'\big(\Phi'_{m_{2}(p)}\big)_{*}\phi_{*_{p}}u+J'\big(\Phi'_{\phi(p)}\big)_{*}(m_{2})_{*_{p}}u.
	\end{eqnarray}
	Comparing \eqref{nefnwdkkenk} and \eqref{nfkenknkwnk} we get 
	\begin{eqnarray}\label{nefknqkwnkndk}
		\big(\Phi'_{m_{2}(p)}\big)_{*_{\phi(p)}}\phi_{*_{p}}Ju= J'\big(\Phi'_{m_{2}(p)}\big)_{*_{\phi(p)}}\phi_{*_{p}}u.
	\end{eqnarray}
	Given $a\in \mathbb{T}^{d}$, let $L_{a}:\mathbb{T}^{d}\to \mathbb{T}^{d},$ $b\mapsto ab$. 
	Then $\Phi'_{m_{2}(p)}=\Phi'_{\phi(p)}\circ  \Phi'_{m_{2}(p)} \circ L_{\phi(p)^{-1}}$ and hence
	\begin{eqnarray*}	
		\big(\Phi'_{m_{2}(p)}\big)_{*_{\phi(p)}}=(\Phi'_{\phi(p)})_{*_{m_{2}(p)}} (\Phi'_{m_{2}(p)})_{*_{e}}(L_{\phi(p)^{-1}})_{*_{\phi(p)}}.
	\end{eqnarray*}
	Using this, we can rewrite \eqref{nefknqkwnkndk} as
	\begin{eqnarray}
		(\Phi'_{\phi(p)})_{*_{m_{2}(p)}}\big(\Phi'_{m_{2}(p)}\big)_{*_{e}}(L_{\phi(p)^{-1}})_{*_{\phi(p)}}\phi_{*_{p}}Ju= 
		J'(\Phi'_{\phi(p)})_{*_{m_{2}(p)}}\big(\Phi'_{m_{2}(p)}\big)_{*_{e}}(L_{\phi(p)^{-1}})_{*_{\phi(p)}}\phi_{*_{p}}u,
	\end{eqnarray}
	and since $(\Phi'_{\phi(p)})_{*_{m_{2}(p)}}$ is a linear bijection that commutes with $J'$, 
	\begin{eqnarray}\label{dnkwwndknknk}
		\big(\Phi'_{m_{2}(p)}\big)_{*_{e}}(L_{\phi(p)^{-1}})_{*_{\phi(p)}}\phi_{*_{p}}Ju= 
		J'\big(\Phi'_{m_{2}(p)}\big)_{*_{e}}(L_{\phi(p)^{-1}})_{*_{\phi(p)}}\phi_{*_{p}}u. 
	\end{eqnarray}
	Let $\xi:=(L_{\phi(p)^{-1}})_{*_{\phi(p)}}\phi_{*_{p}}Ju$ and $\eta:=(L_{\phi(p)^{-1}})_{*_{\phi(p)}}\phi_{*_{p}}u$. 
	Then $\xi,\eta\in \mathbb{R}^{d}=\textup{Lie}(\mathbb{T}^{d})$. We denote by $\xi_{N'}$ and $\eta_{N'}$ the corresponding fundamental vector fields 
	on $N'$ associated to $\Phi'$. With this notation, \eqref{dnkwwndknknk} can be rewritten as
	\begin{eqnarray*}
		\xi_{N'}(m_{2}(p))=J'\eta_{N'}(m_{2}(p)).
	\end{eqnarray*}
	By Lemma \ref{neekwnkvnkdnsk}, $\xi_{N'}(m_{2}(p))=\eta_{N'}(m_{2}(p))=0$. 
	The action $\Phi'$ being free at $m_{2}(p)\in (N')^{\circ}$, the orbit map $\Phi_{m_{2}(p)}':\mathbb{T}^{d}\to N'$ is immersive at $e$ and hence the equation 
	$0=\eta_{N'}(m_{2}(p))=(\Phi'_{m_{2}(p)})_{*_{e}}\eta$ has a unique solution $\eta=0.$ It follows that $\phi_{*_{p}}u=0$. 
	Since $p\in N^{\circ}$ and $u\in T_{p}N$ are arbitrary and $N^{\circ}$ is connected, $\phi$ is constant on 
	$N^{\circ}$, say $\phi\equiv a\in \mathbb{T}^{d}$. Then $m_{1}=\Phi'_{a}\circ m_{2}$ on $N^{\circ}$. 
	Since $N^{\circ}$ is dense in $N$, $m_{1}=\Phi'_{a}\circ m_{2}$ on $N$. Uniqueness of $a$ is a consequence of the uniqueness of the function $\phi$ 
	(see above). The lemma follows. 
\end{proof}

\begin{lemma}\label{nfenkneknkdnk}
	Let $f:M\to M'$ be an isometric affine immersion and let $m:N\to N'$ be a 
	K\"{a}hler immersion satisfying $m(N^{\circ})\subseteq (N')^{\circ}$ and $\kappa'\circ m=f\circ \kappa$ on $N^{\circ}$.
	Then there exists a unique Lie group homomorphism $\rho:\mathbb{T}^{n}\to \mathbb{T}^{d}$ such that 
	\begin{eqnarray*}
		m\circ \Phi_{a}=\Phi'_{\rho(a)}\circ m
	\end{eqnarray*}
	for all $a\in \mathbb{T}^{n}$. Moreover, the kernel of $\rho$ is finite.
\end{lemma}
\begin{proof}
	Let $a\in \mathbb{T}^{n}$ be arbitrary. The map $m_{a}:=m\circ \Phi_{a}$ is a K\"{a}hler immersion satisfying $m_{a}(N^{\circ})\subseteq (N')^{\circ}$ 
	and $\kappa'\circ m_{a}=f\circ \kappa$. By Lemma \ref{nwfkndkwnk}, there exists a unique $\rho(a)\in \mathbb{T}^{d}$ such that 
	$m_{a}=m\circ \Phi_{a}=\Phi_{\rho(a)}'\circ m$. Let $\rho:\mathbb{T}^{n}\to \mathbb{T}^{d}$, $a\mapsto \rho(a)$. This is a smooth function, 
	as can be seen by trivializing the principal fiber bundles $\kappa:N^{\circ}\to M$ and $\kappa':(N')^{\circ}\to M'$. 
	Using the uniqueness of $\rho$, it is easily seen that $\rho(ab)=\rho(a)\rho(b)$ and $\rho(e)=e$. Therefore $\rho$ is a Lie group homomorphism. Let 
	$K\subseteq \mathbb{T}^{n}$ be the kernel of $\rho$. Suppose that $K$ is not discrete. Then there exists $\xi\in \textup{Lie}(\mathbb{T}^{n})$, $\xi\neq 0$,
	such that $\textup{exp}(t\xi)\in K$ for all $t\in\mathbb{R}$, where $\textup{exp}:\textup{Lie}(\mathbb{T}^{n})\to \mathbb{T}^{n}$ is the exponential map 
	of the torus. Let $p\in N^{\circ}$ be arbitrary and let $\xi_{N}$ denote the fundamental vector field of $\xi$ on $N$. We compute:
	\begin{eqnarray*}
		m_{*_{p}}\xi_{N}(p)=\dfrac{d}{dt}\bigg\vert_{0} (m\circ \Phi_{\textup{exp}(t\xi)})(p)=\dfrac{d}{dt}\bigg\vert_{0} \Phi'(\rho(\textup{exp}(t\xi)),m(p))=
		\dfrac{d}{dt}\bigg\vert_{0}m(p)=0.
	\end{eqnarray*}
	Thus $m_{*_{p}}\xi_{N}(p)=0$. Because $m$ is an immersion, this implies that $\xi_{N}(p)=0$, which means that $\xi$ belongs to the kernel of 
	$(\Phi_{p})_{*_{e}}$. Since $\Phi$ is free at $p$, this kernel is trivial and hence $\xi= 0$, a contradiction. This shows that $K$ is discrete. Since $\mathbb{T}^{n}$ 
	is compact, $K$ is finite.
\end{proof}

\begin{proof}[Proof of Proposition \ref{ncekndkvndknk}]
	(1) By definition, there are compatible covering maps $\tilde{\tau}$ and $\tilde{\tau}'$ such that 
	$m\circ \tilde{\tau}=\tilde{\tau}'\circ f_{*}$. Since compatible covering maps are locally K\"{a}hler isomorphisms, we have the following equivalences
	\begin{eqnarray*}
		f\,\,\,\textup{is an affine isometric map} \,\,\,\,\,\,\,\,\,\,&\Leftrightarrow &\,\,\,\,\,\,\,\,\,\,f_{*}\,\,\,\textup{is a K\"{a}hler immersion} \\
		&\Leftrightarrow &\,\,\,\,\,\,\,\,\,\,m:N^{\circ}\to N' \,\,\,\textup{is a K\"{a}hler immersion},
	\end{eqnarray*}
	where we have used Proposition \ref{nfeknwknekn}. Therefore we must show that $m:N\to N'$ is a K\"{a}hler immersion if and only if $m:N^{\circ}\to N'$ 
	is a K\"{a}hler immersion. One direction is obvious. So assume that $m:N^{\circ}\to N'$ is a K\"{a}hler immersion. 
	To see that $m$ is holomorphic on $N$, let $J$ and $J'$ denote 
	the complex structures on $N$ and $N'$, respectively. The maps $m_{*}\circ J$ and $J'\circ m_{*}$ 
	are continuous on $TN$ and coincide on $TN^{\circ}$ (since $m$ is holomorphic on $N^{\circ}$). Since $TN^{\circ}$ is dense in $TN$ and $TN'$ is Hausdorff, 
	this implies that $m_{*}\circ J$ and $J'\circ m_{*}$ coincide on $TN$, which means that $m$ is holomorphic on $N$. Analogously, one proves that 
	$m:N\to N'$ is isometric. This concludes the proof of (1). (2) This follows from Lemma \ref{nfenkneknkdnk}.  
\end{proof}

	Now we focus our attention on the existence of lifts. Recall that a Riemannian manifold $(M,g)$ is said to be 
	\textit{real analytic} if $M$ is a real analytic manifold and the metric $g$ is a real analytic tensor. 

\begin{theorem}\label{nedwknkdnkn}
	Let $M$	and $M'$ be real analytic Riemannian manifolds. If $M$ is connected and simply connected and if 
	$M'$ is complete, then every isometric immersion $f_{U}$ of a connected open subset $U$ of $M$ 
	into $M'$ can be uniquely extended to an isometric immersion $f$ of $M$ into $M'$. 
\end{theorem}
\begin{proof}
	See \cite{Kobayashi-Nomizu}, Theorem 6.3. 
\end{proof}

	The preceding theorem motivates the following definition.  

\begin{definition}
	A K\"{a}hler manifold $N$ is \textit{regular} if it is connected, simply connected, complete and if the K\"{a}hler metric is real analytic. 
	A torification $\Phi:\mathbb{T}^{n}\times N\to N$ is \textit{regular} if $N$ is regular.
\end{definition}

	Note that the requirement that the K\"{a}hler metric be real analytic makes sense, because complex manifolds are naturally real analytic manifolds\footnote{
	This follows from the following simple observation. 
	If a function $f:U\to \mathbb{C}$ is holomorphic on an open set $U\subseteq \mathbb{C}^{n}\cong \mathbb{R}^{2n}$, 
	then its real and imaginary parts $u:U\to \mathbb{R}$ and $v:U\to \mathbb{R}$ 
	are harmonic functions on $U$ (this follows from the Cauchy-Riemann equations) and hence they are real analytic. }.


\begin{proposition}[\textbf{Existence of lifts}]\label{ncnkwknknk}
	Suppose $\Phi:\mathbb{T}^{n}\times N\to N$ and $\Phi':\mathbb{T}^{d}\times N'\to N'$ are regular torifications of $(M,h,\nabla)$ and $(M',h',\nabla')$, respectively. 
	Let $\tau:TM\to N^{\circ}$ and $\tau':TM'\to (N')^{\circ}$ be compatible covering maps. Then every isometric affine immersion 
	$f:M\to M'$ has a unique lift $m:N\to N'$ with respect to $\tau$ and $\tau'$. 
\end{proposition}
\begin{proof}
%
	Since $\tau$ is an isometric covering map, there is an open cover $\{U_{\alpha}\}_{\alpha\in A}$ of $TM$ such that for every $\alpha\in A$, 
	$U_{\alpha}$ is connected and $\tau\vert_{U_{\alpha}}:U_{\alpha}\to V_{\alpha}:=\tau(U_{\alpha})$ is an isometry. 
	Given $\alpha\in A$, let $m_{\alpha}:V_{\alpha}\to N'$ be defined by 
	\begin{eqnarray*}
		m_{\alpha}=\tau'\circ f_{*} \circ (\tau\vert_{U_{\alpha}})^{-1}.
	\end{eqnarray*}

	Then $m_{\alpha}$ is an isometric immersion from $V_{\alpha}\subseteq N$ into $N'$, because 
	$f_{*}$ is a K\"{a}hler immersion (Proposition \ref{nfeknwknekn}). By Theorem \ref{nedwknkdnkn}, $m_{\alpha}$ extends uniquely to an isometric immersion 
	$\widetilde{m_{\alpha}}:N\to N'$. 

	Let $\alpha,\beta\in A$ be arbitrary. We claim that $\widetilde{m_{\alpha}}=\widetilde{m_{\beta}}$. Indeed, since $TM$ is connected, 
	it is well-chained and hence there exist an integer $s\geq 1$ and indices $\alpha_{0},...,\alpha_{s}\in A$ 
	such that $\alpha_{0}=\alpha$, $\alpha_{s}=\beta$ and $U_{\alpha_{i-1}}\cap U_{\alpha_{i}}\neq \emptyset$ for all $i=1,...,s$. 
	For each $i=1,...,s,$ let $W_{i}$ be an open connected subset of $U_{\alpha_{i-1}}\cap U_{\alpha_{i}}$. Clearly, $\tau(W_{i})\subseteq V_{i-1}\cap V_{i}$ and 
	$(\tau\vert_{U_{\alpha_{i-1}}})^{-1}=(\tau\vert_{U_{\alpha_{i}}})^{-1}$ on $\tau(W_{i})$. 
	It follows that $m_{\alpha_{i-1}}$ and $m_{\alpha_{i}}$ coincide on the connected set $\tau(W_{i})$. By the uniqueness part of Theorem \ref{nedwknkdnkn}, 
	their global extensions coincide on $N$, that is, $\widetilde{m_{\alpha_{i-1}}}=\widetilde{m_{\alpha_{i}}}$. But then 
	$\widetilde{m_{\alpha}}=\widetilde{m_{\alpha_{0}}}=\widetilde{m_{\alpha_{1}}}=...=\widetilde{m_{\alpha_{s}}}=\widetilde{m_{\beta}}$ and hence 
	$\widetilde{m_{\alpha}}=\widetilde{m_{\beta}}$. This concludes the proof of the claim. 

	Set $m=\widetilde{m_{\alpha}}$, where $\alpha$ is any element in $A$ (this is independent of the choice of 
	$\alpha$ by virtue of the claim). By construction, $m$ is an isometric immersion from $N$ into $N'$ satisfying 
	$m\circ \tau=\tau'\circ f_{*}$. In particular, $m$ is a lift of $f$ with respect to $\tau$ and $\tau'$. 

	Uniqueness of $m$ follows from the formula $m\circ \tau=\tau'\circ f_{*}$ and the fact that $N^{\circ}$ is dense in $N$.  
\end{proof}

	As a matter of notation, we will denote by $m_{\tau',\tau}(f)$ the unique lift of the isometric affine immersion $f:M\to M'$ with respect to 
	$\tau:TM\to N^{\circ}$ and $\tau':TM'\to (N')^{\circ}$,
	and by $\rho_{\tau',\tau}:\mathbb{T}^{n}\to \mathbb{T}^{d}$ the corresponding Lie group homomorphism (see Proposition \ref{ncekndkvndknk}). 
\begin{proposition}\label{cnekdwnkndknk}
	Given $i=1,2,3$, let $\Phi_{i}:\mathbb{T}^{n_{i}}\times N_{i}\to N_{i}$ be a regular torification of $(M_{i},h_{i},\nabla_{i})$ with compatible 
	covering map $\tau_{i}:TM_{i}\to N_{i}^{\circ}$. Let $f_{1}:M_{1}\to M_{2}$ and $f_{2}:M_{2}\to M_{3}$ be isometric affine immersions.
	\begin{enumerate}[(1)]
		\item $m_{\tau_{3}\tau_{1}}(f_{2}\circ f_{1})=m_{\tau_{3}\tau_{2}}(f_{2})\circ m_{\tau_{2}\tau_{1}}(f_{1})$. 
		\item $m_{\tau_{1}\tau_{1}}(Id_{M_{1}})=Id_{N_{1}}$, where $Id_{M_{1}}$ and $Id_{N_{1}}$ are the identity maps on $M_{1}$ and $N_{1}$, 
			respectively.
		\item $\rho_{\tau_{3}\tau_{1}}(f_{2}\circ f_{1})=\rho_{\tau_{3}\tau_{2}}(f_{2})\circ \rho_{\tau_{2}\tau_{1}}(f_{1})$.
		\item $\rho_{\tau_{1}\tau_{1}}(Id_{M_{1}})=Id_{\mathbb{T}^{n_{1}}}$, where $Id_{\mathbb{T}^{n_{1}}}$ is the identity map on $\mathbb{T}^{n_{1}}$.
	\end{enumerate}
\end{proposition}
\begin{proof}
	For simplicity, write $m_{1}=m_{\tau_{2}\tau_{1}}(f_{1})$, $m_{2}=m_{\tau_{3}\tau_{2}}(f_{2})$, $\rho_{1}=\rho_{\tau_{2}\tau_{1}}(f_{1})$ and 
	$\rho_{2}=\rho_{\tau_{3}\tau_{2}}(f_{2})$. 

	\noindent (1) Taking into account the definition of a lift, we compute:
	\begin{eqnarray*}
		(m_{2} \circ m_{1})\circ \tau_{1}=m_{2}\circ \tau_{2}\circ (f_{1})_{*}=\tau_{3}\circ (f_{2})_{*}\circ (f_{1})_{*}=\tau_{3}\circ (f_{2}\circ f_{1})_{*}.
	\end{eqnarray*}
	This shows that $m_{2}\circ m_{1}$ is a lift of $f_{2}\circ f_{1}$ with respect to $\tau_{1}$ and $\tau_{3}$. 

	\noindent (2) Since $(Id_{M_{1}})_{*}=Id_{TM_{1}}$, we have $Id_{N_{1}}\circ \tau_{1}= \tau_{1}\circ (Id_{M_{1}})_{*}$. This shows that $Id_{N_{1}}$ is a lift 
	of $Id_{M_{1}}$ with respect to $\tau_{1}$ and $\tau_{1}$. 

	\noindent (3) Let $a\in \mathbb{T}^{n_{1}}$ be arbitrary. Taking into account the definition of $\rho_{i}$, we compute:
	\begin{eqnarray*}
		m_{\tau_{3}\tau_{1}}\circ (\Phi_{1})_{a}&=&  m_{2}\circ m_{1} \circ (\Phi_{1})_{a}=m_{2}\circ (\Phi_{2})_{\rho_{1}(a)}\circ m_{1}\\
		&=& (\Phi_{3})_{\rho_{2}(\rho_{1}(a))}\circ m_{2}\circ m_{1}=(\Phi_{3})_{\rho_{2}(\rho_{1}(a))}\circ m_{\tau_{3}\tau_{1}}. 
	\end{eqnarray*}
	This shows that $\rho_{\tau_{3}\tau_{1}}(f_{2}\circ f_{1})=\rho_{2}\circ \rho_{1}$. 

	\noindent (4) Given $a\in \mathbb{T}^{n_{1}}$, we have 
	\begin{eqnarray*}
		m_{\tau_{1}\tau_{1}}(Id_{M_{1}})\circ (\Phi_{1})_{a}=Id_{N_{1}}\circ (\Phi_{1})_{a}=(\Phi_{1})_{a}\circ Id_{N_{1}}=
		(\Phi_{1})_{a}\circ m_{\tau_{1}\tau_{1}}(Id_{M_{1}}).
	\end{eqnarray*}
	The result follows.  
\end{proof}

\begin{definition}
	Suppose $\Phi:\mathbb{T}^{n}\times N\to N$ and $\Phi':\mathbb{T}^{d}\times N'\to N'$ are torifications of $(M,h,\nabla)$ and $(M',h',\nabla')$, respectively. 
	We shall say that $N$ and $N'$ are \textit{equivalent} if there exist a Lie group isomorphism $\rho:\mathbb{T}^{n}\to \mathbb{T}^{d}$ 
	and a K\"{a}hler isomorphism $G:N\to N'$ such that 
	\begin{eqnarray*}
		G\circ \Phi_{a}=\Phi'_{\rho(a)}\circ G
	\end{eqnarray*}
	for all $a\in \mathbb{T}^{n}$. In this case, we write $N\sim N'$. 
\end{definition}

	\noindent Note that $\sim$ is an equivalence relation. 

\begin{theorem}\label{aaanfeknkenfknk}
	Regular torifications of the dually flat space $(M,h,\nabla)$ are equivalent. 
\end{theorem}
\begin{proof}
	Let $\Phi:\mathbb{T}^{n}\times N\to N$ and $\Phi':\mathbb{T}^{d}\times N'\to N'$ be regular torifications of $(M,h,\nabla)$. Let 
	$\tau:TM\to N^{\circ}$ and $\tau':TM'\to (N')^{\circ}$ be compatible covering maps. By Proposition \ref{cnekdwnkndknk}, $m=m_{\tau'\tau}(Id_{M}): N\to N'$ 
	is a K\"{a}hler isomorphism satisfying $m\circ \Phi_{a}=\Phi_{\rho(a)}\circ m$ for all $a\in \mathbb{T}^{n}$, where $\rho=\rho_{\tau'\tau}$. By 
	Proposition \ref{cnekdwnkndknk}, $\rho$ is a Lie group isomorphism. 
\end{proof}

\begin{definition}
	We shall say that a connected dually flat space $(M,h,\nabla)$ is \textit{toric} if it has a regular torification. 
\end{definition}

\begin{example}
	The torifications considered in Proposition \ref{nfeknwknwknk} are all regular. Therefore the exponential families 
	$\mathscr{P}$, $\mathcal{P}_{n}^{\times}$, $\mathcal{B}(n)$, $\mathcal{M}(m,n)$, $\mathcal{NB}(r)$ are toric.
\end{example}

	From now on, when we deal with a regular torification of a dually flat space $(M,h,\nabla)$, we will say ``the regular torification of $M$", 
	and keep in mind that it is only defined up to an equivariant K\"{a}hler isomorphism.\\

	Now we continue with some consequences of Proposition \ref{cnekdwnkndknk}. 
	Let $N$ be a K\"{a}hler manifold equipped with a torus action $\Phi:\mathbb{T}^{n}\times N\to N$, with K\"{a}hler metric $g$. We will use the following 
	notation:
	\begin{itemize}
		\item $\textup{Aut}(\mathbb{T}^{n})$ is the group of Lie group isomorphisms $\rho:\mathbb{T}^{n}\to \mathbb{T}^{n}$. 
		\item $\textup{Aut}(N,g)$ is the group of holomorphic and isometric transformations of $N$.
		\item $\textup{Aut}(N,g)^{\mathbb{T}^{n}}$ is the subgroup of $\textup{Aut}(N,g)$ characterized by the following condition: 
			for each $\varphi\in \textup{Aut}(N,g)^{\mathbb{T}^{n}}$, there exists $\rho\in \textup{Aut}(\mathbb{T}^{n})$ such that 
			$\varphi\circ \Phi_{a}=\Phi_{\rho(a)}\circ \varphi$ for all $a\in \mathbb{T}^{n}$. 
	\end{itemize}	
\begin{remark}\label{nfekdnkwfnwknk}
	With the usual identification $\mathbb{R}^{n}=\textup{Lie}(\mathbb{T}^{n})$, the exponential map $\textup{exp}:\mathbb{R}^{n}\to \mathbb{T}^{n}$ is 
	just the quotient map $t\mapsto [t]$. Therefore, if $\rho\in \textup{Aut}(\mathbb{T}^{n})$, then $\rho([t])=[\rho_{*_{e}}t]$ for all $t\in \mathbb{R}^{n}$. 
	This forces the matrix representation of $\rho_{*_{e}}$ with respect to the canonical basis 
	to be an element of $\textup{GL}(n,\mathbb{Z})$. It is then easy to see that $\textup{Aut}(\mathbb{T}^{n})\to \textup{GL}(n,\mathbb{Z})$, 
	$\rho \mapsto \rho_{*_{e}}$ is a group isomorphism.
\end{remark}
\begin{lemma}\label{efnkfenknfknk}
	Let $\Phi:\mathbb{T}^{n}\times N\to N$ be a regular torification of $(M,h,\nabla)$, with toric factorization $\pi=\kappa\circ \tau$. Let 
	$g$ be the K\"{a}hler metric on $N$. Two maps $\tau':TM\to N^{\circ}$ and $\kappa':N^{\circ}\to M$ form a toric factorization 
	if and only if there exists $G\in \textup{Aut}(N,g)^{\mathbb{T}^{n}}$ such that $\tau'=G\circ \tau$ and $\kappa'=\kappa\circ G^{-1}$. 
\end{lemma}
\begin{proof}
	$(\Rightarrow)$ Suppose $\pi=\kappa\circ \tau=\kappa'\circ \tau'$ are toric factorizations. Let $G:=m_{\tau'\tau}(Id_{M})$. 
	By Proposition \ref{cnekdwnkndknk}, $G\in \textup{Aut}(N,g)^{\mathbb{T}^{n}}$ and $G\circ \tau=\tau'$. 
	To see that $\kappa'=\kappa\circ G^{-1}$, let $y=\tau'(x)$ be arbitrary. We have 
	\begin{eqnarray}
		\kappa'(y)=(\kappa'\circ \tau')(x)=\pi(x)=(\kappa\circ \tau)(x)=\kappa(G^{-1}(\tau'(x)))=(\kappa\circ G^{-1})(y).
	\end{eqnarray}
	$(\Leftarrow)$ Let $G\in \textup{Aut}(N,g)^{\mathbb{T}^{n}}$ be arbitrary. There exists $A\in \textup{GL}(n,\mathbb{Z})$ such that 
	$G\circ \Phi_{[t]}=\Phi_{[At]}\circ G$ for all $t\in \mathbb{R}^{n}$ (see Remark \ref{nfekdnkwfnwknk}). 
	The fact that $\pi=\kappa\circ \tau$ is a toric factorization means that there is a toric parametrization $(L,X,F)$ such that $\tau=F\circ q_{L}$ and 
	$\kappa=\pi_{L}\circ F^{-1}.$ Let $X'\in \textup{gen}(L)$ be defined by $X'=(A^{T})^{-1}X$. Using Lemma \ref{nkendwkddknk}, it is easy to check that 
	$(L,X',G\circ F)$ is a toric parametrization and that the induced toric factorization is $(\tau',\kappa')$, where $\tau'=G\circ \tau$ and $\kappa'=\kappa\circ G^{-1}$.  
\end{proof}

	In the next two propositions, $\Phi:\mathbb{T}^{n}\times N\to N$ and $\Phi':\mathbb{T}^{d}\times N'\to N'$ are regular torifications of the dually flat spaces
	$(M,h,\nabla)$ and $(M',h',\nabla')$, respectively. We denote by $g$ and $g'$ the K\"{a}lher metrics on $N$ and $N'$, respectively.

\begin{proposition}[\textbf{Lifts are conjugate}]\label{neknmfkrnknk}
	Let $m:N\to N'$ be a lift of the isometric affine immersion $f:M\to M'$. Given a smooth map $m':N\to N'$, the following are equivalent:
	\textbf{}
	\begin{enumerate}[(1)]
		\item $m'$ is a lift of $f$. 
		\item There are $G_{1}\in \textup{Aut}(N,g)^{\mathbb{T}^{n}}$ and 
			$G_{2}\in \textup{Aut}(N',g')^{\mathbb{T}^{d}}$ such that $m'\circ G_{1}=G_{2}\circ m$. 
	\end{enumerate}
\end{proposition}
\begin{proof}
	$(1)\Rightarrow (2)$ Suppose $m,m':N\to N'$ are lifts of $f$. 
	This means that there are compatible covering maps $\tau_{1},\tau_{2}:TM\to N^{\circ}$ and $\tau_{1}',\tau_{2}':TM'\to (N')^{\circ}$ such that 
	$m\circ \tau_{1}=\tau_{1}'\circ f_{*}$ and $m'\circ \tau_{2}=\tau_{2}'\circ f_{*}$. Let $G_{1}=m_{\tau_{2}\tau_{1}}(Id_{M})$ and $G_{2}=m_{\tau_{2}'\tau_{1}'}(Id_{M'})$. 
	By Proposition \ref{cnekdwnkndknk}, $G_{1}\in \textup{Aut}(N,g)^{\mathbb{T}^{n}}$, $G_{2}\in \textup{Aut}(N',g')^{\mathbb{T}^{d}}$ and 
	$G_{1}\circ \tau_{1}=\tau_{2}$ and $G_{2}\circ \tau_{1}'=\tau_{2}'$. Thus
	\begin{eqnarray*}
		m'\circ \tau_{2}=\tau_{2}'\circ f_{*}  \,\,\,\,\,\,\,&\Rightarrow &\,\,\,\,\,\,\,m'\circ G_{1}\circ \tau_{1}= G_{2}\circ \tau_{1}' \circ f_{*} \\
		&\Rightarrow &\,\,\,\,\,\,\,m'\circ G_{1}\circ \tau_{1}= G_{2}\circ m\circ \tau_{1}\\
		&\Rightarrow &\,\,\,\,\,\,\,m'\circ G_{1}= G_{2}\circ m \,\,\,\textup{on}\,\,\,N^{\circ}.
	\end{eqnarray*}
	Since $N^{\circ}$ is dense in $N$, $m'\circ G_{1}= G_{2}\circ m$ on $N$. 

	$(2)\Rightarrow (1)$. Suppose that $m:N\to N'$ is a lift of $f$ and that $m':N\to N'$ is a smooth map 
	satisfying $m'\circ G_{1}=G_{2}\circ m$ for some $G_{1}\in \textup{Aut}(N,g)^{\mathbb{T}^{n}}$ and some $G_{2}\in \textup{Aut}(N',g')^{\mathbb{T}^{d}}$.
	Because $m$ is a lift, there are compatible projections $\tau:TM\to N^{\circ}$ and $\tau':TM'\to (N')^{\circ}$ such that $m\circ \tau=\tau'\circ f_{*}$, and so
	\begin{eqnarray}\label{neknwdknknk}
		m'\circ G_{1}=G_{2}\circ m \,\,\,\,\,\,\,&\Rightarrow& \,\,\,\,\,\,\,m'\circ G_{1}\circ \tau=G_{2}\circ m\circ \tau\nonumber\\
		&\Rightarrow& \,\,\,\,\,\,\,m'\circ G_{1}\circ \tau=G_{2}\circ \tau'\circ f_{*}.
	\end{eqnarray}
	By Lemma \ref{nfekdnkwfnwknk}, $G_{1}\circ \tau$ and $G_{2}\circ \tau'$ are compatible covering maps. It follows from this and \eqref{neknwdknknk} that 
	$m'$ is a lift of $f$ with respect to $G_{1}\circ \tau$ and $G_{2}\circ \tau'$. In particular, $m'$ is a lift of $f$. 
\end{proof}

\begin{proposition}[\textbf{The derivative of an isometric affine immersion is equivariant}]\label{nfekwnknfkenk}
	Let $(L,X,F)$ and $(L',X',F')$ be toric factorizations of $N$ and $N'$, respectively, with induced toric parametrizations 
	$\pi=\kappa\circ \tau:TM\to M$ and $\pi'=\kappa'\circ \tau':TM'\to M'$. If $f:M\to M'$ is an isometric affine immersion, then its derivative $f_{*}:TM\to TM'$ satisfies 
	\begin{eqnarray*}
		f_{*}\circ (T_{X})_{t}=(T_{X'})_{(\rho_{\tau'\tau}(f))_{*_{e}}t}\circ f_{*}
	\end{eqnarray*}
	for all $t\in\mathbb{R}^{n}$, where $T_{X}:\mathbb{R}^{n}\times TM\to TM,$ $(t,u)\mapsto u+\sum_{k=1}^{n}t_{k}X_{k}$ ($T_{X'}$ is defined 
	similarly). 
\end{proposition}
\begin{proof}
	Let $H:\mathbb{R}^{n}\times TM\to TM'$ be the map defined by $H(t,u)=\big((T_{X'})_{(\rho_{\tau'\tau}(f))_{*_{e}}t}\circ f_{*}\circ (T_{X})_{-t}\big)(u)$.
	We claim that 
	\begin{eqnarray}\label{nfndwkdnknkk}
		\tau'(f_{*}(u))=\tau'(H(t,u))	
	\end{eqnarray}
	for all $(t,u)\in \mathbb{R}^{n}\times TM$. 
	Indeed, it follows from Proposition \ref{cnekdwnkndknk} and the formulas $\tau\circ (T_{X})_{t}=\Phi_{[t]}\circ \tau$ and $\tau'\circ (T_{X'})_{t'}=\Phi'_{[t']}\circ \tau'$ that 
	\begin{eqnarray*}
		\tau'(H(t,u)) &=& \tau'\bigg(\big((T_{X'})_{(\rho_{\tau'\tau}(f))_{*_{e}}t}\circ f_{*}\circ (T_{X})_{-t}\big)(u)\bigg)\\
		&=&\Big(\Phi'_{\rho_{\tau'\tau}(f)([t])}\circ \tau'\circ f_{*}\circ (T_{X})_{-t}\Big)(u)\\
		&=& \Big(\Phi'_{\rho_{\tau'\tau}(f)([t])}\circ m_{\tau'\tau}(f)\circ \tau\circ (T_{X})_{-t}\Big)(u)\\
		&=& \Big(\Phi'_{\rho_{\tau'\tau}(f)([t])}\circ m_{\tau'\tau}(f)\circ \Phi_{[-t]}\circ \tau\Big)(u)\\ 
		&=& \Big(\Phi'_{\rho_{\tau'\tau}(f)([t])}\circ \Phi_{\rho_{\tau'\tau}(f)([-t])}'\circ m_{\tau'\tau}(f)\circ \tau\Big)(u)\\ 
		&=& (m_{\tau'\tau}(f)\circ \tau)(u)\\
		&=& (\tau'\circ f_{*})(u).
	\end{eqnarray*}
	This concludes the proof of the claim. Taking the derivative with respect to $t$ in \eqref{nfndwkdnknkk} yields 
	\begin{eqnarray*}
		0=\dfrac{\partial}{\partial t}(\tau'(H(t,u)))=\tau'_{*_{H(t,u)}} \dfrac{\partial}{\partial t} H(t,u),
	\end{eqnarray*}
	which implies $\tfrac{\partial}{\partial t}H(t,u)=0$, since $\tau'_{*}$ is a linear bijection at every point of $TM'$. Thus 
	$H(t,u)$ is independent of $t$. It follows that $H(t,u)=H(0,u)$, which proves the proposition.  
\end{proof}

\section{Fundamental lattices and K\"{a}hler functions}\label{nfknnkndknknknk}

	In this section, we explore the close relationship between parallel lattices on a toric dually flat space $M$ and the space of K\"{a}hler 
	functions on $TM$ (see Definition \ref{fdkjfekgjrkgj} below). We begin with the following simple observation.

\begin{lemma}\label{nfeknwkenknk}
	Let $\Phi:\mathbb{T}^{n}\times N\to N$ and $\Phi':\mathbb{T}^{n}\times N'\to N'$ be regular torifications of the dually flat space 
	$(M,h,\nabla)$. If $(L,X,F)$ and $(L',X',F')$ are toric parametrizations of $N$ and $N'$, respectively, then $L=L'$. 
\end{lemma}
\begin{proof}
	It suffices to show that $\Gamma(L)=\Gamma(L')$. Let $\pi=\kappa\circ \tau$ (resp. $\pi=\kappa'\circ \tau'$) 
	be the toric factorization induced by $(L,X,F)$ (resp. $(L',X',F')$). Since $\Gamma(L)=\textup{Deck}(\tau)$ and $\Gamma(L')=\textup{Deck}(\tau')$, 
	we must show that $\textup{Deck}(\tau)=\textup{Deck}(\tau')$. Let $m=m_{\tau'\tau}(Id_{M}):N\to N'$ be the lift of $Id_{M}$ with respect to $\tau$ and $\tau'$. 
	By Proposition \ref{cnekdwnkndknk}, $m$ is a K\"{a}hler isomorphism satisfying $m\circ \tau=\tau'\circ (Id_{M})_{*}=\tau'$, that is, $m\circ \tau=\tau'.$
	Now, let $\varphi:TM\to TM$ be a diffeomorphism. We have the following equivalences: 
	\begin{eqnarray*}
		\varphi\in \textup{Deck}(\tau) \,\,\,\,\Leftrightarrow \,\,\,\,\tau\circ \varphi=\tau 
		\,\,\,\,\Leftrightarrow \,\,\,\,m\circ \tau\circ \varphi=m\circ\tau 
		\,\,\,\,\Leftrightarrow \,\,\,\,\tau'\circ\varphi =\tau' \,\,\,\,\Leftrightarrow \,\,\,\,\varphi\in \textup{Deck}(\tau'). 
	\end{eqnarray*}
	Thus $\textup{Deck}(\tau)=\textup{Deck}(\tau')$. 
\end{proof}
	
	In view of this lemma, we are led to introduce the following definition.

\begin{definition}
	Let $(M,h,\nabla)$ be a toric dually flat space. The \textit{fundamental lattice} of $M$ is the parallel lattice described in Lemma \ref{nfeknwkenknk}.
\end{definition}

	Next we turn our attention to K\"{a}hler functions, that we now define. 
	
\begin{definition}\label{fdkjfekgjrkgj}
	Let $N$ be a K\"{a}hler manifold with K\"{a}hler structure $(g,J,\omega).$
	A smooth function $f:N\rightarrow \mathbb{R}$ is called a \textit{K\"{a}hler function} if it satisfies $\mathscr{L}_{X_{f}}g=0$,
	where $X_{f}$ is the Hamiltonian vector field associated to $f$ 
	(i.e.\ $\omega(X_{f},\,.\,)=df(.)$) and where $\mathscr{L}_{X_{f}}$ is the Lie derivative in the direction of $X_{f}.$
\end{definition}

	Following \cite{Cirelli-Quantum}, we denote by $\mathscr{K}(N)$ the space of K\"{a}hler functions on $N.$ 
	When $N$ has a finite number of connected components, then 
	$\mathscr{K}(N)$ is a finite dimensional\footnote{The fact that $\mathscr{K}(N)$ is finite dimensional 
	comes from the following result: if $(M,h)$ is a connected Riemannian manifold, 
	then its space of Killing vector fields $\mathfrak{i}(M):=
	\{X\in \mathfrak{X}(M)\,\big\vert\,\mathscr{L}_{X}h=0\}$ 
	is finite dimensional (see for example \cite{Jost}). } Lie algebra for the Poisson 
	bracket $\{f,g\}:=\omega(X_{f},X_{g})\,.$

	We say that $\mathscr{K}(N)$ \textit{separates the points of $N$} if $f(p)=f(q)$ for all $f\in \mathscr{K}(N)$ implies $p=q$. \\

	The proposition below is the main result of this section. 

\begin{proposition}\label{nfeknkwneknk}
	Let $\Phi:\mathbb{T}^{n}\times N\to N$ be a regular torification of the dually flat space $(M,h,\nabla)$, with fundamental lattice 
	$\mathscr{L}\subset TM$. Then 
	\begin{eqnarray}\label{njenfnenkfkenk}
		\Gamma(\mathscr{L})\subseteq \big\{\varphi\in \textup{Diff}(TM)\,\,\big\vert\,\, f\circ \varphi=f\,\,\forall\,f\in \mathscr{K}(TM)\big\}. 
	\end{eqnarray}
	When $\mathscr{K}(N)$ separates the points of $N$, \eqref{njenfnenkfkenk} is an equality.
\end{proposition}

	Before proving Proposition \ref{nfeknkwneknk}, we give an application of it. 

\begin{corollary}\label{nekdnwkwndknk}
	Let $(M,h,\nabla)$ be a connected dually flat space of dimension $n$. If $\mathscr{K}(TM)$ separates the points of $TM$, 
	then $M$ is not toric. 
\end{corollary}
\begin{example}\label{nceknwdkenknk}
	Consider the set $\mathcal{N}(\mu)$ of normal distributions with known variance $\sigma=1$:
	\begin{eqnarray*}
		p(x;\mu)=\frac{1}{\sqrt{2\pi}}\exp\biggl\{-\frac{\bigl(x-\mu\bigr)^{2}}{2}\biggr\}
		\quad\bigl(x\in\mathbb{R}\bigr).
	\end{eqnarray*}
	The set $\mathcal{N}(\mu)$ is a $1$-dimensional statistical manifold parameterized by 
	the mean $\mu\in \mathbb{R}$. It is an exponential family, because 
	$p(x;\mu)=\exp\big\{C(x)+\theta F(x)-\psi(\theta)\big\}$, where
	\begin{align*}
		 \theta=\mu,\quad\quad C(x)=\ln\bigg(\dfrac{1}{\sqrt{2\pi}}\bigg)-\dfrac{x^{2}}{2}, \quad\quad F(x)= x, \quad\quad  \psi(\theta)=\dfrac{\theta^{2}}{2}.  
	\end{align*}
	The Hessian of $\psi$ is $[1]$. It follows from this and Proposition \ref{prop:4.1} 
	that $T\mathcal{N}(\mu)$ is K\"{a}hler isomorphic to $\mathbb{C}$ endowed with the flat canonical K\"{a}hler structure. Let $f_{1}:\mathbb{C}\to \mathbb{R}$ and 
	$f_{2}:\mathbb{C}\to \mathbb{R}$ be defined by $f_{1}(z)=\textup{Real}(z)$ and $f_{2}(z)=\textup{Im}(z)$. A simple verification shows that 
	$f_{1}$ and $f_{2}$ are K\"{a}hler functions. Clearly, if $f_{k}(z)=f_{k}(w)$, $k=1,2$, then $z=w$. This implies that 
	$\mathscr{K}(\mathbb{\mathbb{C}})$ separates the points of $\mathbb{C}$. By Corollary \ref{nekdnwkwndknk}, $\mathcal{N}(\mu)$ is not toric. 
\end{example}

\begin{example}
	Let $\mathscr{P}$ be the set of Poisson distributions defined over $\Omega=\mathbb{N}=\{0,1,...\}$ (see Example \ref{nfeknwdkenfknk}). 
	Then $\mathscr{P}$ is toric with regular torification $\Phi:\mathbb{T}\times \mathbb{C}\to \mathbb{C}$, $([t],z)\to e^{2i\pi t}z$. 
	Let $\mathscr{L}\subset T\mathscr{P}$ be the fundamental lattice of $\mathscr{P}$. 
	As we noted in Example \ref{nceknwdkenknk}, $\mathscr{K}(\mathbb{C})$ separates the points of $\mathbb{C}$. Therefore $\Gamma(\mathscr{L})$ 
	coincides with the set of diffeomorphisms $\varphi:T\mathscr{P}\to T\mathscr{P}$ satisfying $f\circ \varphi=f$ for all $f\in \mathscr{K}(T\mathscr{P})$.  
	The coordinate expression for the Fisher metric $h_{F}$ is the Hessian of the cumulant generating function: 
	$h_{F}(\theta)=\textup{Hess}(\psi)=\big[e^{\theta}\big].$ 	
	It follows from this and Proposition \ref{prop:4.1} that $T\mathscr{P}$ is K\"{a}hler isomorphic to $\mathbb{C}$ endowed with the K\"{a}hler metric 
	\begin{eqnarray*}
		g_{z}(u,v)=e^{x}\,\textup{Real}(\overline{u}v),
	\end{eqnarray*}
	where $z,u,v\in \mathbb{C}$, $z=x+iy$, $x,y\in \mathbb{R}$. The space of K\"{a}hler functions on $T\mathscr{P}=\mathbb{C}$ 
	is spanned by 
	\begin{eqnarray*}
		 1, \,\,\,\,\,\,\,e^{x}, \,\,\,\,\,\,\,e^{\tfrac{x}{2}}\cos\big(\tfrac{y}{2}\big),
		\,\,\,\,\,\,\, e^{\tfrac{x}{2}}\sin\big(\tfrac{y}{2}\big)
	\end{eqnarray*}
	(to see this, use Proposition 2.25 in \cite{Molitor2014}). Let $\varphi:\mathbb{C}\to \mathbb{C}$ be a diffeomorphism satisfying 
	$f\circ \varphi=f$ for all $f\in \mathscr{K}(\mathbb{C})=\mathscr{K}(T\mathscr{P})$. Let $f_{0}:\mathbb{C}\to \mathbb{C}$, $z\mapsto e^{\tfrac{z}{2}}$. 
	Because the real and imaginary parts of $f_{0}$ are K\"{a}hler functions, we have $f_{0}\circ \varphi=f_{0}$. 
	It is well known that the map $\mathbb{C}\to \mathbb{C}^{*},$ $z\mapsto e^{z}$ is a covering map whose 
	Deck transformation group is $\textup{Deck}(e^{z})=2i\pi\mathbb{Z}$ (here we identify translations in $\mathbb{C}$ with 
	complex numbers). Therefore $f_{0}:\mathbb{C}\to \mathbb{C}^{*}$ is a covering map with Deck transformation group $4i\pi \mathbb{Z}$. 
	Thus $\varphi\in \textup{Deck}(f_{0})=4i\pi \mathbb{Z}$ and hence there exists $k\in \mathbb{Z}$ such that $\varphi(z)=z+4\pi ik$ for all $z\in \mathbb{C}$. 
	Conversely, if $\varphi\in \textup{Diff}(\mathbb{C})$ is of the form $\varphi(z)=z+4\pi i k$ for some $k\in \mathbb{Z}$, 
	then clearly $f\circ \varphi=f$ for all $f\in \mathscr{K}(\mathbb{C})$. It follows that $\Gamma(\mathscr{L})$
	is the set of diffeomorphisms $\varphi\in \textup{Diff}(\mathbb{C})$ of the form 
	$\varphi(z)= z+4\pi ik$, where $k\in \mathbb{Z}$. If $\theta$ is the natural parameter of $\mathscr{P}$ as described in Example \ref{nfeknwdkenfknk}, then 
	we see that the fundamental lattice $\mathscr{L}$ is generated by $X=4\pi\tfrac{\partial}{\partial\theta}\in \mathfrak{X}(\mathscr{P})$.

	\end{example}

	Now we proceed with the proof of Proposition \ref{nfeknkwneknk}. It is based on the following result, which is due to Nomizu \cite{Nomizu}:

\begin{theorem}\label{nfeknkfenknk}
	Let $(M,g)$ be a real analytic Riemannian manifold and let $X$ be a Killing vector field defined on the open 
	set $U\subseteq M$. If $M$ and $U$ are connected and $M$ is simply connected, then $X$ extends uniquely to a 
	global Killing vector field on $M$. 
\end{theorem}

	We also need the following lemma.

\begin{lemma}\label{nfkdnwkefnknk}
	Let $\Phi:\mathbb{T}^{n}\times N\to N$ be a regular torification of a dually flat space $(M,h,\nabla)$. 
	Let $\tau:TM\to N^{\circ}$ be a compatible covering map and $f:TM\to \mathbb{R}$ 
	a smooth function. Then $f$ is K\"{a}hler if and only if there exists a K\"{a}hler function $\overline{f}$ on $N$ such that 
	$f=\overline{f}\circ \tau$.
\end{lemma}
\begin{proof}
	One direction is immediate: if $\overline{f}:N\to \mathbb{R}$ is K\"{a}hler, then so is 
	$f=\overline{f}\circ \tau$ (since $\tau$ is a K\"{a}hler covering map). Conversely, suppose $f:TM\to \mathbb{R}$ is K\"{a}hler. 
	Let $\{U_{\alpha}\}_{\alpha\in A}$ be a cover of $TM$ by connected open sets $U_{\alpha}$ such that 
	$\tau\vert_{U_{\alpha}}:U_{\alpha}\to V_{\alpha}:=\tau(U_{\alpha})$ is a K\"{a}hler isomorphism for 
	all $\alpha\in A$. Given $\alpha\in A$, define $f_{\alpha}:V_{\alpha}\to \mathbb{R}$ by 
	\begin{eqnarray*}
		 f_{\alpha}:=f\circ (\tau\vert_{U_{\alpha}})^{-1}.
	\end{eqnarray*}
	Since $\tau\vert_{U_{\alpha}}$ is a K\"{a}hler isomorphism, $f_{\alpha}$ 
	is a K\"{a}hler function on $V_{\alpha}\subseteq N$. This implies that the Hamiltonian vector field 
	$X_{f_{\alpha}}$ of $f_{\alpha}$ is a Killing vector field on 
	$V_{\alpha}$. Since $V_{\alpha}$ is connected and $N$ is a connected and 
	simply connected real analytic Riemannian manifold, $X_{f_{\alpha}}$ extends uniquely to a globally 
	defined Killing vector field, say $X_{\alpha}$, on $N$ (see Theorem \ref{nfeknkfenknk}). 
	Note that $\mathscr{L}_{X_{\alpha}}\omega=0$ on $V_{\alpha}$, by construction. By an argument entirely analogous 
	to that in the proof of Proposition \ref{ncnkwknknk}, one sees that $X_{\alpha}=X_{\beta}$ on $N$ whenever $\alpha,\beta\in A$.
%
%
	Set $X=X_{\alpha}\in \mathfrak{X}(N)$, where $\alpha$ is any element in $A$ (this is independent of the choice of 
	$\alpha$). Since $\{V_{\alpha}\}_{\alpha\in A}$ is a cover of 
	$N^{\circ}$ and since $\mathscr{L}_{X}\omega=\mathscr{L}_{X_{\alpha}}\omega=0$ on each $V_{\alpha}$, we see that 
	$\mathscr{L}_{X}\omega=0$ on $N^{\circ}$. Since $N^{\circ}$ is dense in $N$, $\mathscr{L}_{X}\omega=0$ 
	on $N$. Because of this and because $N$ is simply connected, there exists a smooth function 
	$\overline{f}:N\to \mathbb{R}$ such that $X=X_{\overline{f}}$. The function $\overline{f}$ is K\"{a}hler, since 
	$X$ is a Killing vector field. By construction, $X_{\overline{f}}=X_{f_{\alpha}}$ on $V_{\alpha}$ for every $\alpha\in A$. 
	Since $V_{\alpha}$ is connected, there is $C_{\alpha}\in \mathbb{R}$ such that $\overline{f}=f_{\alpha}+C_{\alpha}$ on 
	$V_{\alpha}$. This implies $\overline{f}\circ \tau=f_{\alpha}\circ \tau+C_{\alpha}$ on $U_{\alpha}$ for all $\alpha\in A$. 
	Since $f_{\alpha}=f\circ (\tau\vert_{U_{\alpha}})^{-1}$, this implies $\overline{f}\circ \tau=f+C_{\alpha}$ on $U_{\alpha}$ 
	for all $\alpha\in A$. Since $\overline{f}\circ \tau -f$ is continuous on $TM$, $C_{\alpha}=C_{\beta}\equiv C$ whenever $\alpha,\beta\in A$. 
	Thus $\overline{f}\circ \tau=f+C$ on $TM$. Redefining $\overline{f}$ by adding a constant if necessary, we can assume $C=0.$ 
	The lemma follows. 
\end{proof}

\begin{proof}[Proof of Proposition \ref{nfeknkwneknk}]
	Let $\varphi\in \Gamma(\mathscr{L})$ be arbitrary. Let $\tau:TM\to N^{\circ}$ be a compatible covering map. 
	By Lemma \ref{nfeknwkenknk}, $\Gamma(\mathscr{L})=\textup{Deck}(\tau)$ and hence $\tau\circ \varphi=\tau$. It follows that
	$\overline{f}\circ \tau\circ \varphi=\overline{f}\circ \tau$ for all $\overline{f}\in 
	\mathscr{K}(N)$. By Lemma \ref{nfkdnwkefnknk}, this implies that $f\circ \varphi=f$ for all 
	$f\in \mathscr{K}(TM)$. This shows the inclusion \eqref{njenfnenkfkenk}. 

	Suppose now that $\mathscr{K}(N)$ separates the points of $N$. Let $\varphi\in \textup{Diff}(TM)$ 
	be such that $f\circ \varphi=f$ for all $f\in \mathscr{K}(TM).$
	By Lemma \ref{nfkdnwkefnknk}, $\overline{f}\circ \tau\circ \varphi=\overline{f}\circ \tau$ 
	for all $\overline{f}\in \mathscr{K}(N)$. Because $\mathscr{K}(N)$ separates the points of $N$, this implies that 
	$\tau\circ \varphi=\tau$, and so $\varphi\in \textup{Deck}(\tau)=\Gamma(\mathscr{L})$. This shows the converse inclusion.
\end{proof}


\section{Torifications and projective varieties}\label{nwkwnkenfkwnknknk}

	Throughout this section, $\mathcal{E}$ is an exponential family of dimension $n$ defined over a finite set $\Omega=\{x_{0},...,x_{r}\}$, with elements of the form $p(x;\theta)=
	\textup{exp}(C(x)+\langle F(x),\theta\rangle-\psi(\theta))$, where $\langle\,,\,\rangle$ is the Euclidean pairing on $\mathbb{R}^{n}$ and 
	\begin{eqnarray*}
		x\in \Omega,\quad\quad\theta\in \mathbb{R}^{n},\quad\quad C:\Omega\to \mathbb{R}, \quad\quad F=(F_{1},...,F_{n}):
		\Omega\to \mathbb{R}^{n}, \quad\quad \psi:\mathbb{R}^{n}\to \mathbb{R}.
	\end{eqnarray*}
	It is assumed that the functions $1,F_{1},...,F_{n}:\Omega\to \mathbb{R}$ are independent 
	so that the map $\mathcal{E}\to \mathbb{R}^{n},$ $p(.,\theta)\mapsto\theta$ becomes a bijection. 
	Note that:
	\begin{itemize}
	\item The condition $\sum_{k\in \Omega}p(k;\theta)=1$ implies that $\psi(\theta)=\ln\big(\sum_{k\in \Omega}\textup{exp}(C(k)-\langle F(k),\theta\rangle)\big)$ for all 
	$\theta\in \mathbb{R}^{n}$. 
	\item $\mathcal{E}$ is a subset of $\mathcal{P}_{r+1}^{\times}$ (see Example \ref{exa:5.5}). 
	\end{itemize}
	We endow $\mathcal{E}$ and $\mathcal{P}_{r+1}^{\times}$ with their canonical dually flat structures (given by the Fisher metric and 
	exponential connection).
\begin{lemma}\label{nefknknknk}
	The inclusion map $j:\mathcal{E}\hookrightarrow \mathcal{P}_{r+1}^{\times}$ is an affine isometric immersion.
\end{lemma}
\begin{proof}
	Let $\theta=(\theta_{1},...,\theta_{n})$ and $\theta'=(\theta_{1}',...,\theta_{r}')$ be the natural parameters of $M$ and 
	$\mathcal{P}_{r+1}^{\times}$, respectively. Recall that $p(x_{i};\theta')=e^{\theta_{i+1}'-\psi'(\theta')}$ 
	if $i=0,...,r-1$ and $p(x_{r};\theta')=e^{-\psi'(\theta')}$, where $\psi'(\theta')=\ln(1+\sum_{k=1}^{r}e^{\theta_{k}'})$ 
	is the cumulant generating function of $\mathcal{P}_{r+1}^{\times}$ (see Example \ref{exa:5.5}). 
	Let $p=p(.,\theta)=p(.,\theta')\in \mathcal{E}$ and $i=0,...,r$ be arbitrary. Since $p(x_{i},\theta)=p(x_{i},\theta'),$ we have 
	\begin{eqnarray*}
		\textup{exp}\bigg(C(x_{i})+\sum_{j=1}^{n}F_{j}(x_{i})\theta_{j}-\psi(\theta)\bigg)=
		\left\lbrace
		\begin{array}{llll}
			& \textup{exp}\big(\theta'_{i+1}-\psi'(\theta')\big), &\textup{if}   & i=0,...,r-1,\\[0.5em]
			& \textup{exp}\big(-\psi'(\theta')\big),              &  \textup{if} & i=r,
		\end{array}
		\right.
	\end{eqnarray*}
	and thus
	\begin{eqnarray*}
		\left \lbrace
		\begin{array}{llll}
			C(x_{i})+\sum_{j=1}^{n}F_{j}(x_{i})\theta_{j}-\psi(\theta) &=& \theta'_{i+1}-\psi'(\theta'), & i=0,...,r-1,\\[0.5em]
			C(x_{r})+\sum_{j=1}^{n}F_{j}(x_{r})\theta_{j}-\psi(\theta) &=& -\psi'(\theta').            & 
		\end{array}
		\right.
	\end{eqnarray*}
	Therefore 
	\begin{eqnarray*}
		\theta_{i}'=C(x_{i-1})-C(x_{r})+\big\langle \theta,F(x_{i-1})-F(x_{r})\big\rangle \,\,\,\,\,\,\,\,\,\,(i=1,...,r)
	\end{eqnarray*}
	that is, 
	\begin{eqnarray}\label{nfeknwkdnknk}
		\begin{bmatrix}
			\theta_{1}'  \\
			\vdots       \\
			\theta_{r}'
		\end{bmatrix}
		=\begin{bmatrix}
			F_{1}(x_{0})-F_{1}(x_{r})  &  \cdots    &   F_{n}(x_{0})-F_{n}(x_{r})  \\
			\vdots                     &            &   \vdots                     \\
			F_{1}(x_{r-1})-F_{1}(x_{r}) &  \cdots    &   F_{n}(x_{r-1})-F_{n}(x_{r}) 
		\end{bmatrix}
		\begin{bmatrix}
			\theta_{1}  \\
			\vdots       \\
			\theta_{n}
		\end{bmatrix}
		+
		\begin{bmatrix}
			C(x_{0})-C(x_{r})  \\
			\vdots       \\
			C(x_{r-1})-C(x_{r})
		\end{bmatrix}
		.
	\end{eqnarray}
	Formula \eqref{nfeknwkdnknk} is the coordinate expression for $j$ in the natural parameters. 
	This shows that $j$ is affine. 

	Next we prove that $j$ is isometric. Let $h$ and $h'$ be the Fisher metrics on $\mathcal{E}$ and $\mathcal{P}_{r+1}^{\times}$, respectively, and let 
	$g$ be the Euclidean metric on $\mathbb{R}^{r+1}$. Let $f:M\to \mathbb{R}^{r+1}$, $p\mapsto (\sqrt{p(x_{0})},...,\sqrt{p(x_{r})})$. Given $1\leq i,j\leq n$, 
	a simple calculation shows that 
	\begin{eqnarray}\label{enfwknkkn}
		g_{f(p)}\bigg( f_{*_{p}}\dfrac{\partial}{\partial \theta_{i}},f_{*_{p}}\dfrac{\partial}{\partial \theta_{j}}\bigg)=\dfrac{1}{4}\sum_{k=0}^{r}\bigg(F_{i}(x_{k})-
		\dfrac{\partial \psi}{\partial \theta_{i}}\bigg)
		\bigg(F_{j}(x_{k})-\dfrac{\partial \psi}{\partial \theta_{j}}\bigg)p(x_{k}). 
	\end{eqnarray}
	In \cite{Amari-Nagaoka}, Formula 3.59, it is observed that the coordinate expression for $h$ in the natural parameters is given by 
	\begin{eqnarray}\label{nekwndknkdn}
		h_{ij}(\theta)=
		\sum_{k=0}^{r}\bigg(F_{i}-\dfrac{\partial \psi}{\partial \theta_{i}}\bigg)\bigg(F_{j}-\dfrac{\partial \psi}{\partial \theta_{j}}\bigg)p(x_{k};\theta). 
	\end{eqnarray}
	Comparing \eqref{enfwknkkn} and \eqref{nekwndknkdn} we obtain $f^{*}g=\tfrac{1}{4}h$. Now let 
	$\tilde{f}:\mathcal{P}_{r+1}^{\times}\to \mathbb{R}^{r+1}$, $p\mapsto (\sqrt{p(x_{0})},...,\sqrt{p(x_{r})})$. Note that $f=\tilde{f}\circ j$. 
	By the same argument as above with $\mathcal{E}$ replaced by $\mathcal{P}_{r+1}^{\times}$, 
	we get $\tilde{f}^{*}g=\tfrac{1}{4}h'$ and hence 
	\begin{eqnarray*}
		h=4f^{*}g=4(\tilde{f}\circ j)^{*}g=4 j^{*}\tilde{f}^{*}g=4j^{*}\tfrac{1}{4}h'=j^{*}h'.
	\end{eqnarray*}
	Therefore $h=j^{*}h'$.  
\end{proof}

	Given $n\geq 1$, recall the notation $\Phi_{n}:\mathbb{T}^{n}\times \mathbb{P}_{n}(c)\to \mathbb{P}_{n}(c)$ defined before Proposition \ref{nfeknwknwknk}. 

\begin{theorem}
	Let $\mathcal{E}$ be an exponential family of dimension $n$ defined over a finite set $\Omega=\{x_{0},x_{1},...,x_{r}\}$ 
	(as described in the begining of this section). Suppose $\mathcal{E}$ toric with regular torification $\Phi:\mathbb{T}^{n}\times N\to N$. Then there is a K\"{a}hler immersion 
	$m:N\to \mathbb{P}_{r}(1)$ and a Lie group homomorphism $\rho:\mathbb{T}^{n}\to \mathbb{T}^{r}$ with finite kernel such that $m\circ \Phi_{a}=
	(\Phi_{r})_{\rho(a)}\circ m$ for all $a\in \mathbb{T}^{n}$. 
\end{theorem}
\begin{proof}
	Since the inclusion map $j:\mathcal{E}\to \mathcal{P}_{r+1}^{\times}$ is an isometric affine immersion between toric dually flat spaces, it has a lift 
	$m$ with the desired properties. 
\end{proof}

	We now look at some examples. Let $\mathcal{B}(n)$ be the set of Binomial distributions defined over $\Omega=\{0,1,...,n\}$ (see Example \ref{exa:5.6}). 
	Recall that $\Phi_{1}:\mathbb{T}^{1}\times \mathbb{P}_{1}(\tfrac{1}{n})\to \mathbb{P}_{1}(\tfrac{1}{n})$ is the regular torification of 
	$\mathcal{B}(n)$ (see Proposition \ref{nfeknwknwknk}).
\begin{proposition}\label{ncednwkneksnk}
	The lift of the inclusion map $\mathcal{B}(n)\hookrightarrow \mathcal{P}_{n+1}^{\times}$ is the Veronese embedding:
		\begin{center}
		\begin{tabular}{llllll}
			$m:$ & $\mathbb{P}_{1}(\tfrac{1}{n})$  & $\to$     & $\mathbb{P}_{n}(1)$\\
			    & $[z_{1},z_{2}]$   & $\mapsto$ & $\bigg[z_{1}^{n},...,\displaystyle\binom{n}{k}^{1/2}z_{1}^{n-k}z_{2}^{k},...,z_{2}^{n}\bigg]$.
		\end{tabular}
		\end{center}
		The corresponding Lie group homomorphism $\rho:\mathbb{T}^{1}\to \mathbb{T}^{n}$ is given by 
		\begin{eqnarray*}
			\rho([t])=\big[nt,..., (n-k)t,..., t].
		\end{eqnarray*}
\end{proposition}
\begin{proof}
	Let $j:\mathcal{B}(n)\hookrightarrow \mathcal{P}_{n+1}^{\times}$ be the inclusion map. In the proof of Lemma \ref{nefknknknk}, we computed 
	the coordinate expression for $j$ in the natural parameters. With $F(k)=k$ and $C(k)=\ln\binom{n}{k}$, $k=0,...,n$ (see Example \ref{exa:5.6}), this yields 
	\begin{eqnarray*}
		j(\theta)=-\theta(n,n-1,...,1)+\big(\ln\textstyle\binom{n}{0},...,\ln\textstyle\binom{n}{n-1}\big)
	\end{eqnarray*}
	for all $\theta\in \mathbb{R}$.
	Under the usual identifications $T\mathcal{B}(n)=\mathbb{C}$ and $T\mathcal{P}_{n+1}^{\times}=\mathbb{C}^{n}$, we get 
	\begin{eqnarray*}
		j_{*}(z)=-z(n,n-1,...,1)+\big(\ln\textstyle\binom{n}{0},...,\ln\textstyle\binom{n}{n-1}\big)
	\end{eqnarray*}
	for all $z\in \mathbb{C}$. Let $\tau:\mathbb{C}=T\mathcal{B}(n)\to \mathbb{P}_{1}(\tfrac{1}{n})^{\circ}$, $z\mapsto [e^{z/2},1]$ and 
	$\tau':\mathbb{C}^{n}=T\mathcal{P}_{n+1}^{\times}\to \mathbb{P}_{n}(1)^{\circ}$, $(z_{1},...,z_{n})\mapsto [e^{z_{1}/2},...,e^{z_{n}/2},1]$. 
	By Proposition \ref{njewndknknknk}, $\tau$ and $\tau'$ are compatible covering maps. We compute:
	\begin{eqnarray}
		(\tau'\circ j_{*})(z)&=& \tau'\big(-z(n,n-1,...,1)+\big(\ln\textstyle\binom{n}{0},...,\ln\textstyle\binom{n}{n-1}\big)\big)\nonumber\\
		&=&\big[\textstyle\binom{n}{0}^{1/2}(e^{-z/2})^{n},...,\textstyle\binom{n}{k}^{1/2}(e^{-z/2})^{n-k},...,1\big]\nonumber\\
		&=&\big[\textstyle\binom{n}{0}^{1/2},...,\binom{n}{k}^{1/2}(e^{z/2})^{k},...,(e^{z/2})^{n}\big],\label{ndkwqndkwndknk}
	\end{eqnarray}
	where, in the last line, we have multiplied every entry by $(e^{z/2})^{n}$. Let $\widetilde{m}:\mathbb{P}_{1}(\tfrac{1}{n})\to \mathbb{P}_{n}(1)$ 
	be defined by $\widetilde{m}([z_{1},z_{2}])=\big[z_{2}^{n},...,\binom{n}{k}^{1/2}z_{2}^{n-k}z_{1}^{k},...,z_{1}^{n}\big]$. 
	We compute:
	\begin{eqnarray}\label{nkwnkdndknknk}
		(\widetilde{m}\circ \tau)(z)=\widetilde{m}([e^{z/2},1])=\big[1,...,\textstyle\binom{n}{k}^{1/2}(e^{z/2})^{k},...,(e^{z/2})^{n}\big]
	\end{eqnarray}
	for all $z\in \mathbb{C}$. Comparing \eqref{ndkwqndkwndknk} and \eqref{nkwnkdndknknk}, we see that $\widetilde{m}\circ \tau=\tau'\circ j_{*}$. Therefore $\widetilde{m}$ is the
	lift of $j$ with respect to $\tau$ and $\tau'$. Let $G$ be the holomorphic isometry of $\mathbb{P}_{1}(\tfrac{1}{n})$ defined by $G([z_{1},z_{2}])=[z_{2},z_{1}]$. 
	Since $G\circ (\Phi_{1})_{[t]}=(\Phi_{1})_{[-t]}\circ G$ for all $t\in \mathbb{R}$, Proposition \ref{neknmfkrnknk} implies that 
	$m=\widetilde{m}\circ G$ is a lift of $j$. The second formula is obtained by a direct computation. 
\end{proof}

%
%
%

	Let $\mathcal{M}(m+1,n)$ be the set of multinomial distributions defined over 
	$\Omega_{m+1,n}=\{(k_{1},...,k_{m+1})\in \mathbb{N}^{m+1}\,\,\big\vert\,\,k_{1}+...+k_{m+1}=n\}$
	(see Example \ref{jeknkkdefjdk}). Recall that $\mathcal{M}(m+1,n)$ is toric with regular torification 
	$\Phi_{m}:\mathbb{T}^{m}\times \mathbb{P}_{m}(\tfrac{1}{n})\to \mathbb{P}_{m}(\tfrac{1}{n})$ (see Proposition \ref{nfeknwknwknk}). 
	Since $\textup{Card}(\Omega_{m+1,n})=\binom{m+n}{m}=\tfrac{(m+n)!}{m!n!}$, $\mathcal{M}(m+1,n)$ is a subset of $\mathcal{P}_{\binom{m+n}{m}}^{\times}$.

\begin{proposition}\label{nfknknkefnsknkfnk}
	The lift of the inclusion map $\mathcal{M}(m+1,n)\hookrightarrow \mathcal{P}_{\binom{m+n}{m}}^{\times}$ is the $n${-}$th$ Veronese embedding:
		\begin{center}
		\begin{tabular}{llllll}
			 $\mathbb{P}_{m}(\tfrac{1}{n})$  & $\to$     & $\mathbb{P}_{\binom{m+n}{m}-1}(1)$\\[0.9em]
			 $[z_{1},z_{2},...,z_{m+1}]$   & $\mapsto$ & $\bigg[\,\sqrt{\dfrac{n!}{k_{1}!...k_{m+1}!}}\,\,z_{1}^{k_{1}}...z_{m+1}^{k_{m+1}}\,\bigg]_{k\in 
			    \Omega_{m+1,n}}$.
		\end{tabular}
		\end{center}
\end{proposition}
\begin{proof}
	Let $\phi:\{1,2,...,\binom{m+n}{m}\}\to \Omega_{m+1,n}$ be a bijection. Let $K:=\phi(\binom{m+n}{m})$ and $A=\Omega_{m+1,n}-\{K\}$ ($K$ is just the ``last" element of 
	$\Omega_{m+1,n}$). The coordinate expression for the inclusion map $j:\mathcal{M}(m+1,n)\hookrightarrow \mathcal{P}_{\binom{m+n}{m}}^{\times}$ in the natural parameters 
	is given by $j(\theta_{1},...,\theta_{m})=[\theta'_{k}]_{k\in A}$, where 
	\begin{eqnarray*}
		\left\lbrace
		\begin{array}{llll}
			\theta_{k}'=\sum_{j=1}^{m}(F_{j}(k)-F_{j}(K))\theta_{j} +C(k)-C(K),\\[0.5em]
			F_{j}(k)=k_{j} \,\,\,\,\,\,\,\textup{and} \,\,\,\,\,\,\,C(k)=\ln\big(\tfrac{n!}{k_{1}!....k_{m+1}!}\big)
		\end{array}
		\right.
	\end{eqnarray*}
	(see Example \ref{jeknkkdefjdk} and the proof of Lemma \ref{nefknknknk}).  
	Under the usual identifications $T\mathcal{M}(m+1,n)=\mathbb{C}^{m}$ and $T\mathcal{P}_{\binom{m+n}{m}}^{\times}=\mathbb{C}^{\binom{m+n}{m}-1}$, we have 
	$j_{*}(z_{1},...,z_{m})=[z_{k}']_{k\in A}$, where 
	\begin{eqnarray*}
		z_{k}'=\sum_{j=1}^{m}(F_{j}(k)-F_{j}(K))z_{j}+C(k)-C(K). 
	\end{eqnarray*}
	Let $\tau:\mathbb{C}^{m}\to \mathbb{P}_{m}(\tfrac{1}{n})$, $(z_{1},...,z_{m})\mapsto [e^{z_{1}/2},...,e^{z_{m}/2},1]$ and 
	$\tau':\mathbb{C}^{\binom{m+n}{m}-1}\to \mathbb{P}_{\binom{m+n}{m}-1}(1)$, $(z_{k})_{k\in A}\mapsto [(e^{z_{k}/2})_{k\in A},1]$. 
	By Proposition \ref{njewndknknknk}, $\tau$ and $\tau'$ are compatible covering maps. We compute:
	\begin{eqnarray*}
		(\tau'\circ j_{*})(z_{1},...,z_{m})&=& \tau'([z_{k}']_{k\in A})=[(e^{z_{k}'/2})_{k\in A},1]\\
		&=&\Big[\Big(e^{\tfrac{1}{2}\big(\sum_{j=1}^{m}(F_{j}(k)-F_{j}(K))z_{j}+C(k)-C(K)\big)}\Big)_{k\in A},1\Big]\\
		&=&\Big[\Big(e^{\tfrac{1}{2}\big(\sum_{j=1}^{m}F_{j}(k)z_{j}+C(k)\big)}\Big)_{k\in A},e^{\tfrac{1}{2}\big(\sum_{j=1}^{m}F_{j}(K)z_{j}+C(K)\big)}\Big]\\
		&=&\Big[e^{\big(\tfrac{1}{2}\sum_{j=1}^{m}F_{j}(k)z_{j}+C(k)\big)}\Big]_{k\in \Omega_{m+1,n}}\\
		&=&\bigg[\,\sqrt{\dfrac{n!}{k_{1}!...k_{m+1}!}}\,\,(e^{z_{1}/2})^{k_{1}}...(e^{z_{m}/2})^{k_{m}}\,\bigg]_{k\in \Omega_{m+1,n}}\\
		&=&(m\circ \tau)(z_{1},...,z_{m}),
	\end{eqnarray*}
	where $m$ is the $n${-}$th$ Veronese embedding. This shows that $m$ is the lift of $j$. 
\end{proof}

	Now we show how to construct new examples from old ones.

\begin{proposition}[\textbf{Torification of products}]\label{nfeknknefknfknk}
	Given $i=1,2$, let $\psi_{i}:\mathbb{R}^{n_{i}}\to \mathbb{R}$ be a smooth function whose Hessian is positive definite at each point of $\mathbb{R}^{n}$. 
	Suppose that $\Phi_{i}:\mathbb{T}^{n_{i}}\times N_{i}\to N_{i}$ is a torification 
	of $(\mathbb{R}^{n_{i}}, \textup{Hess}(\psi_{i}),\nabla^{\textup{flat}})$, $i=1,2$. 
	Let $\Phi=\Phi_{1}\times \Phi_{2}:\mathbb{T}^{n_{1}+n_{2}}\times (N_{1}\times N_{2})\to N_{1}\times N_{2}$ be the torus action defined by 
	\begin{eqnarray*}
		\Phi\big((a,b),(x,y)\big)=\big(\Phi_{1}(a,x), \Phi_{2}(b,y)\big).
	\end{eqnarray*}
	\begin{enumerate}[(1)]
		\item $N_{1}\times N_{2}$, together with the torus action $\Phi$, is a torification of $(\mathbb{R}^{n_{1}+n_{2}},\textup{Hess}(\psi),\nabla^{\textup{flat}})$, where 
			$\psi:\mathbb{R}^{n_{1}+n_{2}}=\mathbb{R}^{n_{1}}\times \mathbb{R}^{n_{2}} \to \mathbb{R}$, $(x,y)\mapsto \psi_{1}(x)+\psi_{2}(y)$.
		\item If $\tau_{i}:T\mathbb{R}^{n_{i}}\to N_{i}^{\circ}$ is a compatible covering map, $i=1,2$, then the map $\tau:T\mathbb{R}^{n_{1}+n_{2}}=
			T\mathbb{R}^{n_{1}}\times T\mathbb{R}^{n_{2}}\to N_{1}^{\circ}\times 
			N_{2}^{\circ}$, defined by $\tau(u,v)=\big(\tau_{1}(u), \tau_{2}(v)\big)$, 
			is a compatible covering map.  
		\item If $N_{1}$ and $N_{2}$ are regular, then 
		$N_{1}\times N_{2}$ is regular. 
	\end{enumerate}

	\end{proposition}
\begin{proof}
	(1) Let $(L_{i},X_{i},F_{i})$ be a toric factorization of $\Phi_{i}:\mathbb{T}^{n_{i}}\times N_{i}\to N_{i}$, $i=1,2.$ Write $X_{i}=((X_{i})_{1},...,(X_{i})_{n_{i}})$, where 
	$(X_{i})_{j}$ is a vector field on $\mathbb{R}^{n_{i}}$. Given $1\leq i\leq n_{1}$ and $1\leq j\leq n_{2}$ define the vector fields $(\widetilde{X}_{1}){i}$ and 
	$(\widetilde{X}_{2})_{j}$ on $\mathbb{R}^{n_{1}+n_{2}}=\mathbb{R}^{n_{1}}\times \mathbb{R}^{n_{2}}$ by letting 
	\begin{eqnarray*}
		(\widetilde{X}_{1})_{i}(p_{1},p_{2})=((X_{1})_{i}(p_{1}),0) \,\,\,\,\,\,\,\,\,\,\textup{and}  \,\,\,\,\,\,\,\,\,\,
		(\widetilde{X}_{2})_{j}(p_{1},p_{2})=(0, (X_{2})_{j}(p_{2})).
	\end{eqnarray*}
	Clearly the vector fields $(\widetilde{X}_{i})_{j}$ are pointwise linearly independent and parallel with respect to the flat connection (since they are constant). 
	It follows that $\widetilde{X}=((\widetilde{X}_{1})_{1},...,(\widetilde{X}_{1})_{n_{1}}, (\widetilde{X}_{2})_{1},..., (\widetilde{X}_{2})_{n_{2}})$ 
	is the generator of a parallel lattice $L\subset T\mathbb{R}^{n_{1}+n_{2}}$. 

	Under the natural identifications $\Gamma(L)=\mathbb{Z}^{n_{1}+n_{2}}$ and $\Gamma(L_{i})=\mathbb{Z}^{n_{i}}$, the map $\mathbb{Z}^{n_{1}+n_{2}}\to 
	\mathbb{Z}^{n_{1}}\times \mathbb{Z}^{n_{2}}$, 
	$(k_{1},...,k_{n_{1}+n_{2}})\mapsto ((k_{1},...,k_{n_{1}}), (k_{n_{1}+1},...,k_{n_{1}+n_{2}}))$ induces a group isomorphism 
	$\phi:\Gamma(L)\to \Gamma(L_{1})\times \Gamma(L_{2})$. 
	Let $\pi_{i}:\mathbb{R}^{n_{1}+n_{2}}=\mathbb{R}^{n_{1}}\times \mathbb{R}^{n_{2}}\to \mathbb{R}^{n_{i}}$ be the projection onto 
	$\mathbb{R}^{n_{i}}$ and let $G:T\mathbb{R}^{n_{1}+n_{2}}\to T\mathbb{R}^{n_{1}}\times T\mathbb{R}^{n_{2}}$ be the diffeomorphism defined by $G(u)= ((\pi_{1})_{*}u, (\pi_{2})_{*}u)$.
	A direct verification using $(\pi_{k})_{*_{u}}(\widetilde{X_{i}})_{j}= \delta_{ki}(X_{i})_{j}(\pi_{k}(u))$ ($\delta_{ki}=$ Kronecker delta) shows that 
	$G$ is equivariant in the sense that
	\begin{eqnarray*}
		G\circ \gamma =\phi(\gamma)\cdot G
	\end{eqnarray*}
	for all $\gamma\in \Gamma(L)$, where the action of $\Gamma(L_{1})\times \Gamma(L_{2})$ on $T\mathbb{R}^{n_{1}}\times T\mathbb{R}^{n_{2}}$ is given by 
	$(\gamma_{1},\gamma_{2})\cdot (u,v)=(\gamma_{1}(u),\gamma_{2}(v))$. Moreover, a direct verification using Proposition \ref{prop:4.1} and the formula 
	\begin{eqnarray*}
		\textup{Hess}(\psi)=
		\begin{bmatrix}
			\textup{Hess}(\psi_{1})   &   0 \\
			0                         &   \textup{Hess}(\psi_{2})
		\end{bmatrix}
	\end{eqnarray*}
	shows that $G$ is a K\"{a}hler isomorphism. It follows that $G$ descends to a 
	K\"{a}hler isomorphism,
	\begin{eqnarray*}
		\widetilde{G}\,\,:\,\,\mathbb{R}^{n_{1}+n_{2}}_{L}\to \mathbb{R}^{n_{1}}_{L_{1}}\times \mathbb{R}^{n_{2}}_{L_{2}}. 
	\end{eqnarray*}
	Let $\Phi_{\widetilde{X}}:\mathbb{T}^{n_{1}+n_{2}}\times \mathbb{R}^{n_{1}+n_{2}}_{L}\to \mathbb{R}^{n_{1}+n_{2}}_{L}$ and 
	$\Phi_{X_{i}}:\mathbb{T}^{n_{i}}\times \mathbb{R}^{n_{i}}_{L_{i}}\to \mathbb{R}^{n_{i}}_{L_{i}}$ be the torus actions associated to the generators $\widetilde{X}$ and $X_{i}$, 
	respectively (see \eqref{nndkfneknkw}). Let $\Phi_{X_{1}}\times \Phi_{X_{2}}$ be the action of $\mathbb{T}^{n_{1}+n_{2}}=\mathbb{T}^{n_{1}}\times \mathbb{T}^{n_{2}}$ on 
	$\mathbb{R}^{n_{1}}_{L_{1}}\times \mathbb{R}^{n_{2}}_{L_{2}}$ defined by $(\Phi_{X_{1}}\times \Phi_{X_{2}})((a,b),(x,y))=
	\big(\Phi_{X_{1}}(a,x),\Phi_{X_{2}}(b,y)\big)$. A direct calculation shows that 
	\begin{eqnarray*}
		\widetilde{G}\circ \big(\Phi_{\widetilde{X}}\big)_{a}=(\Phi_{X_{1}}\times \Phi_{X_{2}})_{a}\circ \widetilde{G}
	\end{eqnarray*}
	for all $a\in \mathbb{T}^{n_{1}+n_{2}}$. Thus $\widetilde{G}$ is equivariant. 

	Now let $F:\mathbb{R}^{n_{1}}_{L_{1}}\times \mathbb{R}^{n_{2}}_{L_{2}}\to N_{1}^{\circ}\times N_{2}^{\circ}=(N_{1}\times N_{2})^{\circ}$ 
	be the equivariant K\"{a}hler isomorphism defined by $F(x,y)= (F_{1}(x),F_{2}(y))$. Since the composition $F\circ \widetilde{G}$ is an equivariant K\"{a}hler isomorphism 
	from $\mathbb{R}^{n_{1}+n_{2}}_{L}$ to $(N_{1}\times N_{2})^{\circ}$, $N_{1}\times N_{2}$ is a torification of $(\mathbb{R}^{n_{1}+n_{2}},\textup{Hess}(\psi),\nabla^{\textup{flat}})$ 
	with corresponding torus action $\Phi_{1}\times \Phi_{2}$. 

	(2) Suppose that $\tau_{i}$ is induced by the toric factorization $(L_{i},X_{i},F_{i})$, $i=1,2.$ This means that 
	$\tau_{i}=F_{i}\circ q_{L_{i}}$, where $q_{L_{i}}:T\mathbb{R}^{n_{i}}\to \mathbb{R}^{n_{i}}_{L_{i}}$ is the quotient map associated to the action of 
	$\Gamma(L_{i})$ on $T\mathbb{R}^{n_{i}}$. Let $\widetilde{X},L,G$ and $\widetilde{G}$ be defined as above. 
	It follows from the discussion above that $(L,\widetilde{X},F\circ \widetilde{G})$ is a toric factorization and that the diagram
	\begin{eqnarray*}
	\begin{tikzcd}
		T\mathbb{R}^{n_{1}+n_{2}}  \arrow[swap]{d}{q_{L}}\arrow{r}{\displaystyle G}  
		& T\mathbb{R}^{n_{1}}\times T\mathbb{R}^{n_{2}}  \arrow[swap]{d}{q_{L_{1}}\times q_{L_{2}}}\arrow{rd}{\tau_{1}\times \tau_{2}} 
		&
		\\
		\mathbb{R}_{L}^{n_{1}+n_{2}}\arrow[swap]{r}{\displaystyle \widetilde{G}}   
		& \mathbb{R}_{L_{1}}^{n_{1}}\times \mathbb{R}_{L_{2}}^{n_{2}} \arrow[swap]{r}{\displaystyle F} 
		& N_{1}^{\circ}\times N_{2}^{\circ}. 
	\end{tikzcd}
	\end{eqnarray*}
	is commutative. From this we see that $F\circ \widetilde{G}\circ q_{L}= (\tau_{1}\times \tau_{2})\circ G$ is a compatible covering map. 
	If $T\mathbb{R}^{n_{1}+n_{2}}$ and $T\mathbb{R}^{n_{1}}\times T\mathbb{R}^{n_{2}} $ are identified 
	via the map $G$, then $\tau_{1}\times \tau_{2}$ itself is a compatible covering map. 

	(3) This is immediate.
\end{proof}

\begin{definition}\label{nfknknefknknekn}
	Let $\mathcal{E}_{1}$ and $\mathcal{E}_{2}$ be exponential families defined over the finite sets $\Omega_{1}=\{x_{1},...,x_{r}\}$ and 
	$\Omega_{2}=\{y_{1},...,y_{s}\}$, respectively. The \textit{product} of $\mathcal{E}_{1}$ and $\mathcal{E}_{2}$, denoted by $\mathcal{E}_{1}\times \mathcal{E}_{2}$, 
	is the set of all maps $p:\Omega=\Omega_{1}\times \Omega_{2}\to \mathbb{R}$ of the form $p(x_{i},y_{j})=p_{1}(x_{i})p_{2}(y_{j})$, where $p_{1}\in \mathcal{E}_{1}$ and 
	$p_{2}\in \mathcal{E}_{2}$. 
\end{definition}

	One can readily check that $\mathcal{E}_{1}\times \mathcal{E}_{2}$ is an exponential family of dimension $\textup{dim}(\mathcal{E}_{1})+\textup{dim}(\mathcal{E}_{2})$ 
	defined over $\Omega=\Omega_{1}\times \Omega_{2}$. If $\psi_{1}:\mathbb{R}^{n_{1}}\to \mathbb{R}$ and $\psi_{2}:\mathbb{R}^{n_{2}}\to \mathbb{R}$ are the 
	cumulant generating functions of $\mathcal{E}_{1}$ and $\mathcal{E}_{2}$, respectively, then $\psi:\mathbb{R}^{n_{1}+n_{2}}=\mathbb{R}^{n_{1}}\times \mathbb{R}^{n_{2}}
	\to \mathbb{R},$ $(\theta,\theta')\mapsto \psi_{1}(\theta)+\psi_{2}(\theta')$ is the cumulant generating function of $\mathcal{E}_{1}\times \mathcal{E}_{2}$.
	It follows from this and Proposition \ref{nfeknknefknfknk} that if $\Phi_{i}:\mathbb{T}^{n_{i}}\times N_{i}\to N_{i}$ is a torification of $\mathcal{E}_{i}$, then 
	$N_{1}\times N_{2}$ is a torification of $\mathcal{E}_{1}\times \mathcal{E}_{2}$ with torus action 
	$\Phi_{1}\times \Phi_{2}.$ Moreover, if $\tau_{i}:T\mathcal{E}_{i}\to N_{i}^{\circ}$ is a compatible covering map, then $\tau_{1}\times \tau_{2}:T\mathcal{E}_{1}\times T\mathcal{E}_{2}\to 
	N_{1}^{\circ}\times N_{2}^{\circ}$, $(u,v)\mapsto (\tau_{1}(u),\tau_{2}(v))$ is a compatible covering map. 

\begin{example}
	$\mathbb{P}_{n}(1)\times \mathbb{P}_{m}(1)$ is the regular torification of $\mathcal{P}_{n+1}^{\times}\times \mathcal{P}_{m+1}^{\times}.$ 
\end{example}


\begin{proposition}\label{nekwnknkwdkejjkknk}
	The lift of the inclusion map $\mathcal{P}_{n+1}^{\times}\times \mathcal{P}_{m+1}^{\times}\to \mathcal{P}_{(n+1)(m+1)}^{\times}$ is the Segre embedding:
		\begin{center}
		\begin{tabular}{llllll}
			$\mathbb{P}_{n}(1)\times \mathbb{P}_{m}(1)$  & $\to$     & $\mathbb{P}_{(n+1)(m+1)-1}(1)$\\[0.9em]
			 $([z_{i}],[w_{j}])$   & $\mapsto$ & $[z_{i}w_{j}],$
		\end{tabular}
		\end{center}
	where $i=0,...,n,j=0,...,m$ and lexicographic ordering is adopted.
\end{proposition}
\begin{proof}[Sketch of proof]
	The proof is entirely analogous to the proof for Proposition \ref{nfknknkefnsknkfnk}. For simplicity, we assume $n=m=1$. In this case, the Segre embedding 
	$f:\mathbb{P}_{1}(1)\times \mathbb{P}_{1}(1)\to \mathbb{P}_{3}(1)$ reads $f([z_{1},z_{2}],[w_{1},w_{2}])=
	[z_{1}w_{1},z_{1}w_{2},z_{2}w_{1},z_{2}w_{2}]$. Let $\theta$ be the natural parameter on $\mathcal{P}_{2}^{\times}$. Elements of $\mathcal{P}_{2}^{\times}$ 
	are parametrized as follows: $p(x_{i};\theta)=e^{\delta_{1i}\theta-\psi(\theta)}$, where $x_{i}\in \Omega=\{x_{1},x_{2}\}$ and 
	$\psi(\theta)=\ln(1+e^{\theta})$ (see Example \ref{exa:5.5}). By definition, if $p\in \mathcal{P}_{2}^{\times}\times \mathcal{P}_{2}^{\times}$, then there are real numbers 
	$\theta_{1}$ and $\theta_{2}$ such that $p(x_{i},x_{j})=p(x_{i};\theta_{1})p(x_{j};\theta_{2})$ for all $x_{i},x_{j}\in \Omega$, and hence 
	\begin{eqnarray*}
		p(x_{i},x_{j})=e^{\delta_{1i}\theta_{1}+\delta_{1j}\theta_{2}-\psi(\theta_{1})-\psi(\theta_{2})}=e^{H_{1}(x_{i},x_{j})\theta_{1}+H_{2}(x_{i},x_{j})\theta_{2}-
		\phi(\theta_{1},\theta_{2})}, 
	\end{eqnarray*}
	where $H_{1}(x_{i},x_{j})=\delta_{1i}$, $H_{2}(x_{i},x_{j})=\delta_{1j}$ 
	and $\phi(\theta_{1},\theta_{2})=\psi(\theta_{1})+\psi(\theta_{2})$. This shows in particular that $\mathcal{P}_{2}^{\times}\times 
	\mathcal{P}_{2}^{\times}$ is an exponential family with natural parameters $(\theta_{1},\theta_{2})$ and cumulant generating function $\phi$.

	In order to find the local expression for the inclusion $j:\mathcal{P}_{2}^{\times}\times \mathcal{P}_{2}^{\times}\to \mathcal{P}_{4}^{\times}$, we need 
	to give $\Omega\times \Omega$ an ordering. Let $y_{1}=(x_{1},x_{1}),$ $y_{2}=(x_{1},x_{2})$, $y_{3}=(x_{2},x_{1})$ and $y_{4}=(x_{2},x_{2})$. Then $\Omega\times \Omega=
	\{y_{1},y_{2},y_{3},y_{4}\}$. Let $(\theta_{1}',\theta_{2}',\theta_{3}')$ be the natural parameters on $\mathcal{P}_{4}^{\times}$. 
	Taking into account the proof of Lemma \ref{nefknknknk}, we see that the coordinate expression for the inclusion map $j$ in the natural parameters is given by 
	\begin{eqnarray*}
	\begin{bmatrix}
		\theta_{1}'\\
		\theta_{2}'\\
		\theta_{3}'
	\end{bmatrix}
	=
	\begin{bmatrix}
		H_{1}(y_{1})-H_{1}(y_{4})   &   H_{2}(y_{1})-H_{2}(y_{4})  \\
		H_{1}(y_{2})-H_{1}(y_{4})   &   H_{2}(y_{2})-H_{2}(y_{4})  \\
		H_{1}(y_{3})-H_{1}(y_{4})   &   H_{2}(y_{3})-H_{2}(y_{4}) 
	\end{bmatrix}
	\begin{bmatrix}
	\theta_{1}\\
	\theta_{2} 
	\end{bmatrix}.
	\end{eqnarray*}
	Since $H_{1}(y_{1})=H_{1}(y_{2})=H_{2}(y_{1})=H_{2}(y_{3})=1$ and all the other values of $H_{1}$ and $H_{2}$ are zero, we find
	\begin{eqnarray*}
	\begin{bmatrix}
		\theta_{1}'\\
		\theta_{2}'\\
		\theta_{3}'
	\end{bmatrix}
	=
	\begin{bmatrix}
		1  & 1\\
		1  & 0\\
		0  & 1 
	\end{bmatrix}
	\begin{bmatrix}
	\theta_{1}\\
	\theta_{2} 
	\end{bmatrix}.
	\end{eqnarray*}
	Thus $j(\theta_{1},\theta_{2})=(\theta_{1}+\theta_{2},\theta_{1},\theta_{2}).$ 
	Under the usual identification $T\mathcal{P}_{n+1}^{\times} =\mathbb{C}^{n}$, the derivative $j_{*}:\mathbb{C}\times
	\mathbb{C}\to \mathbb{C}^{3}$ is given by $j_{*}(z_{1},z_{2})=(z_{1}+z_{2}, z_{1},z_{2})$. 
	Let $\tau':\mathbb{C}^{3}\to \mathbb{P}_{3}(1)$, $(z_{1},z_{2},z_{3})\mapsto [e^{z_{1}/2},e^{z_{2}/2},e^{z_{3}/3},1]$. 
	Recall that $\tau'$ is a compatible covering map (see Proposition \ref{njewndknknknk}). We compute:
	\begin{eqnarray}\label{nwndnknkdnkn}
		(\tau'\circ j_{*})(z_{1},z_{2})=\tau(z_{1}+z_{2},z_{1},z_{2})=[e^{(z_{1}+z_{2})/2},e^{z_{1}/2},e^{z_{2}/2},1].
	\end{eqnarray}
	Let $\tau:\mathbb{C}\to \mathbb{P}_{1}(1)$, $z\mapsto [e^{z/2},1]$. 
	By Proposition \ref{nfeknknefknfknk}, $\tau\times \tau$ is a compatible covering map. We compute:
	\begin{eqnarray}\label{nwknknkfnknkn}
		(f\circ (\tau\times \tau))(z_{1},z_{2})=f([e^{z_{1}/2},1],[e^{z_{2}/2},1])=[e^{(z_{1}+z_{2})/2},e^{z_{1}/2},e^{z_{2}/2},1]. 
	\end{eqnarray}
	Comparing \eqref{nwndnknkdnkn} and \eqref{nwknknkfnknkn}, we see that $\tau'\circ j_{*}=f\circ (\tau\times \tau)$. This shows that the Segre embedding 
	$f$ is a lift of $j$.  
\end{proof}

\section{Duality}\label{nnknkenkwnkdnk}

	In this section, we summarize some of the results of this paper in the form of a duality (bijection) similar to Delzant correspondence in symplectic geometry \cite{Delzant}.

	As the literature is not uniform, we give the following definition. 

\begin{definition}\label{nfwknkenkwnknk}
	A \textit{K\"{a}hler toric manifold} is a connected K\"{a}hler manifold $N$ 
	of complex dimension $n$ equipped with an effective isometric and holomorphic action $\Phi:\mathbb{T}^{n}\times N\to N$ of the $n${-}dimension real torus 
	$\mathbb{T}^{n}=\mathbb{R}^{n}/\mathbb{Z}^{n}$ such that for every $\xi\in \textup{Lie}(\mathbb{T}^{n})$, the vector vector field $J\xi_{N}$ is complete, 
	where $J$ is the complex structure on $N$ and $\xi_{N}$ is the fundamental vector field on $N$ associated to $\xi.$
\end{definition}

	Note that the definition of a torification $N$ does not require the vector fields $J\xi_{N}$ to be complete (see Section \ref{nfeknkwdnkk}). Therefore torifications are not 
	necessarily K\"{a}hler toric manifolds in the sense of Definition \ref{nfwknkenkwnknk}. Note also that if a K\"{a}hler toric manifold is regular, 
	then it is simply connected and hence the torus action is Hamiltonian (this follows, for example, from \cite{Ortega}, Propositions 4.5.17 and 4.5.19).

%

	Recall that two K\"{a}hler toric manifolds $\Phi:\mathbb{T}^{n}\times N\to N$ and $\Phi':\mathbb{T}^{n}\times N'\to N'$ are \textit{equivalent} 
	if there exist a K\"{a}hler isomorphism $G:N\to N'$ and a Lie group isomorphism $\rho:\mathbb{T}^{n}\to \mathbb{T}^{n}$ such that 
	$G\circ \Phi_{a}=\Phi_{\rho(a)}\circ G$ for all $a\in \mathbb{T}^{n}$. In this case, we write $N\sim N'.$

	Recall that two dually flat manifolds $(M,h,\nabla)$ and $(M',h',\nabla')$ are \textit{equivalent} if there is an isomorphism of dually flat spaces between 
	them. In this case, we also write $M\sim M'$.

	Equivalence classes are denoted by $[M,h,\nabla]$ and $[\Phi:\mathbb{T}^{n}\times N\to N]$, or simply $[M]$ and $[N]$. 

\begin{theorem}\label{neknkneknkn} 
	Let $A$ be the set of toric dually flat manifolds $(M,h,\nabla)$ of dimension $n$ admitting a global pair of dual coordinate 
	systems and whose regular torifications are K\"{a}hler toric manifolds. Let $B$ be the set of regular K\"{a}hler toric manifolds $N$ of complex dimension $n$. 
	The map 
	\begin{center}
		\begin{tabular}{lllll}
			$A/{\sim}$  & $\to$      & $B/{\sim},$\\[0.5em]
			$[M,h,\nabla]$     & $\mapsto$  & $[\textup{regular torification of}\,\,M]$	
		\end{tabular}
	\end{center}
	is a bijection. 
\end{theorem}
\begin{proof}
	Let $F$ be the map in the theorem. The existence of lifts guarantees that $F$ is well defined. Injectivity of $F$ is a direct consequence of Proposition \ref{nenwknkfwnknk}. 
	Surjectivity of $F$ follows from Theorem \ref{nckdnwknknednenkenk}.  
\end{proof}

\appendix\label{nfeknknkefnknknk}
\section{Legendre transform}
	Throughout this section, $(x_{1},...,x_{n})$ are standard coordinates on $\mathbb{R}^{n}$ and $\langle\,,\,\rangle$ is the ordinary 
	inner product in $\mathbb{R}^{n}$. Given a differentiable function $h:U\to \mathbb{R}$ defined on an open set $U\subseteq \mathbb{R}^{n}$, 
	we will denote by $\textup{grad}(h)$ the corresponding gradient map. Thus $\textup{grad}(h)(x)=
	\big(\tfrac{\partial h}{\partial x_{1}}(x),...,\tfrac{\partial h}{\partial x_{n}}(x)\big)$, $x\in U$. 

	The material presented in this section is taken from \cite{Rockafellar2} (see also \cite{Rockafellar}). 
\begin{definition}\label{jefkwjkjkejk}
	Let $h$ be a differentiable real-valued function defined on a non-empty open set $U$ in $\mathbb{R}^{n}$. Let $U^{*}$ 
	be the image of $U$ under the gradient map $\textup{grad}(h)$. If $\textup{grad}(h)$ is injective, 
	then the function 
	\begin{eqnarray*}
		h^{*}(x^{*})=\big\langle x^{*}, (\textup{grad}(h))^{-1}(x^{*})\big\rangle -h\big((\textup{grad}(h))^{-1}(x^{*})\big)
	\end{eqnarray*}
	is well-defined on $U^{*}$. The pair $(U^{*},h^{*})$ is called the \textit{Legendre transform} of $(U,h)$. 
\end{definition}

\begin{definition}
	We shall say that a pair $(U,h)$ is a \textit{convex function of Legendre type} on $\mathbb{R}^{n}$ if the following conditions hold:
	\begin{enumerate}[(1)]
		\item $U$ is a non-empty open convex set in $\mathbb{R}^{n}$.
		\item $h:U\to \mathbb{R}$ is a strictly convex differentiable function on $U$. 
		\item $\lim_{\lambda\to 0^{+}}\tfrac{d}{d\lambda}h(\lambda a +(1-\lambda)x)=-\infty$ whenever $a\in U$ and $x$ is a boundary point of $U$.
	\end{enumerate}
\end{definition}

	Note that the third condition is automatically satisfied when $U=\mathbb{R}^{n}$ (since there is no boundary point in this case). 

\begin{theorem}[\cite{Rockafellar2}]\label{neknknfkfneknk}
	Let $(U,h)$ be a convex function of Legendre type on $\mathbb{R}^{n}$. The Legendre transform $(U^{*},h^{*})$ is then well-defined. It is 
	another convex function of Legendre type on $\mathbb{R}^{n}$, and $\textup{grad}(h^{*})=\big(\textup{grad}(h)\big)^{-1}$ on $U^{*}$. 
	The Legendre transform of $(U^{*},h^{*})$ is $(U,h)$ again.
\end{theorem}

\begin{footnotesize}\bibliography{bibtex}\end{footnotesize}
\end{document}